\documentclass[a4paper,twoside,11pt]{amsart}

\usepackage[a4paper]{geometry}
\usepackage[latin1]{inputenc}
\usepackage[T1]{fontenc}

\usepackage{lipsum}
\usepackage{comment}
\usepackage{color}

\usepackage[francais,english]{babel}

\usepackage{amssymb,amsmath,amsthm,amscd,amsbsy}
\usepackage{mathrsfs}
\usepackage{stmaryrd}
\usepackage{esint}

\usepackage{subfig}

\usepackage{tabularx}
\usepackage{calc}
\usepackage[pdftex]{graphicx}
\usepackage[all]{xy}
\xyoption{v2}
\xyoption{2cell}
\UseAllTwocells

\usepackage[pdfauthor={Christophe Prange},pdftitle={Asymptotic analysis of boundary layer correctors in periodic homogenization},pdftex]{hyperref}

\numberwithin{equation}{section}

\DeclareMathOperator{\Idd}{I}

\DeclareMathOperator{\logg}{ln}
\newcommand{\argmin}{\operatornamewithlimits{argmin}}

\let\oldmarginpar\marginpar
\renewcommand\marginpar[1]{\oldmarginpar{\color{red}\raggedleft\tiny #1}} 

\title{Asymptotic analysis of boundary layer correctors in periodic homogenization}

\begin{document}

\selectlanguage{english}

\newtheorem{theo}{Theorem}[section]
\newtheorem{prop}[theo]{Proposition}
\newtheorem{lem}[theo]{Lemma}
\newtheorem{cor}[theo]{Corollary}
\newtheorem*{theo*}{Theorem}

\theoremstyle{definition}
\newtheorem{defi}[theo]{Definition}

\theoremstyle{remark}
\newtheorem{rem}[theo]{Remark}
\newtheorem*{rems*}{Remarks}

\title{Asymptotic analysis of boundary layer correctors in periodic homogenization}
\author{Christophe Prange$^*$}
\thanks{$^*$ Institut Math\'ematique de Jussieu, $175$ rue du Chevaleret, $75013$ Paris, France\\
\emph{E-mail address:} \texttt{prange@math.jussieu.fr}}

\begin{abstract}
This paper is devoted to the asymptotic analysis of boundary layers in periodic homogenization. We investigate the behaviour of the boundary layer corrector, defined in the half-space $\Omega_{n,a}:=\{y\cdot n-a>0\}$, far away from the boundary and prove the convergence towards a constant vector field, the boundary layer tail. This problem happens to depend strongly on the way the boundary $\partial\Omega_{n,a}$ intersects the underlying microstructure. Our study complements the previous results obtained on the one hand for $n\in\mathbb R\mathbb Q^d$, and on the other hand for $n\notin\mathbb R\mathbb Q^d$ satisfying a small divisors assumption. We tackle the case of arbitrary $n\notin\mathbb R\mathbb Q^d$ using ergodicity of the boundary layer along $\partial\Omega_{n,a}$. Moreover, we get an asymptotic expansion of Poisson's kernel $P=P(y,\tilde{y})$, associated to the elliptic operator $-\nabla\cdot A(y)\nabla\cdot$ and $\Omega_{n,a}$, for $|y-\tilde{y}|\rightarrow\infty$. Finally, we show that, in general, convergence towards the boundary layer tail can be arbitrarily slow, which makes the general case very different from the rational or the small divisors one.
\end{abstract}

\maketitle

\pagestyle{plain}

\selectlanguage{english}

\pagestyle{plain}

\section{Introduction}
In this paper we investigate the behaviour of $v_{bl}=v_{bl}(y)\in\mathbb R^N$, solving the elliptic system
\begin{equation}\label{sysbl}
\left\{\begin{array}{rll}
-\nabla\cdot A(y)\nabla v_{bl}&=0,&y\cdot n-a>0\\
v_{bl}&=v_0(y),& y\cdot n-a=0
\end{array}\right. 
\end{equation}
with periodically oscillating coefficients and Dirichlet data, far away from the boundary of the half-space $\{y\cdot n-a>0\}\subset\mathbb R^d$. Understanding these asymptotics is an important issue in the study of boundary layer correctors in periodic homogenization.

Throughout this paper, $N\geq 1$, $d\geq 2$, $n\in\mathbb S^{d-1}$ is a given unit vector and $a\in\mathbb R$. The Dirichlet data $v_0=v_0(y)\in\mathbb R^N$ is defined for $y\in\mathbb R^d$; so is the family of matrices $A=A^{\alpha\beta}(y)\in M_N\left(\mathbb R\right)$ indexed by $1\leq\alpha,\beta\leq d$. Small greek letters like $\alpha,\ \beta,\ \gamma,\ \eta$ usually denote integers belonging to $\{1,\ldots\ d\}$, whereas $i,\ j,\ k$ stand for integers in $\{1,\ldots\ N\}$. Therefore, taking advantage of Einstein's convention:
\begin{equation*}
\Bigl[\nabla\cdot A(y)\nabla v_{bl}\Bigl]_i=\partial_{y_\alpha}\Bigl(A^{\alpha\beta}_{ij}(y)\partial_{y_\beta}v_{bl,j}\Bigr).
\end{equation*}
The main assumptions on $A$ are now
\begin{description}
\item[(A1) ellipticity] there exists $\lambda>0$ such that for every family $\xi=\xi^\alpha\in\mathbb R^N$ indexed by $1\leq\alpha\leq d$, for all $y\in\mathbb R^d$,
\begin{equation*}
\lambda\xi^\alpha\cdot\xi^\alpha\leq A^{\alpha\beta}(y)\xi^\alpha\cdot\xi^\beta\leq \lambda^{-1}\xi^\alpha\cdot\xi^\alpha;
\end{equation*}
\item[(A2) periodicity] $A$ is $1$-periodic i.e. for all $y\in\mathbb R^d$, for all $\xi\in \mathbb Z^d$,
\begin{equation*}
A(y+\xi)=A(y);
\end{equation*}
\item[(A3) regularity] $A$ is supposed to belong to $C^\infty\left(\mathbb R^d\right)$.
\end{description}
Moreover, we assume
\begin{description}
\item[(B1) periodicity] $v_0$ is a $1$-periodic function;
\item[(B2) regularity] $v_0$ is smooth.
\end{description}

Before focusing on system \eqref{sysbl} itself, let us recall the main steps of the homogenization procedure leading to the study of $v_{bl}$. We carry out our analysis on the elliptic system with Dirichlet boundary conditions
\begin{equation}\label{sysuepsbound}
\left\{
\begin{array}{rll}
-\nabla\cdot A\left(\frac{x}{\varepsilon}\right)\nabla u^\varepsilon&=f,&x\in\Omega\\
u^\varepsilon&=0,& x\in\partial\Omega
\end{array}
\right. 
\end{equation}
posed in a bounded Lipschitz domain $\Omega\subset\mathbb R^d$ endowed with a periodic microstructure. This system may model, for instance, heat conduction ($N=1$), or linear elasticity ($d=2$ or $3$ and $N=d$) in a periodic composite medium. We are interested in the asymptotical behaviour of $u^\varepsilon$, when $\varepsilon\rightarrow 0$. This problem is of particular importance when it comes to the numerical analysis of \eqref{sysuepsbound}. The oscillations at scale $\varepsilon$ are too fast to be captured by classical methods. One therefore looks for an approximation of $u^\varepsilon$, taking into account the small scales, without solving them explicitely. To gain an insight into these numerical issues, we refer to \cite{versa1}; for a general overview of the homogenization theory see \cite{blp} or \cite{cio}.

\subsection{Some error estimates in homogenization}

For a given source term $f\in L^2(\Omega)$, boundedness of $\Omega$ and ellipticity of $A$ yield the existence and uniqueness of a weak solution $u^\varepsilon$ in $H^1_0(\Omega)$ to \eqref{sysuepsbound}. Moreover, we deduce from the a priori bound
\begin{equation*}
\left\|u^\varepsilon\right\|_{H^1_0(\Omega)}\leq C\left\|f\right\|_{L^2(\Omega)},
\end{equation*}
where $C>0$ is a constant independent from $\varepsilon$, that up to the extraction of a subsequence, $u^\varepsilon$ converges weakly in $H^1_0(\Omega)$. A classical method to investigate this convergence is to have recourse to multiscale asymptotic expansions. We expand, at least formally, $u^\varepsilon$ in powers of $\varepsilon$
\begin{equation}\label{multiscale}
u^\varepsilon\approx u^0\left(x,\frac{x}{\varepsilon}\right)+\varepsilon u^1\left(x,\frac{x}{\varepsilon}\right)+\varepsilon^2u^2\left(x,\frac{x}{\varepsilon}\right)+\ldots
\end{equation}
assuming that $u^i=u^i(x,y)\in\mathbb R^N$ is $1$-periodic with respect to the fast variable $y$. Plugging \eqref{multiscale} into \eqref{sysuepsbound},  identifying the powers of $\varepsilon$ and solving the cascade of equations then gives:
\begin{enumerate} 
\item that $u^0$ does not depend on $y$ and that it solves the homogenized system
\begin{equation*}
\left\{
\begin{array}{rll}
-\nabla\cdot A^0\nabla u^0&=f,& x\in\Omega\\
u^0&=0,& x\in\partial\Omega
\end{array}
\right.
\end{equation*}
where the constant homogenized tensor $A^0=A^{0,\alpha\beta}\in M_N\left(\mathbb R\right)$ is given by
\begin{equation*}
A^{0,\alpha\beta}:=\int_{\mathbb T^d}A^{\alpha\beta}(y)dy+\int_{\mathbb T^d}A^{\alpha\gamma}(y)\partial_{y_\gamma}\chi^\beta(y)dy,
\end{equation*}
and $\chi=\chi^\beta(y)\in M_N\left(\mathbb R\right)$ is the family, indexed by $\beta\in\{1,\ldots\ d\}$, of solutions to the cell problem
\begin{equation}\label{syschi}
\left\{\begin{array}{rll}
-\nabla_y\cdot A(y)\nabla_y\chi^\beta&=\partial_{y_\alpha}A^{\alpha\beta}&,\ y\in\mathbb T^d\\
\int_{\mathbb T^d}\chi^\beta(y)dy&=0&
\end{array}\right.;
\end{equation}
\item that for all $x\in\Omega$ and $y\in\mathbb T^d$, $u^1(x,y)=\chi^\alpha(y)\partial_{x_\alpha}u^0(x)+\bar{u}^1(x)$;
\item and that for all $x\in\Omega$ and $y\in\mathbb T^d$, $u^2(x,y):=\Gamma^{\alpha\beta}(y)\partial_{x_\alpha}\partial_{x_\beta}u^0(x)+\chi^\alpha(y)\partial_{x_\alpha}\bar{u}^1(x)+\bar{u}^2(x)$, where $\Gamma=\Gamma^{\alpha\beta}(y)\in M_N\left(\mathbb R\right)$ is the family, indexed by $\alpha,\ \beta\in\{1,\ldots\ d\}$, solving
\begin{equation*}
\left\{\begin{array}{rll}
-\nabla_y \cdot A(y)\nabla_y \Gamma^{\alpha\beta}&=B^{\alpha\beta}-\int_{\mathbb T^d}B^{\alpha\beta}(y)dy,& y\in \mathbb T^d\\
\int_{\mathbb T^d}\Gamma^{\alpha\beta}(y)dy&=0&
\end{array}\right. ,
\end{equation*}
and
\begin{equation*}
B^{\alpha\beta}(y):=A^{\alpha\beta}(y)+A^{\alpha\gamma}(y)\partial_{y_\gamma}\chi^\beta(y)+\partial_{y_\gamma}\bigl(A^{\gamma\alpha}(y)\chi^\beta(y)\bigr).
\end{equation*}
\end{enumerate}

One can then carry out energy estimates on the error $r^{2,\varepsilon}_{bl}:=u^\varepsilon(x)-u^0(x)-\varepsilon\chi\left(\frac{x}{\varepsilon}\right)\nabla u^0(x)-\varepsilon^2\Gamma\left(\frac{x}{\varepsilon}\right)\cdot\nabla^2u^0(x)$ solving the elliptic system
\begin{equation*}
\left\{
\begin{array}{rll}
-\nabla\cdot A\left(\frac{x}{\varepsilon}\right)\nabla r^{2,\varepsilon}&=f^\varepsilon,&x\in\Omega\\
r^{2,\varepsilon}&=g^\varepsilon,& x\in\partial\Omega
\end{array}
\right. .
\end{equation*}
To bound $r^{2,\varepsilon}$ one needs some regularity on $u^0$, say $u^0\in H^4(\Omega)$ (for refined estimates, involving lower regularity on $u^0$, see \cite{moscovog,homovp}). Under this coarse assumption
\begin{equation*}
\left\|f^\varepsilon\right\|_{L^2(\Omega)}\leq C\varepsilon\left\|u^0\right\|_{H^4(\Omega)}
\end{equation*}
and 
\begin{equation*}
\left\|g^\varepsilon\right\|_{H^\frac{1}{2}(\Omega)}=\varepsilon\left\|\chi\left(\frac{x}{\varepsilon}\right)\nabla u^0-\varepsilon\Gamma\left(\frac{x}{\varepsilon}\right)\cdot\nabla^2u^0\right\|_{H^\frac{1}{2}(\Omega)}\leq C\varepsilon^\frac{1}{2}\left\|u^0\right\|_{H^4(\Omega)}.
\end{equation*}
One can therefore show that $\left\|r^{2,\varepsilon}\right\|_{H^1(\Omega)}\leq C\varepsilon^\frac{1}{2}\left\|u^0\right\|_{H^4(\Omega)}$, which implies
\begin{equation}\label{coarseestr2eps}
\left\|u^\varepsilon-u^0-\varepsilon\chi\left(\frac{x}{\varepsilon}\right)\cdot\nabla u^0\right\|_{H^1(\Omega)}\leq\left\|r^{2,\varepsilon}\right\|_{H^1(\Omega)}+\varepsilon^2\left\|\Gamma\left(\frac{x}{\varepsilon}\right)\cdot\nabla^2u^0\right\|_{H^1(\Omega)}\leq C\varepsilon^\frac{1}{2}\left\|u^0\right\|_{H^4(\Omega)}
\end{equation}
and 
\begin{equation*}
\left\|u^\varepsilon-u^0\right\|_{L^2(\Omega)}=O\bigl(\varepsilon^\frac{1}{2}\bigr).
\end{equation*}
The latter estimate shows, that the zeroth-order term $u^0$ is a correct approximation of $u^\varepsilon$. Estimate \eqref{coarseestr2eps} is however limited by the trace on $\partial\Omega$ of $u^1\left(\cdot,\frac{\cdot}{\varepsilon}\right)$. Actually, the periodicity assumption of the ansatz \eqref{multiscale} with respect to the microscopic variable $y$, is not compatible with the homogeneous Dirichlet condition on the boundary. In order to force the Dirichlet condition one introduces a boundary layer term $u^{1,\varepsilon}_{bl}:=u^{1,\varepsilon}_{bl}(x)\in\mathbb R^N$ at order $\varepsilon^1$ in the expansion \eqref{multiscale}. More precisely, $u^{1,\varepsilon}_{bl}$ solves
\begin{equation}\label{sysubl}
\left\{\begin{array}{rll}
-\nabla\cdot A\left(\frac{x}{\varepsilon}\right)\nabla u^{1,\varepsilon}_{bl}&=0,&x\in\Omega\\
u^{1,\varepsilon}_{bl}&=-u^1\left(x,\frac{x}{\varepsilon}\right),& x\in\partial\Omega
\end{array}\right. 
\end{equation}
and the corrected Ansatz is
\begin{equation}\label{multiscalebl}
u^\varepsilon\approx u^0\left(x\right)+\varepsilon\left[u^1\left(x,\frac{x}{\varepsilon}\right)+u^{1,\varepsilon}_{bl}\right]+\varepsilon^2u^2\left(x,\frac{x}{\varepsilon}\right)+\ldots
\end{equation}
Adding the boundary layer at first-order improves \eqref{coarseestr2eps}: if $u^0\in H^4(\Omega)$,
\begin{equation}\label{improvedbound}
\left\|u^\varepsilon-u^0-\varepsilon\chi\left(\frac{x}{\varepsilon}\right)\cdot\nabla u^0-\varepsilon u^{1,\varepsilon}_{bl}\right\|_{H^1(\Omega)}\leq C\varepsilon\left\|u^0\right\|_{H^4(\Omega)}.
\end{equation}
However, such a trick remains useless as long as one is not able to describe the asymptotics of $u^{1,\varepsilon}_{bl}$ when $\varepsilon\rightarrow\infty$.

\subsection{Homogenization of boundary layer systems}
\label{subsechomobl}

As for \eqref{sysuepsbound}, the problem is to show that \eqref{sysubl} can be in some sense homogenized. A few remarks are in order:
\begin{enumerate}
\item System \eqref{sysubl} exhibits oscillations in the coefficients as well as on the boundary.
\item The oscillations along $\partial\Omega$ are not periodic in general. This can be expressed by the fact that the boundary breaks the periodic microstructure.
\item  The a priori bound on $u^{1,\varepsilon}_{bl}$
\begin{equation*}
\left\|u^{1,\varepsilon}_{bl}\right\|_{H^1(\Omega)}\leq C\left\|u^i\left(x,\frac{x}{\varepsilon}\right)\right\|=O\bigl(\varepsilon^{-\frac{1}{2}}\bigr)
\end{equation*}
does not provide a uniform bound in $H^1(\Omega)$.
\end{enumerate}
These issues make the homogenization of the boundary layer system \eqref{sysubl} far more difficult than the homogenization of \eqref{sysuepsbound}. The results obtained in this direction are still partial. The first step towards the asymptotic behaviour of $u^{1,\varepsilon}_{bl}$ is to get a priori bounds uniform in $\varepsilon$. If $N=1$, the maximum principle furnishes a uniform bound in $L^\infty(\Omega)$ on $u^{1,\varepsilon}_{bl}$. A way of investigating \eqref{sysubl} in the case when $N>1$ is to represent $u^{1,\varepsilon}_{bl}$ in terms of the oscillating Poisson kernel associated to $-\nabla\cdot A\left(\frac{x}{\varepsilon}\right)\nabla\cdot$ and $\Omega$. In the series of papers \cite{alin,Alin90P,alinLp}, Avellaneda and Lin manage to get estimates, uniform in $\varepsilon$, on these Green and Poisson kernels, as well as expansions valid in the limit $\varepsilon\rightarrow 0$ (for recent progress in this direction, see \cite{kls12}). One of the results of \cite{alin} (see theorem $3$) is the uniform bound $\bigl\|u^{1,\varepsilon}_{bl}\bigr\|_{L^p(\Omega)}\leq C$, for $1<p\leq\infty$, valid under the assumption that $\partial\Omega$ is at least $C^{1,\alpha}$, with $0<\alpha\leq 1$. Note that an $L^2(\Omega)$ bound on $u^{1,\varepsilon}_{bl}$ yields an $H^1(\omega)$ bound on the gradient, for $\omega\Subset\Omega$ compactly supported in $\Omega$ (see \cite{allam,dgvnm}). The strong oscillations of $\nabla u^{1,\varepsilon}_{bl}$ are filtered out in the interior of the domain and concentrate near the boundary. Hence, the multiscale expansion is right up to the order $1$ in $\varepsilon$ in the interior:
\begin{equation}\label{estinterior}
\left\|u^\varepsilon-u^0-\varepsilon\chi\left(\frac{x}{\varepsilon}\right)\cdot\nabla u^0\right\|_{H^1(\omega)}\leq C\varepsilon\left\|u^0\right\|_{H^4(\Omega)}.
\end{equation}
However, to improve the asymptotics up to the boundary, one needs another approach. Namely, we seek after a $2$-scale approximation of the boundary layer corrector
\begin{equation*}
u^{1,\varepsilon}_{bl}\approx v\left(x,\frac{x}{\varepsilon}\right).
\end{equation*}
Take $x_0\in\partial\Omega$ a point at which there exists a tangent hyperplane directed by $n:=n(x_0)\in\mathbb S^{d-1}$ and assume that $\Omega$ is contained in $\{x\cdot n-x_0\cdot n\geq 0\}$. Then plugging formally $v$ into \eqref{sysubl} gives at order $\varepsilon^{-2}$
\begin{equation}\label{approxblx_0}
\left\{
\begin{array}{rll}
-\nabla_y\cdot A(y)\nabla_yv(x_0,y)&=0,&y\cdot n-\frac{n\cdot x_0}{\varepsilon}> 0\\
v(x_0,y)&=-\chi(y)\cdot\nabla u^0(x_0),&y\cdot n-\frac{n\cdot x_0}{\varepsilon}=0
\end{array}
\right. .
\end{equation}
The variable $x_0$ is nothing more than a parameter in this system. If $\Omega$ is convex, the boundary layer is approximated in the vicinity of each point $x_0$ of the boundary by a $v_{x_0}=v(x_0,\cdot)$ solving \eqref{approxblx_0}. This formal idea has been made rigorous for polygonal convex bounded domains $\Omega\subset\mathbb R^2$, for which only a finite number of correctors $v^k$ has to be considered, one for each edge $K^k$. By linearity, $v^k$ factors into 
\begin{equation*}
v^k(x,y)=v^{k,\alpha}_{bl}(y)\partial_{x_\alpha}u^0(x)
\end{equation*}
where for all $\alpha=1,\ldots\ d$, $v^{k,\alpha}_{bl}=v^{k,\alpha}_{bl}(y)\in\mathbb R^N$ solves
\begin{equation}\label{sysblalpha}
\left\{\begin{array}{rll}
-\nabla\cdot A(y)\nabla v^{k,\alpha}_{bl}&=0,&y\cdot n^k-a>0\\
v^{k,\alpha}_{bl}&=\chi^\alpha(y),& y\cdot n^k-a=0
\end{array}\right. 
\end{equation}
with $a:=\frac{x\cdot n}{\varepsilon}$. Dropping the exponents $k$ and $\alpha$, we end up with \eqref{sysbl}.

System \eqref{sysbl} is linear elliptic in divergence form. The main source of difficulties one encounters is the lack of boundedness of the domain $\Omega_{n,a}:=\{y\cdot n-a>0\}$. This complicates the existence theory, but even more the study of the asymptotical behaviour. Moreover, the analysis of \eqref{sysbl} depends, in a nontrivial manner, on the interaction between $\partial\Omega_{n,a}$ and the underlying lattice. So far, it has been carried out in two different contexts:
\begin{description}
\item[(RAT)] the rational case, i.e. $n\in\mathbb R\mathbb Q^d$;
\item[(DIV)] the small divisors case, when there exists $C,\ \tau>0$ such that for all $\xi\in\mathbb Z^d\setminus\{0\}$, for all $i=1,\ldots\ d-1$,
\begin{equation}\label{smalldivisors1}
\left|n_i\cdot\xi\right|\geq C\left|\xi\right|^{-d-\tau}
\end{equation}
where $(n_1,\ldots\ n_{d-1},n)$ forms an orthogonal basis of $\mathbb R^d$.
\end{description}
Assumption \eqref{smalldivisors1} means that the distance from every point, except $0$, of the lattice $\mathbb Z^d$, to the line $\{\lambda n,\ \lambda\in\mathbb R\}$, is in some sense bounded from below. Note that this condition, albeit generic, in the sense that it is satisfied for almost every $n\in\mathbb S^{d-1}$ (for more quantitative results, see \cite{dgvnm2}), is not fulfilled by every vector $n\notin\mathbb R\mathbb Q^d$. 

Let us explain why the description of the asymptotics of $v_{bl}$ far away from the boundary is a crucial step in the homogenization of \eqref{sysubl}. Roughly speaking, one proves in both contexts {\bf (RAT)} and {\bf (DIV)}, that there exits a smooth $v_{bl}$ solving \eqref{sysbl} and that this solution converges very fast, when $y\cdot n\rightarrow\infty$, towards a constant vector field $v^\infty_{bl}\in\mathbb R^N$, the boundary layer tail. For a polygonal domain $\Omega$ with edges satisfying for instance the small divisors assumption, we approximate $u^{1,\varepsilon}_{bl}$ by $\bar{u}^1$ solution of
\begin{equation}\label{homsyssysubl}
\left\{
\begin{array}{rll}
-\nabla\cdot A^0\nabla \bar{u}^1&=0,&x\in\Omega\\
\bar{u}^1&=-v^{k,\infty}\cdot\nabla u^0,&x\in\partial\Omega\cap K^k
\end{array}
\right. .
\end{equation}
Indeed, 
\begin{equation*}
\left\|u^{1,\varepsilon}_{bl}-\bar{u}^1\right\|_{L^2(\Omega)}\leq \Bigl\|u^{1,\varepsilon}_{bl}-\bar{u}^1-\sum_{k}\left[v^k_{bl}-v^{k,\infty}\right]\cdot\nabla u^0\Bigr\|_{L^2(\Omega)}+\sum_{k}\left\|\left[v^k_{bl}-v^{k,\infty}\right]\cdot\nabla u^0\right\|_{L^2(\Omega)}
\end{equation*}
and using the decay of the boundary layer correctors in the interior of $\Omega$, one proves: 
\begin{theo}[\cite{homovp}, theorem $3.3$]\label{theohomobl}
If $u^0\in H^{2+\omega}(\Omega)$, with $\omega>0$, then there exists $\kappa>0$ such that 
\begin{equation*}
\left\|u^{1,\varepsilon}_{bl}-\bar{u}^1\right\|_{L^2(\Omega)}=O\left(\varepsilon^\kappa\right).
\end{equation*}
\end{theo}
This shows that the oscillating Dirichlet data of \eqref{sysubl} can be homogenized. The limit system \eqref{homsyssysubl} involves the tails of the boundary layer correctors. The corrector $\bar{u}^1$ has been used to implement numerical shemes in \cite{versa1}. More recently, the obtention of an \ref{theohomobl}-like theorem for a smooth uniformly convex two-dimensional domain has been achieved in \cite{dgvnm2}. Again, it relies on the approximation of the domain by polygonals with edges satisfying the small divisors condition, emphasizing the key role of systems \eqref{sysbl}.

We devote section \ref{secdgvnm} of this manuscript to a review of the previous results obtained on \eqref{sysbl}, with an emphasis on the techniques used in the small divisors case. Let us give an insight into these theorems:
\begin{description}
\item[(RAT)] Among the rich litterature about this case, we refer to \cite{allam} (lemma $4.4$), \cite{moscovog} (appendix $6$) and the references therein. A precise statement is given in theorem \ref{theoexpdecayrat}. It consists of two parts: 
\begin{description}
\item[existence] there exists a variational solution $v_{bl}\in C^\infty\left(\overline{\Omega_{n,a}}\right)$ to \eqref{sysbl} (unique if appropriate decay of $\nabla v_{bl}$ is prescribed);
\item[convergence] there exists a constant vector, called boundary layer tail, $v_{bl}^{a,\infty}$ depending on $a$ such that $v_{bl}(y)-v_{bl}^{a,\infty}$ and its derivatives tend to $0$ exponentially fast when $y\cdot n\rightarrow\infty$.
\end{description}
\item[(DIV)] This case was treated in the recent paper \cite{dgvnm} by G\'erard-Varet and Masmoudi. A precise statement is given in theorem \ref{theodecaydiv}. It consists again of two parts: 
\begin{description}
\item[existence] there exists a variational solution $v_{bl}\in C^\infty\left(\overline{\Omega_{n,a}}\right)$ to \eqref{sysbl} (unique if appropriate decay of $\nabla v_{bl}$ is prescribed);
\item[convergence] there exists a boundary layer tail $v_{bl}^{\infty}$ independent from $a$ such that $v_{bl}(y)-v_{bl}^{\infty}$ and its derivatives tend to $0$ when $y\cdot n\rightarrow\infty$, faster than any negative power of $y\cdot n$.
\end{description}
\end{description}
The main difference is that the boundary layer tail depends on $a$ in the rational setting and not in the small divisors one, which implies that the homogenization theorem \ref{theohomobl} is true up to the extraction of a subsequence $\varepsilon_n$ in the former. In both cases, one can come down to some periodic framework to prove the existence of a variational solution. Fast convergence follows from a St-Venant estimate. We come back to these points in detail in section \ref{secdgvnm}.

\subsection{Outline of our results and strategy}

Our goal is to analyse \eqref{sysbl} in the case when $n\notin\mathbb R\mathbb Q^d$ does not meet the small divisors assumption \eqref{smalldivisors1}. Again, one first wonders if the system has a solution. This question can be investigated by methods analogous to those of \cite{dgvnm}. Indeed, their well-posedness result does not rely on the small divisors assumption. The latter hypothesis is however essential to show the convergence towards the boundary layer tail in the work of G\'erard-Varet and Masmoudi. We therefore have recourse to another approach based on an integral representation of the variational solution to \eqref{sysbl} by the mean of Poisson's kernel $P=P(y,\tilde{y})\in M_N(\mathbb R)$ associated to $-\nabla\cdot A(y)\nabla\cdot$ and the domain $\Omega_{n,a}$. Basically,
\begin{equation*}
v_{bl}(y)=\int_{\partial\Omega_{n,a}}P(y,\tilde{y})v_0(\tilde{y})d\tilde{y},
\end{equation*}
for every $y\in\Omega_{n,a}$. At first glance, if $n=e_d$ the $d$-th vector of the canonical basis of $\mathbb R^d$, if $a=0$ and $y=\left(0,\varepsilon^{-1}\right)$, $\varepsilon>0$, then 
\begin{multline*}
v_{bl}\left(0,\frac{1}{\varepsilon}\right)=\int_{\mathbb R^{d-1}}P\left(\frac{1}{\varepsilon}(0,1),(\tilde{y}',0)\right)v_0(\tilde{y}',0)d\tilde{y}'\\=\int_{\mathbb R^{d-1}}\frac{1}{\varepsilon^{d-1}}P\left(\frac{1}{\varepsilon}(0,1),\frac{1}{\varepsilon}(\tilde{x}',0)\right)v_0\left(\frac{1}{\varepsilon}(\tilde{x}',0)\right)d\tilde{x}'.
\end{multline*}
Examining the asymptotics of $v_{bl}$ far away from the boundary, requires subsequently to understand the behaviour of the oscillating kernel $\frac{1}{\varepsilon^{d-1}}P\left(\frac{x}{\varepsilon},\frac{\tilde{x}}{\varepsilon}\right)$ when $\varepsilon\rightarrow 0$, or equivalently the asymptotical comportment of $P(y,\tilde{y})$, when $y\cdot n\rightarrow\infty$ and $\tilde{y}\in\partial\Omega_{n,0}$. This is done in section \ref{secexp}, relying on ideas and results of Avellaneda and Lin \cite{alin,alinLp}. We prove an expansion for $P$ associated to the domain $\Omega_{n,0}$ for arbitrary $n\in\mathbb S^{d-1}$. \emph{To put it in a nutshell, one demonstrates that there exists $\kappa>0$ such that for all $y\in\Omega_{n,0}$, for all $\tilde{y}\in\partial\Omega_{n,0}$,}
\begin{equation*}
\left|P(y,\tilde{y})-P_{exp}(y,\tilde{y})\right|\leq \frac{C}{|y-\tilde{y}|^{d-1+\kappa}},
\end{equation*}
\emph{where $P_{exp}=P_{exp}(y,\tilde{y})$ is an explicit kernel, with ergodicity properties tangentially to the boundary. The precise statement of this key result is postponed to section \ref{secexp}: see theorem \ref{theodvptP}}. We have an explicit expression for the corrector terms. Although this expansion is stated (and proved) for the domain $\Omega_{n,0}$ it extends to $\Omega_{n,a}$ by a simple translation. Studying the tail of $v_{bl}$ now boils down to examining the limit when $y\cdot n\rightarrow\infty$ of
\begin{equation*}
\int_{\partial\Omega_{n,a}}\left[P_{exp}(y,\tilde{y})\right]v_0(\tilde{y})d\tilde{y}+\int_{\partial\Omega_{n,a}}R(y,\tilde{y})v_0(\tilde{y})d\tilde{y}
\end{equation*}
with rest $R=R(y,\tilde{y})\in M_N(\mathbb R)$ satisfying $\left|R(y,\tilde{y})\right|\leq\frac{C}{|y-\tilde{y}|^{d-1+\kappa}}$. The rest integral tends to $0$. One takes advantage of the oscillations of $v_0$ on the boundary to show the convergence of the corrector integrals by the mean of an ergodic theorem. Doing so, we demonstrate the following theorem, which is the core of this paper:
\begin{theo}\label{theoCVBLtail}
Assume that $n\notin\mathbb R\mathbb Q^d$. Then, 
\begin{enumerate}
\item there exists a unique solution $v_{bl}\in C^\infty\left(\overline{\Omega_{n,a}}\right)\cap L^\infty\left(\Omega_{n,a}\right)$ of \eqref{sysbl} satisfying 
\begin{subequations}
\begin{align}
&\left\|\nabla v_{bl}\right\|_{L^\infty\left(\left\{y\cdot n-t>0\right\}\right)}\stackrel{t\rightarrow\infty}{\longrightarrow}0,\label{nablavbltends0}\\
&\int_a^\infty\left\|\partial_nv_{bl}\right\|^2_{L^\infty\left(\left\{y\cdot n-t=0\right\}\right)}dt<\infty\label{partial_z_dvbl},
\end{align}
\end{subequations}
\item and a boundary layer tail $v_{bl}^\infty\in\mathbb R^N$, independent from $a$, such that for all $y\in\Omega_{n,a}$
\begin{equation*}
v_{bl}(y)\stackrel{y\cdot n\rightarrow\infty}{\longrightarrow}v_{bl}^\infty,
\end{equation*}
locally uniformly in the tangential variable.
\end{enumerate}
Furthermore, one has an explicit expression for $v^\infty_{bl}$ (see \eqref{CVBLexprBLtail}).
\end{theo}
The use of the ergodic theorem to prove the convergence does not yield any rate. One wonders therefore how fast convergence of $v_{bl}$ towards $v^\infty_{bl}$ is. A partial answer is given by the next theorem, whose proof is addressed in section \ref{secslow}:  
\begin{theo}
\label{theoslow}
Assume that $d=2$, $N=1$, $A=\Idd_2$ and that $n\notin\mathbb R\mathbb Q^2$ does not satisfy \eqref{smalldivisors1}. Then for every $l>0$, for all $R>0$, there exists a smooth function $v_0$ and a sequence $(t_M)_{M}\in]a,\infty[^{\mathbb N}$ such that:
\begin{enumerate}
\item $\left(t_M\right)_M$ is strictly increasing and tends to $\infty$;
\item the unique solution $v_{bl}$ of \eqref{sysbl}, in the variational sense, converges towards $v^\infty_{bl}$ as $y\cdot n\rightarrow\infty$ and for all $M\in\mathbb N$, for all $y'\in\partial\Omega_{n,0}\cap B(0,R)$,
\end{enumerate}
\begin{equation*}
\left|v_{bl}(y'+t_Mn)-v^\infty_{bl}\right|\geq t_M^{-l}.
\end{equation*} 
\end{theo}
This result means that convergence can be as slow as we wish in some sense: for fixed $n\notin\mathbb Q\mathbb R^d$ which does not satisfy the small divisors assumption, for every power function $t\mapsto t^{-l}$ with $l>0$ there exists a $1$-periodic $v_0=v_0(y)\in\mathbb R$, such that the solution $v_{bl}$ of 
\begin{equation*}
\left\{\begin{array}{rll}
-\Delta v_{bl}&=0,&y\cdot n-a>0\\
v_{bl}&=v_0(y),& y\cdot n-a=0
\end{array}\right. 
\end{equation*}
converges slower to its tail than $\left(y\cdot n\right)^{-l}$ when $y\cdot n\rightarrow\infty$. The main obstruction preventing $v_{bl}$ from converging faster in general lies indeed in the fact that the distance between a given point $\xi\in\mathbb Z^d\setminus\{0\}$ and the line $\{\lambda n,\ \lambda\in\mathbb R\}$ is not bounded from below. This is in big contrast with the small divisors case, where \eqref{smalldivisors1} asserts the existence of a lower bound. It underlines the strong dependence of the boundary layer on the interaction between $\partial\Omega_{n,a}$ and $\mathbb Z^d$.

\subsection{Organization of the paper}

In this paper we address the proofs of theorems \ref{theodvptP} (expansion of Poisson's kernel), \ref{theoCVBLtail} (convergence of the boundary layer corrector) and \ref{theoslow} (slow convergence). Section \ref{secdgvnm} is devoted to a review of the rational and small divisors settings. We insist on the existence proof in the non rational case and underline the role of the small divisors assumption in the asymptotic analysis. In section \ref{secprel}, some essential properties and estimates on Green and Poisson kernels associated to elliptic operators with periodic coefficients are recalled. We prove a uniqueness theorem for \eqref{sysbl}, which makes it possible to rely on Poisson's formula to represent the variational solution of \eqref{sysbl} and explain to what extent the description of the large scale asymptotics of Poisson's kernel $P$ boils down to an homogenization problem. The latter is the focus of sections \ref{secdual}, where we study a dual problem, and \ref{secexp} in which an asymptotic expansion of Poisson's kernel, for $y\cdot n\rightarrow\infty$, is established (see theorem \ref{theodvptP}). This work is the central step in our proof of theorem \ref{theoCVBLtail}. The last step is done in section \ref{secCV}, where we prove on the one hand the convergence towards the boundary layer tail using theorem \ref{theodvptP} and on the other hand the independence of $v_{bl}$ from $a$. Section \ref{secslow} is concerned with the proof of theorem \ref{theoslow}.

\subsection{Notations}

The following notations apply for the rest of the paper. The half-space $\{y\cdot n-a>0\}$ is denoted by $\Omega_{n,a}$ and in the sequel, for any $y\in\overline{\Omega_{n,a}}$ and $r>0$, $D(y,r):=B(y,r)\cap\Omega_{n,a}$ and $\Gamma(y,r):=B(y,r)\cap\partial\Omega_{n,a}$. The case when $a=0$ is frequently used: $\Omega_{n,0}=:\Omega_n$. For a function $H=H(y,\tilde{y})$ depending on $y,\ \tilde{y}\in\mathbb R^d$ we may use the following notation: $\partial_{1,\alpha}H:=\partial_{y_\alpha}H$ (resp. $\partial_{2,\alpha}H:=\partial_{\tilde{y}_\alpha}H$) for all $\alpha\in\{1,\ldots\ d\}$. The vectors $e_1,\ldots\ e_d$ represent the canonical basis of $\mathbb R^d$. The matrix $\mathrm{M}\in M_d\left(\mathbb R\right)$ is an orthogonal matrix such that $\mathrm{M}e_d=n$; $\mathrm{N}\in M_{d,d-1}\left(\mathbb R\right)$ is the matrix of the $d-1$ first columns of $\mathrm{M}$. All along these lines, $C>0$ denotes an arbitrary constant independent from $\varepsilon$.

\selectlanguage{english}

\pagestyle{plain}

\section{Review of the rational and small divisors settings}
\label{secdgvnm}

In this section we make a short review of the mathematical results on system \eqref{sysbl}. We concentrate on giving precise statements for the theorems announced in the introduction (cf. section \ref{subsechomobl}) and insist much more on the small divisors case, whose existence part is useful to us. To determine the role of $n$, we make the change of variable $z=\mathrm{M}^Ty$ in \eqref{sysbl}. One obtains that $v(z):=v_{bl}(\mathrm{M}z)$ solves
\begin{equation}\label{sysblz}
\left\{
\begin{array}{rll}
-\nabla\cdot B(\mathrm{M}z)\nabla v&=0,&z_d>a\\
v&=v_0(\mathrm{M}z),&z_d=a
\end{array}
\right. .
\end{equation}
The family of matrices $B=B^{\alpha\beta}(y)\in M_N\left(\mathbb R\right)$, indexed by $1\leq\alpha,\beta\leq d$, satisfies for every $i,\ j\in\{1,\ldots\ N\}$,
\begin{equation*}
B_{ij}=\mathrm{M}A_{ij}\mathrm{M}^T.
\end{equation*}
and is hence $1$-periodic, elliptic and smooth.

\subsection{Rational case}

This case has been studied by many authors (see \cite{allam,moscovog}). The assumption $n\in\mathbb Q\mathbb R^d$ simplifies much the existence and convergence proof. Indeed, as $n$ has rational coordinates, one can choose an orthogonal matrix $\mathrm{M}$ with columns in $\mathbb R\mathbb Q^d$ sending $e_d$ on $n$. Subsequently, there exists a $d$-uplet of periods $(L_1,\ldots\ L_d)$ such that 
\begin{equation*}
z\mapsto B(\mathrm{M}z)\qquad\mbox{and}\qquad z\mapsto v_0(\mathrm{M}z)
\end{equation*}
are $(L_1,\ldots\ L_d)$-periodic functions. Without loss of generality, let us assume that $L_1=\ldots=L_d=1$. We can also fix $a=0$ for the moment. We come back later to this hypothesis. Then, lifting the Dirichlet data $v_0$ using $\varphi\in C^\infty_c\left(\mathbb R\right)$, compactly supported in $[-1,1]$ such that $0\leq\varphi\leq 1$ and $\varphi\equiv 1$ on $\left[-\frac{1}{2},\frac{1}{2}\right]$, yields that $\tilde{v}:=v-\varphi(z_d)v_0(\mathrm{M}(z',0))$ solves
\begin{equation}\label{sysblztilde}
\left\{
\begin{array}{rll}
-\nabla\cdot B(\mathrm{M}z)\nabla \tilde{v}&=\nabla\cdot B(\mathrm{M}z)\nabla\left(\varphi(z_d)v_0(\mathrm{M}(z',0))\right),&z_d>0\\
\tilde{v}&=0,&z_d=0
\end{array}
\right. .
\end{equation}
An appropriate framework to write a variational formulation for \eqref{sysblztilde} is the completion of $C^\infty_c\left(\mathbb T^{d-1}\times\mathbb R_+\right)$ with respect to the $L^2\left(\mathbb T^{d-1}\times\mathbb R_+\right)$ of the gradient. The fact that the source term $\nabla\cdot B(\mathrm{M}z)\nabla\left(\varphi(z_d)v_0(\mathrm{M}(z',0))\right)$ in \eqref{sysblztilde} is compactly supported in the normal direction, allows to resort to a Poincar\'e inequality. We can prove the existence of a weak solution $\tilde{v}$ to \eqref{sysblztilde} by the mean of the Lax-Milgram lemma.

The existence part asserts that $\nabla v\in L^2\left(\mathbb T^{d-1}\times\mathbb R_+\right)$. We aim at showing that $\nabla v$ actually decays faster. The key observation is that for any $k\in\mathbb N$, $v^{(k)}:=v(z',z_d-k)$ defined for $z_d>k$ solves
\begin{equation*}
\left\{
\begin{array}{rll}
-\nabla\cdot B(\mathrm{M}z)\nabla v^{(k)}&=0,&z_d>k\\
v^{(k)}&=v_0(\mathrm{M}z),&z_d=k
\end{array}
\right. .
\end{equation*}
Hence, one has for all $k\in\mathbb N$ the St-Venant estimate
\begin{equation}\label{StVenantest}
\left\|\nabla v\right\|_{L^2\left(\mathbb T^{d-1}\times]k+1,\infty[\right)}\leq C\left[\left\|\nabla v\right\|_{L^2\left(\mathbb T^{d-1}\times]k,\infty[\right)}-\left\|\nabla v\right\|_{L^2\left(\mathbb T^{d-1}\times]k+1,\infty[\right)}\right],
\end{equation}
which yields the exponential decay of $\left\|\nabla v\right\|_{L^2\left(\mathbb T^{d-1}\times]t,\infty[\right)}$, when $t\rightarrow\infty$. The existence of the boundary layer tail $v_{bl}^{0,\infty}\in\mathbb R^N$ comes from the fact that $\int_{\mathbb T^{d-1}\times]k,k+1[}v(t)dt$ is a Cauchy sequence. The decay of higher order derivatives follows from elliptic regularity (see \cite{adn2}). This implies, through Sobolev injections, pointwise convergence of $v(z',z_d)$ towards $v_{bl}^{0,\infty}$, at an exponential rate. 

Let us come back to the assumption $a=0$. Let $a$ be any real number, $v$ the associated solution of \eqref{sysblz} and denote by $\bar{a}=a-[a]$ its fractional part. Then, $v^a=v(\cdot+ae_d)$ 
\begin{equation*}
\left\{
\begin{array}{rll}
-\nabla\cdot B\left(\mathrm{M}(z+\bar{a}e_d\right)\nabla v^a&=0,&z_d>0\\
v^a&=v_0\left(\mathrm{M}(z+\bar{a}e_d)\right),&z_d=0
\end{array}
\right. .
\end{equation*}
If we now assume that $N=1$, $d=2$ and that $B(\mathrm{M}\cdot)$ is the constant identity matrix $\Idd_2$, we can carry out Fourier analysis to compute the tail. We get that the tail $v^{a,\infty}_{bl}$ of $v$ and $v^a$ is equal to
\begin{equation*}
v_{bl}^{a,\infty}=\int_0^1v_0\left(\mathrm{M}(z',\bar{a})\right)dz',
\end{equation*}
which depends on $\bar{a}$. Thus, the boundary layer tail is not independent from $a$ in this rational setting. We now summarize the preceding results (forgetting about the assumptions $L_1=\ldots=L_d=1$ and $a=0$) in the:
\begin{theo}[lemma 4.4 in \cite{allam}, appendix $6$ in \cite{moscovog}]\label{theoexpdecayrat}
Assume that $n\in\mathbb R\mathbb Q^d$. Then,
\begin{enumerate}
\item there exists a solution $v\in C^\infty\left(\mathbb R^{d-1}\times[a,\infty[\right)$ of \eqref{sysblz}, unique under the condition that for all $R>0$
\begin{equation*}
\nabla v\in L^2\bigl((-R,R)^{d-1}\times]a,\infty[\bigr),
\end{equation*}
\item and $\kappa>0$, $v^{a,\infty}_{bl}\in\mathbb R^N$ depending on $a$ such that for all $\alpha\in\mathbb N^d$, for all $z=(z',z_d)\in\mathbb R^{d-1}\times]a,\infty[$,
\begin{equation}\label{exdecayratcase}
e^{\kappa(z_d-a)}\left|\partial_z^\alpha\left(v(z)-v^{a,\infty}_{bl}\right)\right|\leq C_\alpha.
\end{equation}
\end{enumerate}
\end{theo}
We conclude this section by a remark. It is concerned with the exponent $\kappa=\kappa_n$ in \eqref{exdecayratcase}. In \cite{Neuss00}, the case of a two-dimensional layered media $\Omega_{n,0}$ is considered. It is shown that exponential convergence of $\nabla v_{bl}$ to $0$, uniform in $n\in\mathbb R\mathbb Q^d$, cannot be expected. Indeed, depending on $n$, $\kappa_n$ can be arbitrarily small.

\subsection{Small divisors case}
\label{subsecdgvnmdiv}
All the results we recall here stem from the original article \cite{dgvnm} by G\'erard-Varet and Masmoudi. The periodic framework in the rational case makes the study of \eqref{sysblz} more simple for the main reason that it yields compactness in the horizontal direction. One can thus rely on Poincar\'e-Wirtinger inequalities, which are essential to prove the St-Venant estimate \eqref{StVenantest}. In the non rational case, i.e. when $n\notin\mathbb Q\mathbb R^d$, one also attempts to recover a periodic setting. Note that for all $z=(z',a)\in\mathbb R^{d-1}\times\{a\}$, $v_0\left(\mathrm{M}(z',a)\right)=v_0\left(\mathrm{N}z'+\mathrm{M}(0,a)\right)$, where $\mathrm{N}\in M_{d,d-1}\left(\mathbb R\right)$ is the matrix of the $d-1$ first columns of $\mathrm{M}$. Therefore, $v_0(\mathrm{M}\cdot)$ is quasiperiodic in $z'$, i.e. there exists $V_0=V_0(\theta,t)\in\mathbb R^N$ defined for $\theta\in\mathbb T^d$ and $t\in\mathbb R$ such that for all $z=(z',a)\in\mathbb R^{d-1}\times\{a\}$, $v_0\left(\mathrm{M}(z',a)\right)=V_0(\mathrm{N}z',a)$. Similarly, there exists $\mathcal B=\mathcal B(\theta,t)$ such that for all $z=(z',z_d)\in\mathbb R^{d-1}\times\mathbb R$, $B\left(\mathrm{M}(z',z_d)\right)=\mathcal B(\mathrm{N}z',z_d)$. Hence, one looks for a quasiperiodic solution of \eqref{sysblz}; for details concerning quasiperiodic functions the reader is referred to \cite{JKO}, section $7.1$.  Assume that there exists $V=V(\theta,t)\in\mathbb R^N$ defined for $\theta\in\mathbb T^d$ and $t>a$ such that for all $z=(z',z_d)\in\mathbb R^{d-1}\times]a,\infty[$,
\begin{equation*}
v\left(\mathrm{M}(z',z_d)\right)=V(\mathrm{N}z',z_d).
\end{equation*}
Now, if $V$ solves
\begin{equation}\label{sysVthetat}
\left\{
\begin{array}{rll}
-\begin{pmatrix}
\mathrm{N}^T\nabla_\theta\\
\partial_t
\end{pmatrix}\cdot
\mathcal B(\theta,t)
\begin{pmatrix}
\mathrm{N}^T\nabla_\theta\\
\partial_t
\end{pmatrix}V&=0,&t>a\\
V&=V_0,&t=a
\end{array}
\right. ,
\end{equation}
then $v=v(z',z_d)=V(\mathrm{N}z',z_d)$ solves \eqref{sysblz}. 

We focus now on system \eqref{sysVthetat}. We examine successively the existence of a solution and the convergence towards a constant vector field. By making the change of unknown function, we have gained compactness in the horizontal direction, but lost ellipticity of the differential operator. System \eqref{sysVthetat} is however well-posed (see \cite{dgvnm} proposition $2$): \emph{there exists a unique variational solution $V\in C^\infty\left(\mathbb T^d\times]a,\infty[\right)$ such that for all $l\geq 0$, for all $\alpha\in\mathbb N^d$,}
\begin{equation}\label{aprioriboundV}
\left\|\partial^l_t\partial^\alpha_\theta\begin{pmatrix}
\mathrm{N}^T\nabla_\theta\\
\partial_t
\end{pmatrix}\right\|_{L^2\left(\mathbb T^d\times]a,\infty[\right)}\!\!=\int_{\mathbb T^d}\!\int_a^\infty\!\left(\bigl|N^T\nabla_\theta\partial^l_t\partial^\alpha_\theta V\bigr|^2\!+\bigl|\partial^{l+1}_t\partial^\alpha_\theta V\bigr|^2\right)d\theta dt<\infty.
\end{equation}
To compensate for the lack of ellipticity, one uses a regularization method. We add a corrective term $-\iota\Delta$ to the operator. The regularized system is 
\begin{equation}\label{sysVthetatreg}
\left\{
\begin{array}{rll}
-\begin{pmatrix}
\mathrm{N}^T\nabla_\theta\\
\partial_t
\end{pmatrix}\cdot
\mathcal B(\theta,t)
\begin{pmatrix}
\mathrm{N}^T\nabla_\theta\\
\partial_t
\end{pmatrix}V^\iota-\iota\Delta V^\iota&=0,&t>a\\
V^\iota&=V_0,&t=a
\end{array}
\right. .
\end{equation}
The regularizing term yields ellipticity of the differential operator. Lifting the boundary data $V_0$ into a function compactly supported in the direction normal to the boundary, makes it possible, as in the rational case, to prove the existence of a variational solution $V^\iota$ to \eqref{sysVthetatreg}. Carrying out energy estimates on \eqref{sysVthetatreg} gives \eqref{aprioriboundV}-like a priori bounds on $V^\iota$, uniform in $\iota$. Thanks to this compactness on the sequence $\left(V^\iota\right)_{\iota>0}$, one can extract a subsequence, which converges weakly to $V$ variational solution of \eqref{sysVthetat} when $\iota\rightarrow 0$. Uniqueness follows from the a priori bounds. An important point is that the justification of this well-posedness result does not resort to the small divisors assumption. 

We come to the asymptotical analysis far away from the boundary. Sobolev injections and the a priori bounds \eqref{aprioriboundV} yield pointwise convergence to $0$ of $\mathrm{N}^T\nabla_\theta V(\theta,t)$, $\partial_t V(\theta,t)$ and their derivatives, when $t\rightarrow\infty$, uniformly in $\theta\in\mathbb T^d$. Let us investigate this convergence in more detail. Without loss of generality, we assume temporarily that $a=0$. The idea is to establish a St-Venant estimate on $\begin{pmatrix}
\mathrm{N}^T\nabla_\theta\\
\partial_t
\end{pmatrix}V$. In the same spirit as in the rational case, we look at the quantity
\begin{equation*}
K(T):=\int_{\mathbb T^d}\int_T^\infty\left(\bigl|\mathrm{N}^T\nabla_\theta V\bigr|^2+\bigl|\partial_t V\bigr|^2\right)d\theta dt
\end{equation*}
defined for $T>0$. One proves that
\begin{equation}\label{integrodiffdiv}
K(T)\leq C\left(-K'(T)\right)^\frac{1}{2}\left(\int_{\mathbb T^d}\bigl|\widetilde{V}(\theta,T)\bigr|^2d\theta\right)^\frac{1}{2},
\end{equation}
with $\widetilde{V}(\theta,T):=V(\theta,T)-\int_{\mathbb T^d}V(\cdot,T)$. The stake is to control the second factor in the r.h.s. of \eqref{integrodiffdiv}. The key argument of \cite{dgvnm} is a type of Poincar\'e-Wirtinger inequality implied by the small divisors assumption. \emph{Assume that $n\notin\mathbb R\mathbb Q^d$ satisfies \eqref{smalldivisors1}. Then, for all $1<p<\infty$, there exists $C_p>0$ such that}
\begin{equation}\label{smalldivminoration}
\int_{\mathbb T^d}\bigl|\widetilde{V}(\theta,T)\bigr|^2d\theta\leq C_p\left(\int_{\mathbb T^d}\bigl|\mathrm{N}^T\nabla_\theta V(\theta,T)\bigr|^2d\theta\right)^\frac{1}{p}.
\end{equation}
Note that the r.h.s. does not involve the $L^2(\mathbb T^d)$ norm of $\nabla_\theta V$, but the norm of the incomplete gradient $\mathrm{N}^T\nabla_\theta V$. We give a sketch of the reasoning leading to \eqref{smalldivminoration}. The proof relies on Fourier analysis in the tangential variable. In order to emphasize the role of the small divisors assumption, let us give an equivalent statement of \eqref{smalldivisors1}: taking for $n_1,\ldots\ n_{d-1}$ the $d-1$ first columns $\mathrm{N}\in M_{d,d-1}\left(\mathbb R\right)$ of $\mathrm{M}$, \eqref{smalldivisors1} becomes
\begin{equation}\label{smalldivisors2}
\left|\mathrm{N}^T\xi\right|\geq C\left|\xi\right|^{-d-\tau}.
\end{equation}
Let $T>0$ be fixed, $1<p<\infty$ and $p'$ its conjugate H\"older exponent: $\frac{1}{p}+\frac{1}{p'}=1$. By Parceval's equality we can compute the $L^2\left(\mathbb T^d\right)$ norm on the l.h.s. of \eqref{smalldivminoration} and then apply H\"older's inequality 
\begin{align*}
\int_{\mathbb T^d}\bigl|\widetilde{V}(\theta,T)\bigr|^2d\theta&\leq\left(\sum_{\xi\in\mathbb Z^d\setminus\{0\}}\Bigl|\widehat{\widetilde{V}}(\xi,T)\Bigr|^2\frac{1}{|\xi|^{2(d+\tau)}}\right)^\frac{1}{p}\left(\sum_{\xi\in\mathbb Z^d\setminus\{0\}}\Bigl|\widehat{\widetilde{V}}(\xi,T)\Bigr|^2|\xi|^{2\frac{d+\tau}{p-1}}\right)^\frac{1}{p'},
\end{align*}
with $\alpha:=\frac{1}{p}$, $\beta:=\frac{1}{p'}$ and $\gamma:=\frac{d+\tau}{p}=\frac{d+\tau}{p'(p-1)}$. For the first factor in the r.h.s., which represents the norm of $\widetilde{V}(\cdot,T)$ in an homogeneous Sobolev space with negative exponent, we have, using \eqref{smalldivisors2}:
\begin{multline*}
\left\|\mathrm{N}^T\nabla_\theta\widetilde{V}(\theta,T)\right\|^2_{L^2\left(\mathbb T^d\right)}=\sum_{\xi\in\mathbb Z^d\setminus\{0\}}\Bigl|2i\pi\widehat{\widetilde{V}}(\xi,T)\Bigr|^2\left|\mathrm{N}^T\xi\right|^2\geq C\sum_{\xi\in\mathbb Z^d\setminus\{0\}}\Bigl|\widehat{\widetilde{V}}(\xi,T)\Bigr|^2\frac{1}{\left|\xi\right|^{2(d+\tau)}}.
\end{multline*}
This bound and \eqref{aprioriboundV} imply in particular a control on $\bigl\|\widetilde{V}(\cdot,T)\bigr\|_{H^s\left(\mathbb T^d\right)}$ for all $s\geq 0$. Hence, one bounds the second factor by a constant independent from $T$.

The bound \eqref{smalldivisors2}, in combination with \eqref{integrodiffdiv}, yields for every $1<p<\infty$ the differential inequation on $K$, $K(T)\leq C_p\left(-K'(T)\right)^\frac{p+1}{2p}$ from which we get the decay: for all $T>0$, $0\leq K(T)\leq C_p T^\frac{p+1}{1-p}$. We do the same on higher order derivatives considering for $s\in\mathbb N$,
\begin{equation*}
K_s(T):=\sum_{0\leq|\alpha|+l\leq s}\int_{\mathbb T^d}\int_T^\infty\left(\bigl|\mathrm{N}^T\partial^l_t\partial^\alpha_\theta\nabla_\theta V\bigr|^2+\bigl|\partial^{l+1}_t\partial^\alpha_\theta V\bigr|^2\right)d\theta dt.
\end{equation*}
The existence of a boundary layer tail follows from a procedure similar to the rational case. For arbitrary $a\in\mathbb R$, one can prove (see proposition $5$ in \cite{dgvnm}) that the boundary layer tail does not depend on $a$. This fact is a characteristic of the non rational setting: G\'erard-Varet and Masmoudi prove it under the small divisors assumption. Note that we manage to free ourselves from this assumption in section \ref{secindep}. To summarize (see \cite{dgvnm} propositions $4$ and $5$): \emph{if $n\notin\mathbb R\mathbb Q^d$ satisfies the small divisors assumption \eqref{smalldivisors2}, then there exists a constant vector field $v_{bl}^{\infty}\in\mathbb R^N$, independent from $a$, such that for all $m\in\mathbb N$, for all $\alpha\in\mathbb N^d$, $l\in\mathbb N$, for all $t>a$,}
\begin{equation*}
\sup_{\theta\in\mathbb T^d}\left|(1+|t-a|^m)\partial^\alpha_\theta\partial^l_t\left(V(\theta,t)-v^{\infty}_{bl}\right)\right|\leq C_{m,\alpha}.
\end{equation*}

We end this section by a translation of the existence and convergence statement for $V$ into a statement for $v$ solution of \eqref{sysblz}. We recall that for all $z=(z',z_d)\in\mathbb R^{d-1}\times]a,\infty[$, $v(z',z_d)=V(\mathrm{N}z',z_d)$. A solution $v$ of \eqref{sysblz} (or $v_{bl}$ of \eqref{sysbl}) built like this is called a \emph{variational solution}.
\begin{theo}[\cite{dgvnm}]\label{theodecaydiv}
Assume that $n\notin\mathbb R\mathbb Q^d$.
\begin{enumerate}
\item Then, there exists a solution $v\in C^\infty\left(\mathbb R^{d-1}\times[a,\infty[\right)$ of \eqref{sysblz} satisfying
\begin{subequations}
\begin{align}
&\left\|\nabla v\right\|_{L^\infty\left(\mathbb R^{d-1}\times]t,\infty[\right)}\stackrel{t\rightarrow\infty}{\longrightarrow}0\label{nablavtends0}\\
&\int_a^\infty\left\|\partial_{z_d}v(\cdot,t)\right\|^2_{L^\infty\left(\mathbb R^{d-1}\right)}dt<\infty \label{partial_z_dv}.
\end{align}
\end{subequations}
\item Moreover, if $n$ satisfies the small divisors assumption \eqref{smalldivisors1}, then for all $m\in\mathbb N$, for all $\alpha\in\mathbb N^d$, for all $z=(z',z_d)\in\mathbb R^{d-1}\times]a,\infty[$,
\begin{equation*}
\left(1+|z_d-a|^m\right)\left|\partial^\alpha_z\left(v(z)-v_{bl}^\infty\right)\right|\leq C_{\alpha,m}.
\end{equation*}
\end{enumerate}
\end{theo}
Estimates \eqref{nablavtends0} and \eqref{partial_z_dv} rely on \eqref{aprioriboundV}. Indeed, for every $k\in\mathbb N$,
\begin{equation*}
\left\|\mathrm{N}^T\nabla_\theta V\right\|_{H^k\left(\mathbb T^d\times]t,\infty[\right)}+\left\|\partial_t V\right\|_{H^k\left(\mathbb T^d\times]t,\infty[\right)}\stackrel{t\rightarrow\infty}{\longrightarrow} 0
\end{equation*}
which together with Sobolev's injection theorem yields \eqref{nablavtends0} for $k$ sufficiently big. As 
\begin{equation*}
\left\|\partial_t V\right\|_{H^k\left(\mathbb T^d\times]a,\infty[\right)}\geq \left\|\left\|\partial_t V\right\|_{H^k_\theta\left(\mathbb T^d\right)}\right\|_{L^2_t\left(]0,\infty[\right)},
\end{equation*}
a similar reasoning leads to \eqref{partial_z_dv}. The a priori bounds \eqref{aprioriboundV}, contain further information: for $k\geq 0$ sufficiently large and $k'\geq 1$,
\begin{multline}\label{ineqintnormalderivees}
\left\|\begin{pmatrix}
\mathrm{N}^T\nabla_\theta\\
\partial_t
\end{pmatrix}V\right\|_{H^{k'+k-1}\left(\mathbb T^d\times]a,\infty[\right)}
\geq 
\left\|\left\|\begin{pmatrix}
\nabla_\theta\\
\partial_t
\end{pmatrix}^{k'-1}\begin{pmatrix}
\mathrm{N}^T\nabla_\theta\\
\partial_t
\end{pmatrix}V\right\|_{H^k_\theta\left(\mathbb T^d\right)}\right\|_{L^2_t\left(]a,\infty[\right)}\\
\geq C
\left\|\left\|\begin{pmatrix}
\nabla_\theta\\
\partial_t
\end{pmatrix}^{k'-1}\begin{pmatrix}
\mathrm{N}^T\nabla_\theta\\
\partial_t
\end{pmatrix}V\right\|_{L^\infty_\theta\left(\mathbb T^d\right)}\right\|_{L^2_t\left(]a,\infty[\right)}
\geq C
\left\|\left\|\nabla^{k'}v\right\|_{L^\infty_{z'}\left(\mathbb R^{d-1}\right)}\right\|_{L^2_{z_d}\left(]a,\infty[\right)}.
\end{multline}
We resort to the latter in section \ref{sechompoisson} with $k'=1$ or $2$.
\selectlanguage{english}

\pagestyle{plain}

\section{Poisson's integral representation of the variational solution}
\label{secprel}

\emph{One can always assume that $a=0$. We do so for the rest of the paper (except in section \ref{secindep}), even if it means to work with $v_{bl}^a:=v_{bl}(\cdot+an)$ instead of $v_{bl}$ solving \eqref{sysbl}.} The main advantage of this assumption is that the domain $\Omega_n=\{y\cdot n>0\}$ is invariant under the scaling $y\mapsto\varepsilon y$ for $\varepsilon>0$. The purpose of this section (see in particular section \ref{secintreprfor}) is to prove uniqueness for the solution to \eqref{sysbl} in a larger class than that of \cite{dgvnm}. This result (theorem \ref{theov=w} below) allows to represent the variational solution by the mean of Poisson's kernel.

\subsection{Estimates on Green's and Poisson's kernels}
\label{subsecGreenPoisson}

Let $G=G\left(y,\tilde{y}\right)\in M_N(\mathbb R)$ solving, for all $\tilde{y}\in\Omega_n$, the elliptic system
\begin{equation}\label{sysGreen}
\left\{
\begin{array}{rll}
-\nabla_y\cdot A(y)\nabla_y G\left(y,\tilde{y}\right)&=\delta(y-\tilde{y})\Idd_N, & y\cdot n>0\\
G(y,\tilde{y})&=0,& y\cdot n=0
\end{array}
\right. ,
\end{equation}
with source term $\delta(\cdot-\tilde{y})$ ($\delta(\cdot)$ is the Dirac distribution). The matrix $G$ is called the Green kernel associated to the operator $-\nabla\cdot A(y)\nabla\cdot$ and to the domain $\Omega_n$. Its existence in $\Omega_n$ is proved for $d\geq 3$ in \cite{HofKim07} (see theorem $4.1$), and for $d=2$ in \cite{DK09} (see theorem $2.12$). Similarly, $G^*=G^*\left(y,\tilde{y}\right)\in M_N(\mathbb R)$ is the Green kernel associated to the transposed operator $-\nabla\cdot A^*(y)\nabla\cdot$ and the domain $\Omega_n$, $A^*$ being the transpose of the tensor $A$ defined for all $\alpha,\ \beta\in\{1,\ldots\ d\}$ and for all $i,\ j\in\{1,\ldots N\}$ by 
\begin{equation*}
\left(A^*\right)^{\alpha\beta}_{ij}=A^{\beta\alpha}_{ji}=\bigl(A^{\beta\alpha}\bigr)^T_{ij}.
\end{equation*}
From the smoothness of $A$ and local regularity estimates, it follows that 
\begin{equation*}
G\in C^\infty\left(\Omega_n\times\Omega_n\setminus\{y=\tilde{y}\}\right).
\end{equation*}
Moreover, the following symmetry property holds: for all $y,\ \tilde{y}\in\Omega_n$, 
\begin{equation*}
G^T(y,\tilde{y})=G^*(\tilde{y},y).
\end{equation*}
Using Green's kernel, one defines another function, the Poisson kernel $P=P\left(x,\tilde{x}\right)\in M_N(\mathbb R)$: for all $i,\ j\in\{1,\ldots N\}$, for all $y\in\Omega_n$, $\tilde{y}\in\partial\Omega_n$,
\begin{subequations}\label{exprP}
\begin{align}
P_{ij}(y,\tilde{y})&:=-A^{\alpha\beta}_{kj}(\tilde{y})\partial_{\tilde{y}_\alpha}G_{ik}(y,\tilde{y})n_\beta\\
&=-\left(A^*\right)^{\beta\alpha}_{jk}(\tilde{y})\partial_{\tilde{y}_\alpha}G^*_{ki}(\tilde{y},y)n_\beta\\
&=-\left[A^{*,\beta\alpha}(\tilde{y})\partial_{\tilde{y}_\alpha}G^*(\tilde{y},y)n_\beta\right]_{ji}\\
&=-\left[A^*(\tilde{y})\nabla_{\tilde{y}}G^*(\tilde{y},y)\cdot n\right]^T_{ij}.
\end{align}
\end{subequations}
The kernel $P^*=P^*\left(y,\tilde{y}\right)\in M_N(\mathbb R)$ is defined in the same way as $G^*$.

If one considers Green's kernel $G^0=G^0(y,\tilde{y})\in M_N\left(\mathbb R\right)$ associated to the constant coefficients elliptic operator $-\nabla\cdot A^0\nabla\cdot$ and to the domain $\Omega_n$, there exists $C>0$ (see \cite{adn2} section $6$ for a statement and \cite{Schuwil77} V.$4.2$ (Satz $3$) for a proof) such that for all $\Lambda_1,\ \Lambda_2\in \mathbb N^d$, for all $y,\ \tilde{y}\in\Omega_n$, $y\neq \tilde{y}$,
\begin{subequations}\label{estG0lambda}
\begin{alignat}{2}
\left|G^0(y,\tilde{y})\right|&\leq C\left(\left|\logg\left|y-\tilde{y}\right|\right|+1\right),&\quad\mbox{if } d=2,\\
\left|G^0(y,\tilde{y})\right|&\leq \frac{C}{\left|y-\tilde{y}\right|^{d-2}},&\quad\mbox{if } d\geq 3,\\
\left|\partial_y^{\Lambda_1}\partial_{\tilde{y}}^{\Lambda_2}G^0(y,\tilde{y})\right|&\leq \frac{C}{\left|y-\tilde{y}\right|^{d-2+|\Lambda_1|+|\Lambda_2|}},&\quad\mbox{if } |\Lambda_1|+|\Lambda_2|\geq 1.
\end{alignat}
\end{subequations}
One has similar estimates on Poisson's kernel $P^0=P^0(y,\tilde{y})\in M_N\left(\mathbb R\right)$ and its derivatives. Such bounds on Green's kernel and its derivatives are not known for operators with non constant coefficients. Let us temporarily assume that neither {\bf (A2)} nor {\bf (A3)} hold. We know then from \cite{KangKim10} (see theorem $3.3$) that under certain smoothness assumptions on the coefficients there exists $R_{max}\in]0,\infty]$ and $C>0$, such that for all $y,\ \tilde{y}\in\Omega_n$, $|y-\tilde{y}|\!<\!R_{max}$ implies
\begin{equation}\label{estKangKim10}
|G(y,\tilde{y})|\leq \frac{C}{|y-\tilde{y}|^{d-2}}.
\end{equation}
If $R_{max}<\infty$, estimate \eqref{estKangKim10} does not provide a control of $G(y,\tilde{y})$ for $y$ and $\tilde{y}$ far from each other. However, under the periodicity assumption {\bf (A2)}, we have the following improved global bounds:
\begin{lem}[\cite{alin} theorem $13$ and lemma $21$, \cite{dgvnm2} lemma $6$]
There exists $C>0$, such that
\begin{enumerate} 
\item for all $y,\ \tilde{y}\in\Omega_n$, $y\neq\tilde{y}$,
\begin{subequations}\label{estALinG}
\begin{alignat}{2}
\left|G(y,\tilde{y})\right|&\leq C\left(\left|\logg\left|y-\tilde{y}\right|\right|+1\right),&\quad\mbox{if } d=2,\label{estALinGd=2}\\
\left|G(y,\tilde{y})\right|&\leq\frac{C}{\left|y-\tilde{y}\right|^{d-2}},&\quad\mbox{if } d\geq 3,\label{estALinGd>2}
\end{alignat}
and for all $d\geq 2$,
\begin{align}
\left|G(y,\tilde{y})\right|&\leq C\frac{\left(y\cdot n\right)\left(\tilde{y}\cdot n\right)}{\left|y-\tilde{y}\right|^d},\label{estALinGdqcq}\\
\left|\nabla_yG(y,\tilde{y})\right|&\leq \frac{C}{\left|y-\tilde{y}\right|^{d-1}},\label{estALinnablaGdqcq}\\
\left|\nabla_yG(y,\tilde{y})\right|&\leq C\left(\frac{\tilde{y}\cdot n}{\left|y-\tilde{y}\right|^d}+\frac{\left(y\cdot n\right)\left(\tilde{y}\cdot n\right)}{\left|y-\tilde{y}\right|^{d+1}}\right),\label{estALinnablaGdqcq2}\\
\left|\nabla_y\nabla_{\tilde{y}}G(y,\tilde{y})\right|&\leq \frac{C}{\left|y-\tilde{y}\right|^d};
\end{align}
\end{subequations}
\item for all $y\in\Omega_n$, for all $\tilde{y}\in\partial\Omega_n$, for all $d\geq 2$,
\begin{subequations}\label{estALinP}
\begin{align}
\left|P(y,\tilde{y})\right|&\leq C\frac{y\cdot n}{\left|y-\tilde{y}\right|^d},\label{estALinPP}\\
\left|\nabla_yP(y,\tilde{y})\right|&\leq C\left(\frac{1}{\left|y-\tilde{y}\right|^d}+\frac{y\cdot n}{\left|y-\tilde{y}\right|^{d+1}}\right)\label{estALinPnabla}.
\end{align}
\end{subequations}
\end{enumerate}
\end{lem}
By continuity of $G$ and its derivatives, up to the boundary, the estimates on Green's kernel naturally extend to $y\neq\tilde{y}\in\overline{\Omega_n}$. These bounds are of constant use in our work: for a proof in the half-space see \cite{dgvnm2} (appendix A). The key argument is due to Avellaneda and Lin. The large scale description of $G$ boils down to an homogenization problem. In the paper \cite{alin}, the authors face this homogenization problem under the periodicity assumption {\bf (A2)} and manage to get uniform local estimates on $u^\varepsilon=u^\varepsilon(x)$ satisfying
\begin{equation}\label{sysuepsavlin}
\left\{\begin{array}{rll}
-\nabla\cdot A\left(\frac{x}{\varepsilon}\right)\nabla u^\varepsilon&=f,& x\in D(0,1)\\
u^\varepsilon&=g,& x\in \Gamma(0,1)
\end{array}\right. .
\end{equation}
We recall the two local estimates useful in the sequel:
\begin{theo}[local boundary estimate, \cite{alin} lemma $12$]\label{theolocalboundaryueps}
Let $\mu,\ \delta$ be positive real numbers, $\mu<1$, $F\in L^{d+\delta}\left(D(0,1)\right)$ and $g\in C^{0,1}\left(\Gamma(0,1)\right)$. Assume that $f=\nabla\cdot F$.\\
There exists $C>0$ such that for all $\varepsilon>0$, if $u^\varepsilon$ belongs to $L^2\left(D(0,1)\right)$ and satisfies \eqref{sysuepsavlin}, then
\begin{equation}\label{boundaryest}
\left\|u^\varepsilon\right\|_{C^{0,\mu}\left(D\left(0,\frac{1}{2}\right)\right)}\leq C\left[\left\|u^\varepsilon\right\|_{L^2\left(D(0,1)\right)}+\left\|F\right\|_{L^{d+\delta}\left(D(0,1)\right)}+\|g\|_{C^{0,1}\left(\Gamma(0,1)\right)}\right].
\end{equation}
\end{theo}
\begin{theo}[local boundary gradient estimate, \cite{alin} lemma $20$]\label{theolocalboundarygradientueps}
Let $\nu,\ \delta$ be positive real numbers, $\nu<1$, $f\in L^{d+\delta}\left(D(0,1)\right)$ and $g\in C^{1,\nu}\left(\Gamma(0,1)\right)$.\\
There exists $C>0$ such that for all $\varepsilon>0$, if $u^\varepsilon$ belongs to $L^\infty\left(D(0,1)\right)$ and satisfies \eqref{sysuepsavlin}, then
\begin{equation}\label{boundarygradest}
\left\|\nabla u^\varepsilon\right\|_{L^\infty\left(D\left(0,\frac{1}{2}\right)\right)}\leq C\left[\left\|u^\varepsilon\right\|_{L^\infty\left(D(0,1)\right)}+\left\|f\right\|_{L^{d+\delta}\left(D(0,1)\right)}+\|g\|_{C^{1,\nu}\left(\Gamma(0,1)\right)}\right].
\end{equation}
\end{theo}

\subsection{Integral representation formula}
\label{secintreprfor}

We aim at showing that the solution $v_{bl}$ of \eqref{sysbl} can be represented in terms of an integral formula involving Poisson's kernel. One of the main difficulties arises from the fact that the Dirichlet data $v_0$ of \eqref{sysbl} is not compactly supported in the boundary. The function $v=v(z)\in\mathbb R^N$, such that for all $z\in\mathbb R^{d-1}\times\mathbb R^*_+$, $v(z):=v_{bl}(\mathrm{M}z)$, solves \eqref{sysblz}. Let us now recall what we consider as a solution of \eqref{sysblz}. In fact, we have two kinds of solutions. The first corresponds to the variational construction in \cite{dgvnm}: see section \ref{subsecdgvnmdiv}, in particular theorem \ref{theodecaydiv}. Poisson's kernel $P=P(y,\tilde{y})\in M_N(\mathbb R)$ associated to the domain $\Omega_n$ and the operator $-\nabla\cdot A(y)\nabla\cdot$ makes it possible to define a second function $w_{bl}=w_{bl}(y)\in\mathbb R^N$ solving \eqref{sysbl}. For all $y\in\Omega_n$, we define $w_{bl}(y)$ by
\begin{equation*}
w_{bl}(y)=\int_{\tilde{y}\cdot n=0}P(y,\tilde{y})v_0(\tilde{y})d\tilde{y}
\end{equation*}
and $w$ by for all $z\in\mathbb R^{d-1}\times]0,\infty[$, 
\begin{equation*}
w(z):=w_{bl}(\mathrm{M}z)=\int_{\partial\Omega_n}P(\mathrm{M}z,\tilde{y})v_0(\tilde{y})d\tilde{y}=\int_{\tilde{z}_d=0}P(\mathrm{M}z,\mathrm{M}\tilde{z})v_0(\mathrm{M}\tilde{z})d\tilde{z}.
\end{equation*}
\begin{prop}\label{solreprgreen}
The function $w$ belongs to $C^\infty\left(\mathbb R^{d-1}\times[0,\infty[\right)$ and satisfies \eqref{sysblz}. Furthermore, 
\begin{subequations}
\begin{align}
&\left\|\nabla_zw\right\|_{L^\infty\left(\mathbb R^{d-1}\times]t,\infty[\right)}\stackrel{t\rightarrow\infty}{\longrightarrow}0,\label{nablawtends0}\\
&\partial_{z_d} w\in L^2_{z_d}\left(\mathbb R^*_+,L^\infty_{z'}\bigl(\mathbb R^{d-1}\bigr)\right)\label{partial_z_dw}.
\end{align}
\end{subequations}
\end{prop}
Let us give a sketch of how to deduce these properties from the bound on $\nabla_yP$ given in \eqref{estALinP}. For all $z\in\mathbb R^{d-1}\times]0,\infty[$, 
\begin{align*}
\left|\nabla_zw(z)\right|&\leq\int_{\tilde{z}_d=0}\left|\mathrm{M}^T\nabla_yP(\mathrm{M}z,\mathrm{M}\tilde{z})v_0(\mathrm{M}\tilde{z})\right|d\tilde{z}\\
&\leq C\int_{\tilde{z}_d=0}\left(\frac{1}{|z-\tilde{z}|^d}+\frac{z_d}{|z-\tilde{z}|^{d+1}}\right)d\tilde{z}\\
&\leq C\int_{\mathbb R^{d-1}}\left(\frac{1}{\left(z_d^2+|z'-\tilde{z}'|^2\right)^\frac{d}{2}}+\frac{z_d}{\left(z_d^2+|z'-\tilde{z}'|^2\right)^{\frac{d+1}{2}}}\right)d\tilde{z}'\\
&\leq C\frac{1}{z_d}\int_{\mathbb R^{d-1}}\left(\frac{1}{\left(1+|z'-\tilde{z}'|^2\right)^\frac{d}{2}}+\frac{1}{\left(1+|z'-\tilde{z}'|^2\right)^\frac{d+1}{2}}\right)d\tilde{z}',
\end{align*}
from which one gets \eqref{nablawtends0} as well as \eqref{partial_z_dw}.

Our goal is now to show that the variational solution and the Poisson solution coincide. 
\begin{theo}\label{theov=w}
We have $v=w$.
\end{theo}

We work on the difference $u:=v-w$, which is a $C^\infty$ solution of
\begin{equation*}
\left\{\begin{array}{rll}
-\nabla\cdot B(\mathrm{M}z)\nabla u&=0,&z_d>0\\
u&=0,&z_d=0
\end{array}\right. .
\end{equation*}
We intend to show that $u$ is zero proceeding by duality. Let $f\in C^\infty_c\left(\mathbb R^{d-1}\times]0,\infty[\right)$. We take $U=U(z)\in\mathbb R^N$ the solution to the elliptic boundary value problem
\begin{equation*}
\left\{
\begin{array}{rll}
-\nabla\cdot B^*(\mathrm{M}z)\nabla U&=f,&\quad z_d>0\\
U&=0,&\quad z_d=0
\end{array}
\right. 
\end{equation*}
given by Green's representation formula: for all $z\in\mathbb R^{d-1}\times]0,\infty[$,
\begin{equation*}
U(z)=\int_{\tilde{z}_d>0}G^*(\mathrm{M}z,\mathrm{M}\tilde{z})f(\tilde{z})d\tilde{z}.
\end{equation*}
The bounds on $G^*$ \eqref{estALinGdqcq} and its first-order derivative \eqref{estALinnablaGdqcq}, yield:
\begin{lem}
There is $C>0$ such that for sufficiently large $z\in\mathbb R^{d-1}\times]0,\infty[$, i.e. far enough from the support of $f$,
\begin{subequations}
\begin{align}
\left|U(z)\right|&\leq C\frac{z_d}{\left(z_d^2+|z'|^2\right)^{\frac{d}{2}}},\label{estUaux}\\
\left|\nabla U(z)\right|&\leq C\frac{1}{\left(z_d^2+|z'|^2\right)^{\frac{d}{2}}}\label{estnablaUaux}.
\end{align}
\end{subequations}
\end{lem}
Moreover, thanks to \eqref{partial_z_dv} and \eqref{partial_z_dw}, one manages to estimate $u$ in $L^\infty$: there is $C>0$, such that for all $z\in\mathbb R^{d-1}\times]0,\infty[$,
\begin{equation}\label{estcroissuLinfty}
\left|u(z)\right|\leq\int_0^{z_d}\left|\partial_{z_d}u(z',t)\right|dt\leq z_d^\frac{1}{2}\left(\int_0^{z_d}\left|\partial_{z_d}u(z',t)\right|^2dt\right)^{\frac{1}{2}}\leq Cz_d^\frac{1}{2}.
\end{equation}
We now carry out integrations by parts:
\begin{align*}
&\int_{z_d>0}u(z)f(z)dz=-\int_{z_d>0}u(z)\nabla\cdot B^*(\mathrm{M}z)\nabla U(z)dz\\
&=\int_{z_d>0}B^{\alpha\beta}(\mathrm{M}z)\partial_{z_\beta}u(z)\partial_{z_\alpha}U(z)dz\\
&=-\int_{z_d>0}\partial_{z_\alpha}\left(B^{\alpha\beta}(\mathrm{M}z)\partial_{z_\beta}u(z)\right)U(z)dz=0.
\end{align*}
To fully justify the preceding equalities we have to integrate by parts on the bounded domain $[-R,R]^{d-1}\times[0,R]$ and prove that the boundary integrals vanish in the limit $R\rightarrow\infty$. We actually show that
\begin{subequations}
\begin{align}
&\int_{\partial\left([-R,R]^{d-1}\times[0,R]\right)}u\left(z\right)\left(B^*\left(\mathrm{M}z\right)\nabla U\left(z\right)\right)\cdot n(z)dz\stackrel{R\rightarrow\infty}{\longrightarrow}0,\label{CVipp1}\\
&\int_{\partial\left([-R,R]^{d-1}\times[0,R]\right)}\left[B\left(\mathrm{M}z\right)\nabla u\left(z\right)\right]\cdot n(z)U\left(z\right)dz\stackrel{R\rightarrow\infty}{\longrightarrow}0\label{CVipp2}.
\end{align}
\end{subequations}
On the one hand, \eqref{estcroissuLinfty} together with \eqref{estnablaUaux} yields
\begin{multline*}
\left|\int_{\partial\left([-R,R]^{d-1}\times[0,R]\right)}u\left(z\right)\left(B^*\left(\mathrm{M}z\right)\nabla U\left(z\right)\right)\cdot n(z)dz\right|\\
\leq C\int_{\partial\left([-R,R]^{d-1}\times[0,R]\right)\setminus\left([-R,R]^{d-1}\times\{0\}\right)}z_d^\frac{1}{2}\frac{1}{|z|^d}dz\leq C\frac{1}{R^{d-\frac{1}{2}}}R^{d-1}=\frac{C}{R^\frac{1}{2}},
\end{multline*}
which gives \eqref{CVipp1}. On the other hand, it follows from \eqref{nablavtends0}, \eqref{nablawtends0} and \eqref{estUaux} that
\begin{align*}
&\left|\int_{[-R,R]^{d-1}\times\{R\}}\left[B\left(\mathrm{M}z\right)\nabla u\left(z\right)\right]\cdot n(z)U\left(z\right)dz\right|\\
&\leq C\int_{[-R,R]^{d-1}}\left|\nabla u\left(z',R\right)\right|\left|U\left(z',R\right)\right|dz'\\
&\leq C\int_{[-R,R]^{d-1}}\left\|\nabla u\right\|_{L^\infty\left(\mathbb R^{d-1}\times]R,\infty[\right)}\frac{R}{\left(R^2+|z'|^2\right)^\frac{d}{2}}dz'\\
&\leq C\left\|\nabla u\right\|_{L^\infty\left(\mathbb R^{d-1}\times]R,\infty[\right)}\stackrel{R\rightarrow\infty}{\longrightarrow}0
\end{align*}
and that
\begin{align*}
&\left|\int_{\{R\}\times[-R,R]^{d-2}\times[0,R]}\left[B\left(\mathrm{M}z\right)\nabla u\left(z\right)\right]\cdot n(z)U\left(z\right)dz\right|\\
&\leq C\int_{[-R,R]^{d-2}\times[0,R]}\left\|\nabla u\right\|_{L^\infty\left(\mathbb R^{d-1}\times]z_d,\infty[\right)}\frac{z_d}{\left(R^2+|z_2|^2+\ldots\ |z_d|^2\right)^\frac{d}{2}}dz_2\ldots\ dz_d\\
&\leq C\int_{[-1,1]^{d-2}\times[0,1]}\left\|\nabla u\right\|_{L^\infty\left(\mathbb R^{d-1}\times]Rz_d,\infty[\right)}\frac{z_d}{\left(1+|z_2|^2+\ldots\ |z_d|^2\right)^\frac{d}{2}}dz_2\ldots\ dz_d
\end{align*}
tend to $0$ by dominated convergence, which yields \eqref{CVipp2} and terminates the proof of theorem \ref{theov=w}. 

For practical convenience, we have argued that $v=w$. Yet, theorem \ref{theov=w} proves that the variational solution $v_{bl}$ of \eqref{sysbl} equals the Poisson solution $w_{bl}$. This allows to work, for the rest of the paper, with the solution $v_{bl}$ of \eqref{sysbl} satisfying \eqref{nablavbltends0} and \eqref{partial_z_dvbl}, no matter whether this solution is constructed variationally or via Poisson's kernel. Thanks to the bound \eqref{estALinPP}, $v_{bl}$ is seen to be bounded on $\Omega_n$. 
 
\subsection{An homogenization problem}

Studying the tail of $v_{bl}$, i.e. the limit when $y\cdot n\rightarrow\infty$ of $v_{bl}(y)$, boils down to describing the asymptotics of $P\left(y,\tilde{y}\right)$ for $y$ far away from the boundary $\partial\Omega_n$. One of the main focus of our paper is thus to expand $P(y,\tilde{y})$ for $|y-\tilde{y}|\gg 1$, where $y\in\Omega_n$ and $\tilde{y}\in\partial\Omega_n$, i.e. to describe the large scale asymptotics of $P$. Let $y\in\Omega_n$ and $\tilde{y}\in\partial\Omega_n$ and
\begin{equation*}
\varepsilon:=\frac{1}{|y-\tilde{y}|}.
\end{equation*}
If $|y-\tilde{y}|\gg 1$, then $\varepsilon\ll 1$ is a small parameter. Introducing the rescaled variables 
\begin{equation*}
x:=\varepsilon y\in\Omega_n\quad\mbox{and}\quad \tilde{x}:=\varepsilon\tilde{y}\in\Omega_n
\end{equation*}
yields $|x-\tilde{x}|=1$. Such a scaling transforms our initial question of the large scale asymptotic description of $P$ into the study of $\frac{1}{\varepsilon^{d-1}}P\left(\frac{x}{\varepsilon},\frac{\tilde{x}}{\varepsilon}\right)$ for $\varepsilon\rightarrow 0$ and $|x-\tilde{x}|$ close to $1$.
\begin{lem}\label{lemPepsP}
Let $\varepsilon >0$ and call $G^\varepsilon$ (resp. $P^\varepsilon$) the Green (resp. Poisson) kernel associated to the operator $L^\varepsilon=-\nabla\cdot A\left(\frac{x}{\varepsilon}\right)\nabla\cdot$ and the domain $\Omega_n$.\\
Then,
\begin{enumerate}
\item for all $x,\ \tilde{x}\in\Omega_n$,
\begin{equation}\label{scalingGreen}
G^\varepsilon(x,\tilde{x})=\frac{1}{\varepsilon^{d-2}}G\left(\frac{x}{\varepsilon},\frac{\tilde{x}}{\varepsilon}\right);
\end{equation}
\item for all $x\in\Omega_n$, $\tilde{x}\in\partial\Omega_n$,
\begin{equation}\label{scalingPoisson}
P^\varepsilon(x,\tilde{x})=\frac{1}{\varepsilon^{d-1}}P\left(\frac{x}{\varepsilon},\frac{\tilde{x}}{\varepsilon}\right).
\end{equation}
\end{enumerate}
\end{lem}
\begin{proof}
This lemma follows easily from Green's integral representation formula. Let $f\in C^\infty_c\left(\Omega_n\right)$, $u^\varepsilon=u^\varepsilon(x)\in\mathbb R^N$ the solution of 
\begin{equation*}
\left\{
\begin{array}{rll}
-\nabla\cdot A\left(\frac{x}{\varepsilon}\right)\nabla u^\varepsilon&=\frac{1}{\varepsilon^2}f\left(\frac{x}{\varepsilon}\right),& x\cdot n>0\\
u^\varepsilon&=0,& x\cdot n=0
\end{array}
\right.
\end{equation*}
and $u=u(y)\in\mathbb R^N$ the solution of
\begin{equation*}
\left\{
\begin{array}{rll}
-\nabla\cdot A(y)\nabla u&=f,& y\cdot n>0\\
u&=0,& y\cdot n=0
\end{array}
\right..
\end{equation*}
For all $x\in\Omega_n$, 
\begin{multline*}
\int_{\Omega_n}G^\varepsilon(x,\tilde{x})\frac{1}{\varepsilon^2}f\left(\frac{\tilde{x}}{\varepsilon}\right)d\tilde{x}=u^\varepsilon(x)=u\left(\frac{x}{\varepsilon}\right)\\
=\int_{\Omega_n}G\left(\frac{x}{\varepsilon},\tilde{y}\right)f(\tilde{y})d\tilde{y}=\int_{\Omega_n}\frac{1}{\varepsilon^d}G\left(\frac{x}{\varepsilon},\frac{\tilde{x}}{\varepsilon}\right)f\left(\frac{\tilde{x}}{\varepsilon}\right)d\tilde{x},
\end{multline*}
which yields \eqref{scalingGreen}; \eqref{scalingPoisson} easily follows from analogous ideas.
\end{proof}

It immediately follows from the definition of $G^\varepsilon$, that for all $\tilde{x}\in\Omega_n$, $G^\varepsilon\left(\cdot,\tilde{x}\right)$ solves the system 
\begin{equation}\label{sysGreeneps}
\left\{
\begin{array}{rll}
-\nabla_x\cdot A\left(\frac{x}{\varepsilon}\right)\nabla_x G^\varepsilon\left(x,\tilde{x}\right)&=\delta(x-\tilde{x})\Idd_N, & x\cdot n>0\\
G^\varepsilon(x,\tilde{x})&=0,& x\cdot n=0
\end{array}
\right. .
\end{equation}
The Poisson kernel $P^\varepsilon$ satisfies: for all $i,\ j\in\{1,\ldots N\}$, for all $x\in\Omega_n$, $\tilde{x}\in\partial\Omega_n$,
\begin{subequations}\label{exprPeps}
\begin{align}
P^\varepsilon_{ij}(x,\tilde{x})&=-A^{\alpha\beta}_{kj}\left(\frac{\tilde{x}}{\varepsilon}\right)\partial_{\tilde{x}_\alpha}G^\varepsilon_{ik}(x,\tilde{x})n_\beta\\
&=-\left(A^*\right)^{\beta\alpha}_{jk}\left(\frac{\tilde{x}}{\varepsilon}\right)\partial_{\tilde{x}_\alpha}G^{*,\varepsilon}_{ki}(\tilde{x},x)n_\beta\\
&=-\left[A^{*,\beta\alpha}\left(\frac{\tilde{x}}{\varepsilon}\right)\partial_{\tilde{x}_\alpha}G^{*,\varepsilon}(\tilde{x},x)n_\beta\right]_{ji}\\
&=-\left[A^*\left(\frac{\tilde{x}}{\varepsilon}\right)\nabla_{\tilde{x}}G^{*,\varepsilon}(\tilde{x},x)\cdot n\right]^T_{ij}.
\end{align}
\end{subequations}
where $G^{*,\varepsilon}$ is the Green kernel associated to the operator $L^{*,\varepsilon}=-\nabla\cdot A^*\left(\frac{x}{\varepsilon}\right)\nabla\cdot$ and the domain $\Omega_n$.

The estimates \eqref{estALinG} and \eqref{estALinP} can be rescaled. In particular, there exists $C>0$ such that for all $\varepsilon>0$, for all $x,\ \tilde{x}\in\overline{\Omega_n}$, $\tilde{x}\neq x$, for all $d\geq 2$,
\begin{subequations}
\begin{alignat}{2}
\left|G^\varepsilon(x,\tilde{x})\right|&\leq C\left(\left|\logg\left|x-\tilde{x}\right|\right|+1\right),&\quad\mbox{if } d=2\quad \mbox{(see \cite{alin} theorem $13$ (ii))},\label{estGepsd=2}\\
\left|G^\varepsilon(x,\tilde{x})\right|&\leq \frac{C}{|x-\tilde{x}|^{d-2}},&\quad\mbox{if } d\geq 3,\label{estGeps}\\
\left|G^\varepsilon(x,\tilde{x})\right|&\leq C\frac{\left(x\cdot n\right)\left(\tilde{x}\cdot n\right)}{|x-\tilde{x}|^d},\label{estGepsdqcq}\\
\left|\nabla_{\tilde{x}}G^\varepsilon(x,\tilde{x})\right|&\leq \frac{C}{|x-\tilde{x}|^{d-1}},\label{estnablaGeps}
\end{alignat}
and for all $x\in\Omega_n$, $\tilde{x}\in\partial\Omega_n$,
\begin{align}
\left|P^\varepsilon(x,\tilde{x})\right|&\leq C\frac{x\cdot n}{|x-\tilde{x}|^d},\label{estPeps}\\
\left|\nabla_xP^\varepsilon(x,\tilde{x})\right|&\leq C\left(\frac{1}{|x-\tilde{x}|^d}+\frac{x\cdot n}{|x-\tilde{x}|^{d+1}}\right)\label{estnablaPeps}.
\end{align}
\end{subequations}

According to lemma \ref{lemPepsP}, we now deal with highly oscillating kernels $G^\varepsilon$ and $P^\varepsilon$, instead of looking at $G$ and $P$. Hence the asymptotic description of $G$ (resp. $P$) at large scales, is replaced by an homogenization problem on $G^\varepsilon$ (resp. $P^\varepsilon$). This fact, which has been stressed by Avellaneda and Lin (see \cite{alinLp} corollary on p. $903$), is the cornerstone of our method.

\selectlanguage{english}

\pagestyle{plain}

\section{Homogenization in the half-space}
\label{secdual}

This section is concerned with the asymptotics, for small $\varepsilon$, of $u^\varepsilon=u^\varepsilon(x)\in\mathbb R^N$ solving
\begin{equation}\label{sysueps}
\left\{
\begin{array}{rll}
-\nabla\cdot A\left(\frac{x}{\varepsilon}\right)\nabla u^\varepsilon&=f,& x\cdot n>0\\
u^\varepsilon&=0,& x\cdot n=0
\end{array}
\right. .
\end{equation}
The study of this dual homogenization problem is preparatory to the expansion of the Green and Poisson kernels in the next section. In the introduction, we defined the interior and boundary layer correctors to $u^\varepsilon$ up to the order $\varepsilon$ and reviewed some error estimates in the case of a bounded domain $\Omega$. The purpose of this section is to show similar estimates, yet in the unbounded domain $\Omega_n$.  

Let $f\in C^\infty_c\left(\overline{\Omega_n}\right)$; note that the support of $f$ may intersect the boundary. We define $u^0=u^0(x)\in\mathbb R^N$ as the solution of
\begin{equation}\label{sysu0}
\left\{
\begin{array}{rll}
-\nabla\cdot A^0\nabla u^0&=f,& x\cdot n>0\\
u^0&=0,& x\cdot n=0
\end{array}
\right. ,
\end{equation}
$u^1=u^1(x,y)\in\mathbb R^N$ by $u^1(x,y):=\chi^\alpha(y)\partial_{x_\alpha}u^0(x)$, for all $x\in\Omega_n$ and $y\in\mathbb T^d$, and $u^{1,\varepsilon}_{bl}=u^{1,\varepsilon}_{bl}(x)\in\mathbb R^N$ as the Poisson solution to
\begin{equation*}
\left\{
\begin{array}{rll}
-\nabla\cdot A\left(\frac{x}{\varepsilon}\right)\nabla u^{1,\varepsilon}_{bl}&=0,& x\cdot n>0\\
u^{1,\varepsilon}_{bl}&=-u^1\left(x,\frac{x}{\varepsilon}\right)=-\chi^\alpha\left(\frac{x}{\varepsilon}\right)\partial_{x_\alpha}u^0(x),& x\cdot n=0
\end{array}
\right. .
\end{equation*}
The variational solution $u^\varepsilon$ (resp. $u^0$) coincides with the solution given by Green's integral formula. Besides, $u^\varepsilon$, $u^0$, as well as $u^{1,\varepsilon}_{bl}$ belong to $C^\infty\left(\overline{\Omega_n}\right)$, thanks to the smoothness of the boundary $\partial\Omega_n$, using local regularity estimates from \cite{adn2}. The rest of this section is devoted to the proof of the following proposition:
\begin{prop}\label{propestuepspropa}
Let $r^{1,\varepsilon}_{bl}:=u^\varepsilon(x)-u^0(x)-\varepsilon\chi^\alpha\left(\frac{x}{\varepsilon}\right)\partial_{x_\alpha}u^0-\varepsilon u^{1,\varepsilon}_{bl}(x)$, and $\delta>0$.\\
Then, there exists a constant $C>0$, such that for all $\varepsilon>0$,
\begin{subequations}\label{estuepsprop}
\begin{align}
\left\|r^{1,\varepsilon}_{bl}\right\|_{L^\infty\left(\Omega_n\right)}&\leq C\varepsilon\left\|f\right\|_{W^{1,\frac{d}{2}+\delta}\left(\Omega_n\right)},\label{estuepspropa}\\
\left\|u^{1,\varepsilon}_{bl}\right\|_{L^\infty\left(\Omega_n\right)}&\leq C\left\|f\right\|_{W^{1,\frac{d}{2}+\delta}\left(\Omega_n\right)},\label{estuepspropb}\\
\left\|u^\varepsilon-u^0\right\|_{L^\infty\left(\Omega_n\right)}&\leq C\varepsilon\left\|f\right\|_{W^{1,\frac{d}{2}+\delta}\left(\Omega_n\right)}\label{estuepspropc}.
\end{align}
\end{subequations}
\end{prop}

The proof of \eqref{estuepspropa} relies on estimates in $L^\infty$ of $r^{1,\varepsilon}_{bl}$ solution of the following elliptic system
\begin{equation*}
\left\{
\begin{array}{rll}
-\nabla\cdot A\left(\frac{x}{\varepsilon}\right)\nabla r^{1,\varepsilon}_{bl}&=f^\varepsilon,& x\cdot n>0\\
r^{1,\varepsilon}_{bl}&=0,& x\cdot n=0
\end{array}
\right. ,
\end{equation*}
where $f^\varepsilon=f^\varepsilon(x)\in\mathbb R^N$ is given by
\begin{equation*}
f^\varepsilon:=f+\nabla\cdot A\left(\frac{x}{\varepsilon}\right)\nabla u^0+\varepsilon\nabla\cdot A\left(\frac{x}{\varepsilon}\right)\nabla\left(\chi^\alpha\left(\frac{x}{\varepsilon}\right)\partial_{x_\alpha} u^0(x)\right).
\end{equation*}
The idea is to use the integral representation provided by Green's formula in order to bound $r^{1,\varepsilon}_{bl}$. However, as such, $f^\varepsilon$ does not seem to be of order $\varepsilon$. Let us thus work on $f^\varepsilon$ to make its structure more explicit. Expanding $f^\varepsilon$, we get for all $x\in\Omega_n$
\begin{align}\label{dvptfeps}
f^\varepsilon(x)&=\frac{1}{\varepsilon}\left[\nabla_y\cdot A(y)\nabla_x u^0+\nabla_y\cdot A(y)\nabla_y u^1\right]\left(x,\frac{x}{\varepsilon}\right)\\
&\quad+\left[f+\nabla_x\cdot A(y)\nabla_x u^0+\nabla_x\cdot A(y)\nabla_y u^1+\nabla_y\cdot A(y)\nabla_x u^1\right]\left(x,\frac{x}{\varepsilon}\right)\nonumber\\
&\quad\quad+\varepsilon\left[\nabla_x\cdot A(y)\nabla_x u^1\right]\left(x,\frac{x}{\varepsilon}\right)\nonumber.
\end{align}
We aim to get rid of terms of order $\varepsilon^{-1}$ and $\varepsilon^0$ in \eqref{dvptfeps}. The term of order $\varepsilon^{-1}$ easily cancels thanks to \eqref{syschi}:
\begin{equation*}
\nabla_y\cdot A(y)\nabla_x u^0+\nabla_y\cdot A(y)\nabla_y u^1=\left[\partial_{y_\alpha}\left(A^{\alpha\gamma}(y)\right)+\partial_{y_\alpha}\left(A^{\alpha\beta}(y)\partial_{y_\beta}\chi^\gamma(y)\right)\right]\partial_{x_\gamma}u^0=0.
\end{equation*}
For the term of order $\varepsilon^0$ in the r.h.s. of \eqref{dvptfeps}
\begin{multline}\label{zerothtermfeps}
\left[f+\nabla_x\cdot A(y)\nabla_x u^0+\nabla_x\cdot A(y)\nabla_yu^1+\nabla_y\cdot A(y)\nabla_x u^1\right]\left(x,\frac{x}{\varepsilon}\right)\\
=\left[f+\nabla_x\cdot v+\nabla_y\cdot A(y)\nabla_xu^1\right]\left(x,\frac{x}{\varepsilon}\right),
\end{multline}
where $v=v(x,y):=A(y)\nabla_xu^0+A(y)\nabla_yu^1=A(y)\left[\nabla_xu^0+\nabla_yu^1\right]$, we notice that
\begin{subequations}
\begin{align*}
\nabla_y\cdot\left(v-A^0\nabla u^0\right)=0\qquad\mbox{and}\qquad\fint_{\mathbb T^d}\left(v-A^0\nabla u^0\right)=0.
\end{align*}
\end{subequations}
As $v-A^0\nabla u^0$ factors into $\Phi\nabla u^0$, with
\begin{align*}
\Phi=\Phi(y):=A(y)\left(\Idd+\nabla_y\chi(y)\right)-\int_{\mathbb T^d}A(\tilde{y})\left(\Idd+\nabla_y\chi(\tilde{y})\right)d\tilde{y}\in\mathbb R^{N^2\times d^2},
\end{align*}
one can take advantage of the classical lemma:
\begin{lem}\label{lemexistencePsi}
There exists a smooth function $\Psi=\Psi(y)\in\mathbb R^{N^2\times d^3}$ such that for all $y\in\mathbb T^d$,
\begin{equation*}
\Phi(y)=\nabla_y\cdot\Psi(y).
\end{equation*}
\end{lem}

Via lemma \ref{lemexistencePsi}, \eqref{zerothtermfeps} becomes
\begin{align*}
f+\nabla_x\cdot v+\nabla_y\cdot A(y)\nabla_xu^1&=\nabla_x\cdot\left[v-A^0\nabla_xu^0\right]+\nabla_y\cdot A(y)\nabla_xu^1\\
&=\nabla_x\cdot\left(\nabla_y\cdot\left(\Psi(y)\right)\nabla u^0\right)+\nabla_y\cdot A(y)\nabla_xu^1\\
&=\nabla_y\cdot\left(\Psi(y)\nabla^2 u^0\right)+\Psi(y)\nabla^3u^0.
\end{align*}
Subsequently, for all $x\in\Omega_n$,
\begin{align*}
f^\varepsilon(x)
&=\left[\nabla_y\cdot\left(\Psi(y)\nabla^2 u^0\right)\right]\left(x,\frac{x}{\varepsilon}\right)+\varepsilon\nabla\cdot\left(A\left(\frac{x}{\varepsilon}\right)\chi\left(\frac{x}{\varepsilon}\right)\nabla^2 u^0\right)\\
&=\left[\nabla_y\cdot\left(\Psi(y)\nabla^2 u^0\right)\right]\left(x,\frac{x}{\varepsilon}\right)\\
&\quad+\varepsilon\left[\nabla_x\cdot\left(\Psi(y)\nabla^2 u^0\right)\right]\left(x,\frac{x}{\varepsilon}\right)-\varepsilon\Psi\left(\frac{x}{\varepsilon}\right)\nabla^3u^0\\
&\quad\quad+\varepsilon\nabla\cdot\left(A\left(\frac{x}{\varepsilon}\right)\chi\left(\frac{x}{\varepsilon}\right)\nabla^2u^0\right)\\
&=\varepsilon\nabla\cdot\left(\Psi\left(\frac{x}{\varepsilon}\right)\nabla^2u^0\right)+\varepsilon\nabla\cdot\left(A\left(\frac{x}{\varepsilon}\right)\chi\left(\frac{x}{\varepsilon}\right)\nabla^2u^0\right)-\varepsilon\Psi\left(\frac{x}{\varepsilon}\right)\nabla^3u^0.
\end{align*}
Hence $f^\varepsilon=\varepsilon\nabla\cdot h^\varepsilon+\varepsilon g^\varepsilon$, with
\begin{align*}
h^\varepsilon=h^\varepsilon(x)&:=\Psi\left(\frac{x}{\varepsilon}\right)\nabla^2u^0+A\left(\frac{x}{\varepsilon}\right)\chi\left(\frac{x}{\varepsilon}\right)\nabla^2u^0,\\
g^\varepsilon=g^\varepsilon(x)&:=-\Psi\left(\frac{x}{\varepsilon}\right)\nabla^3u^0.
\end{align*}
The next lemma contains estimates on $h^\varepsilon$ and $g^\varepsilon$ for large $x$.
\begin{lem}\label{lemhepsgeps}
There is a constant $C>0$, such that for all $x$ sufficiently large, for all $\varepsilon>0$,
\begin{subequations}
\begin{align}
\left|h^\varepsilon(x)\right|\leq C\frac{1}{\left|x\right|^d}\label{estheps},\\
\left|g^\varepsilon(x)\right|\leq C\frac{1}{\left|x\right|^{d+1}}\label{estgeps}.
\end{align}
\end{subequations}
\end{lem}

\begin{proof}
Let $\Lambda\in\mathbb N^d$, $2\leq|\Lambda|\leq 3$, and $R>0$ such that the support of $f$ is included in $B(0,R)$. It follows from \eqref{estG0lambda} that for $x$ large enough,
\begin{align*}
\left|\partial^\Lambda_xu^0(x)\right|&\leq \int_{\Omega_n}\left|\partial^\Lambda_{x}G^0(x,\tilde{x})\right||f(\tilde{x})|d\tilde{x}\\
&\leq C\int_{B(0,R)}\frac{1}{\left|x-\tilde{x}\right|^{d-2+|\Lambda|}}d\tilde{x}\\
&\leq C\int_{B(0,R)}\frac{1}{\left(|x|-R\right)^{d-2+|\Lambda|}}d\tilde{x}\\
&\leq \frac{C}{\left(|x|-R\right)^{d-2+|\Lambda|}}\stackrel{\infty}{=}O\left(\frac{1}{|x|^{d-2+|\Lambda|}}\right),
\end{align*}
which ends the proof.
\end{proof}

These preliminaries being done, we now turn to the estimation of $r^{1,\varepsilon}_{bl}$. Let $x\in\Omega_n$ be fixed. Green's formula yields
\begin{equation}\label{r1epsblreprGreen}
r^{1,\varepsilon}_{bl}(x)=\varepsilon\int_{\Omega_n}G^\varepsilon(x,\tilde{x})\left(\nabla\cdot h^\varepsilon+g^\varepsilon\right)(\tilde{x})d\tilde{x}.
\end{equation}
We concentrate on each term of the r.h.s. of \eqref{r1epsblreprGreen} separately. The strategy in both cases is to split the integral in two parts:
\begin{enumerate}
\item for $\tilde{x}$ close to $x$, one relies on Young inequalities to bound this part in $L^\infty$;
\item for $\tilde{x}$ far from $x$, one uses \eqref{estheps} and \eqref{estgeps} to show that this part can be made arbitrarily small uniformly in $x$.
\end{enumerate}
Let $R>0$ and assume for the moment $d\geq 3$. An integration by parts 
\begin{equation*}
\int_{\Omega_n}G^\varepsilon(x,\tilde{x})\nabla\cdot h^\varepsilon(\tilde{x})d\tilde{x}=-\int_{\Omega_n}\left(\nabla_{\tilde{x}}G^\varepsilon\right)(x,\tilde{x})h^\varepsilon(\tilde{x})d\tilde{x},
\end{equation*}
together with \eqref{estGeps} and \eqref{estnablaGeps} gives
\begin{align*}
&\bigl|r^{1,\varepsilon}_{bl}(x)\bigr|\leq\varepsilon\int_{\Omega_n}\frac{C}{|x-\tilde{x}|^{d-1}}\left|h^\varepsilon(\tilde{x})\right|d\tilde{x}+\varepsilon\int_{\Omega_n}\frac{C}{|x-\tilde{x}|^{d-2}}\left|g^\varepsilon(\tilde{x})\right|d\tilde{x}\nonumber\\
&\leq\varepsilon\int_{\mathbb R^d}\frac{C}{|x-\tilde{x}|^{d-1}}1_{B(0,R)}(x-\tilde{x})\bigl|\widetilde{h}^\varepsilon(\tilde{x})\bigr|d\tilde{x}+\varepsilon\int_{\mathbb R^d}\frac{C}{|x-\tilde{x}|^{d-1}}1_{B(0,R)^c}(x-\tilde{x})\bigl|\widetilde{h}^\varepsilon(\tilde{x})\bigr|d\tilde{x}\nonumber\\
&\quad+\varepsilon\int_{\mathbb R^d}\frac{C}{|x-\tilde{x}|^{d-2}}1_{B(0,R)}(x-\tilde{x})\left|\widetilde{g}^\varepsilon(\tilde{x})\right|d\tilde{x}+\varepsilon\int_{\mathbb R^d}\frac{C}{|x-\tilde{x}|^{d-2}}1_{B(0,R)^c}(x-\tilde{x})\left|\widetilde{g}^\varepsilon(\tilde{x})\right|d\tilde{x},
\end{align*}
where $\widetilde{h}^\varepsilon$ (resp. $\widetilde{g}^\varepsilon$) is the extension of $h^\varepsilon$ (resp. $g^\varepsilon$) to $\mathbb R^d$ by $0$ outside of $\Omega_n$. Let us first concentrate on the terms involving $\widetilde{h}^\varepsilon$. First, it simply follows from lemma \ref{lemhepsgeps} that $\bigl|\widetilde{h}^\varepsilon(\tilde{x})\bigr|$ is $O\left(\frac{1}{|\tilde{x}|^{d}}\right)$ in a neighbourhood of $\infty$. One can find $p,\ p'\geq 1$ such that
\begin{equation*}
\frac{1}{p}+\frac{1}{p'}=1,\qquad p>\frac{d}{d-1},\qquad\mbox{and}\quad p'>1.
\end{equation*}
Therefore, 
\begin{equation*}
\left\|\frac{1}{|\tilde{x}|^{d-1}}1_{B(0,R)^c}(\tilde{x})\right\|_{L^p\left(\mathbb R^d\right)}\stackrel{R\rightarrow\infty}{\longrightarrow}0
\end{equation*}
and $\widetilde{h}^\varepsilon$ is bounded uniformly in $\varepsilon$ in $L^{p'}\left(\mathbb R^d\right)$. Thanks to Young's inequality,
\begin{multline}\label{majh^epsinfty}
\left|\int_{\mathbb R^d}\frac{C}{|x-\tilde{x}|^{d-1}}1_{B(0,R)^c}(x-\tilde{x})\bigl|\widetilde{h}^\varepsilon(\tilde{x})\bigr|d\tilde{x}\right|\\
\leq
\left\|\frac{1}{|\tilde{x}|^{d-1}}1_{B(0,R)^c}(\tilde{x})\right\|_{L^p\left(\mathbb R^d\right)}\bigl\|\widetilde{h}^\varepsilon\bigr\|_{L^{p'}\left(\mathbb R^{d}\right)}\stackrel{R\rightarrow\infty}{\longrightarrow}0
\end{multline}
uniformly in $x\in\Omega_n$. The r.h.s. of \eqref{majh^epsinfty} can be made less than $\|f\|_{W^{1,\frac{d}{2}+\delta}\left(\Omega_n\right)}$ for an $R>0$ large enough. We now carry out the analysis of the integral on $|x-\tilde{x}|<R$. An adequate choice of the exponents in Young's inequality leads to \eqref{estuepspropa}. Indeed, for all $1\leq q<\frac{d}{d-1}$, $\frac{1}{|\tilde{x}|^{d-1}}1_{B(0,R)}\in L^q\left(\mathbb R^d\right).$ From the condition $1=\frac{1}{q}+\frac{1}{q'}$ on Young exponents, one deduces $q'>d$. Yet, for all $\delta>0$, for all $q':=d+\delta$, thanks to elliptic regularity and Sobolev's injection
\begin{multline*}
\bigl\|\widetilde{h}^\varepsilon\bigr\|_{L^{q'}\left(\Omega_n\right)}\leq C\left\|\nabla^2u^0\right\|_{L^{q'}\left(\Omega_n\right)}\leq C\left\|\nabla^2u^0\right\|_{W^{1,\frac{d}{2}+\delta}\left(\Omega_n\right)}\\\leq C\left\|u^0\right\|_{W^{3,\frac{d}{2}+\delta}\left(\Omega_n\right)}\leq C\left\|f\right\|_{W^{1,\frac{d}{2}+\delta}\left(\Omega_n\right)}.
\end{multline*}
Young's inequality finally gives
\begin{equation*}
\int_{\mathbb R^d}\frac{1}{|x-\tilde{x}|^{d-1}}1_{B(0,R)}(x-\tilde{x})\bigl|\widetilde{h}^\varepsilon(\tilde{x})\bigr|d\tilde{x}\leq C\left\|f\right\|_{W^{1,\frac{d}{2}+\delta}(\Omega_n)}.
\end{equation*}
The reasoning for $\widetilde{g}^\varepsilon$ is almost the same, except for the exponents in Young's inequalities which have to be adapted. 

Let us briefly indicate how to treat the case $d=2$. Each term can be estimated as above, except the ones involving $g^\varepsilon$. As before, we split the integral:
\begin{multline*}
\int_{\Omega_n}G^\varepsilon(x,\tilde{x})g^\varepsilon(\tilde{x})d\tilde{x}=\int_{\Omega_n}G^\varepsilon(x,\tilde{x})1_{D(0,R)}(x-\tilde{x})g^\varepsilon(\tilde{x})d\tilde{x}\\
+\int_{\Omega_n}G^\varepsilon(x,\tilde{x})1_{D(0,R)^c}(x-\tilde{x})g^\varepsilon(\tilde{x})d\tilde{x}.
\end{multline*}
For $x$ close to $\tilde{x}$ we bound $G^\varepsilon$ by \eqref{estGepsd=2}
\begin{equation*}
\int_{\Omega_n}G^\varepsilon(x,\tilde{x})1_{D(0,R)}(x-\tilde{x})g^\varepsilon(\tilde{x})d\tilde{x}\leq C\int_{\mathbb R^d}\left(\left|\logg\left|x-\tilde{x}\right|\right|+1\right)1_{B(0,R)}(x-\tilde{x})\widetilde{g}^\varepsilon(\tilde{x})d\tilde{x},
\end{equation*}
and for $|x-\tilde{x}|>R$ we have recourse to \eqref{estGepsdqcq} instead of \eqref{estGeps}
\begin{multline*}
\left|\int_{\Omega_n}G^\varepsilon(x,\tilde{x})1_{D(0,R)^c}(x-\tilde{x})g^\varepsilon(\tilde{x})d\tilde{x}\right|\leq C\int_{\mathbb R^d}\frac{|x\cdot n||\tilde{x}\cdot n|}{|x-\tilde{x}|^2}1_{D(0,R)^c}(x-\tilde{x})\left|\widetilde{g}^\varepsilon(\tilde{x})\right|d\tilde{x}\\
\leq C\int_{\mathbb R^d}\left(\frac{|(x-\tilde{x})\cdot n||\tilde{x}\cdot n|}{|x-\tilde{x}|^2}+\frac{|\tilde{x}\cdot n|^2}{|x-\tilde{x}|^2}\right)1_{B(0,R)^c}(x-\tilde{x})\left|\widetilde{g}^\varepsilon(\tilde{x})\right|d\tilde{x}\\
\leq C\int_{\mathbb R^d}\frac{1}{|x-\tilde{x}|}1_{B(0,R)^c}(x-\tilde{x})|\tilde{x}|\left|\widetilde{g}^\varepsilon(\tilde{x})\right|d\tilde{x}\\
+\int_{\mathbb R^d}\frac{1}{|x-\tilde{x}|^2}1_{B(0,R)^c}(x-\tilde{x})|\tilde{x}|^2\left|\widetilde{g}^\varepsilon(\tilde{x})\right|d\tilde{x}
\end{multline*}
and estimate these terms, uniformly in $x$ and $\varepsilon$, via Young's inequality. The bound \eqref{estuepspropa} on $r^{1,\varepsilon}_{bl}$ is shown.

As elliptic regularity and Sobolev injections imply
\begin{equation*}
\left\|\chi^\alpha\left(\frac{\cdot}{\varepsilon}\right)\partial_{x_\alpha}u^0\right\|_{L^{\infty}\left(\Omega_n\right)}\leq C\left\|\nabla u^0\right\|_{W^{2,\frac{d}{2}+\delta}\left(\Omega_n\right)}\leq C\left\|f\right\|_{W^{1,\frac{d}{2}+\delta}\left(\Omega_n\right)},
\end{equation*}
it remains to establish \eqref{estuepspropb} in order to prove \eqref{estuepspropc}. Let $x\in\Omega_n$. Poisson's representation formula for $u^{1,\varepsilon}_{bl}$ yields
\begin{equation*}
u^{1,\varepsilon}_{bl}(x)=-\int_{\partial\Omega_n}P^\varepsilon(x,\tilde{x})\chi^\alpha\left(\frac{\tilde{x}}{\varepsilon}\right)\partial_{x_\alpha}u^0(\tilde{x})d\tilde{x}.
\end{equation*}
From the bound \eqref{estPeps} on $P^\varepsilon$, one gets
\begin{align*}
\bigl|u^{1,\varepsilon}_{bl}(x)\bigr|&\leq\int_{\partial\Omega_n}\left|P^\varepsilon(x,\tilde{x})\right|\left|\chi^\alpha\left(\frac{\tilde{x}}{\varepsilon}\right)\partial_{x_\alpha}u^0(\tilde{x})\right|d\tilde{x}\\
&\leq C\int_{\partial\Omega_n}\frac{x\cdot n}{|x-\tilde{x}|^d}\left|\partial_{x_\alpha}u^0(\tilde{x})\right|d\tilde{x}\\
&\leq C\int_{\mathbb R^{d-1}}\frac{\hat{x}_d}{\left(\hat{x}_d^2+|\hat{x}'-\tilde{x}'|^2\right)^{\frac{d}{2}}}\left|\partial_{x_\alpha}u^0\left(\mathrm{M}\left(\tilde{x}',0\right)\right)\right|d\tilde{x}'\\
&\leq C\int_{\mathbb R^{d-1}}\frac{1}{\left(1+\left|\tilde{x}'\right|^2\right)^{\frac{d}{2}}}\left|\partial_{x_\alpha}u^0\left(\mathrm{M}\left(\hat{x}'-\hat{x}_d\tilde{x}',0\right)\right)\right|d\tilde{x}'\\ 
&\leq C\left\|\nabla u^0\right\|_{L^\infty\left(\Omega_n\right)}\int_{\mathbb R^{d-1}}\frac{1}{\left(1+|\tilde{x}'|^2\right)^\frac{d}{2}}d\tilde{x}'\\
&\leq C\left\|\nabla u^0\right\|_{W^{2,\frac{d}{2}+\delta}\left(\Omega_n\right)}\leq C\left\|f\right\|_{W^{1,\frac{d}{2}+\delta}\left(\Omega_n\right)},
\end{align*}
with $\hat{x}:=\mathrm{M}^Tx$. This establishes \eqref{estuepspropa} and proposition \ref{propestuepspropa}.

\selectlanguage{english}

\pagestyle{plain}

\section{Asymptotic expansion of Poisson's kernel}
\label{secexp}

We intend to get the asymptotics of $P^\varepsilon=P^\varepsilon(x,\tilde{x})$, defined by \eqref{exprPeps}, for $\tilde{x}$ in a neighbourhood of the boundary $\partial\Omega_n$ and $|x-\tilde{x}|$ close to $1$. Our method is in three steps:
\begin{enumerate}
\item homogenization of $u^\varepsilon$ solution of \eqref{sysueps} for $f\in C^\infty_c\left(\overline{\Omega_n}\right)$;
\item expansion of $G^\varepsilon$, thanks to a duality argument;
\item expansion of $P^\varepsilon$ via \eqref{exprPeps} and the expansion for $G^\varepsilon$.
\end{enumerate}
The first point has been the purpose of section \ref{secdual}. We now turn to the second point.

\subsection{Back to Green's kernel}
\label{secback}

\begin{prop}\label{propasGeps0}
For all $0<\kappa<\frac{1}{d}$, there exists $C_\kappa>0$, such that for all $\varepsilon>0$, for all $x,\ \tilde{x}\in\overline{\Omega_n}$, $\frac{1}{4}\leq|x-\tilde{x}|\leq 4$ implies
\begin{equation*}
\left|G^\varepsilon(x,\tilde{x})-G^0(x,\tilde{x})\right|\leq C_\kappa\varepsilon^\kappa.
\end{equation*}
\end{prop}

We start from \eqref{estuepspropc} and proceed by a duality method, mimicked from \cite{alinLp}, to prove the estimate on the kernels. The key is as usual Green's representation formula. Let $x\in\Omega_n$ and $\varepsilon>0$ be fixed for the rest of the proof. We look at
\begin{equation*}
\sigma_{\varepsilon,x}:=\sup_{\substack{\tilde{x}\in\Omega_n\\ \frac{1}{4}\leq|x-\tilde{x}|\leq 4}}\left|G^\varepsilon(x,\tilde{x})-G^0(x,\tilde{x})\right|.
\end{equation*}
The l.u.b. $\sigma_{\varepsilon,x}$ is reached for at least one $\tilde{x}_{\varepsilon,x}$, which may be on the boundary $\partial\Omega_n$. From \eqref{estnablaGeps}, one obtains the existence of $C_1>0$, independent from $\varepsilon$ and $x$, such that for all $\tilde{x}\in\Omega_n$, $\frac{1}{5}\leq|x-\tilde{x}|\leq 5$,
\begin{equation*}
\left|\nabla_{\tilde{x}}G^\varepsilon(x,\tilde{x})\right|+\left|\nabla_{\tilde{x}}G^0(x,\tilde{x})\right|\leq C_1.
\end{equation*}
Let $\rho_{\varepsilon,x}:=\frac{\sigma_{\varepsilon,x}}{2N^2C_1}$. One can always increase $C_1$, so that $\rho_{\varepsilon,x}<1$ and
\begin{equation*}
B\left(\tilde{x}_{\varepsilon,x},\rho_{\varepsilon,x}\right)\subset B(x,5)\setminus \overline{B}\left(x,\frac{1}{5}\right).
\end{equation*}
Then:
\begin{lem}\label{lemdualGeps}
There exists $i,\ j\in\{1,\ldots\ N\}$ such that for all $\tilde{x}\in D\left(\tilde{x}_{\varepsilon,x},\rho_{\varepsilon,x}\right)$,
\begin{equation*}
\left|G^{\varepsilon}_{ij}(x,\tilde{x})-G^0_{ij}(x,\tilde{x})\right|\geq \frac{\sigma_{\varepsilon,x}}{2N^2}.
\end{equation*}
\end{lem}
\begin{proof}
From $\left|G^\varepsilon(x,\tilde{x}_{\varepsilon,x})-G^0(x,\tilde{x}_{\varepsilon,x})\right|=\sigma_{\varepsilon,x}$ it comes the existence of $i,\ j\in\{1,\ldots\ N\}$ such that
\begin{equation}\label{ineqtiroirsGepsij}
\left|G^\varepsilon_{ij}(x,\tilde{x}_{\varepsilon,x})-G^0_{ij}(x,\tilde{x}_{\varepsilon,x})\right|\geq\frac{\sigma_{\varepsilon,x}}{N^2}.
\end{equation}
The integers $i,\ j$ are now fixed such as \eqref{ineqtiroirsGepsij} holds. For all $\tilde{x}\in D\left(\tilde{x}_{\varepsilon,x},\rho_{\varepsilon,x}\right)$, either $\left|G^{\varepsilon}_{ij}(x,\tilde{x})-G^0_{ij}(x,\tilde{x})\right|\geq \frac{\sigma_{\varepsilon,x}}{2N^2}$, or $\left|G^{\varepsilon}_{ij}(x,\tilde{x})-G^0_{ij}(x,\tilde{x})\right|<\frac{\sigma_{\varepsilon,x}}{2N^2}$. Assume the latter. Then, 
\begin{align*}
&0<\frac{\sigma_{\varepsilon,x}}{N^2}-\left|G^{\varepsilon}_{ij}(x,\tilde{x})-G^0_{ij}(x,\tilde{x})\right|\\
&\leq \left|G^{\varepsilon}_{ij}(x,\tilde{x}_{\varepsilon,x})-G^0_{ij}(x,\tilde{x}_{\varepsilon,x})\right|-\left|G^{\varepsilon}_{ij}(x,\tilde{x})-G^0_{ij}(x,\tilde{x})\right|\\
&\leq \left|G^{\varepsilon}_{ij}(x,\tilde{x}_{\varepsilon,x})-G^{\varepsilon}_{ij}(x,\tilde{x})\right|+\left|G^0_{ij}(x,\tilde{x}_{\varepsilon,x})-G^0_{ij}(x,\tilde{x})\right|\\
&\leq \left[\left\|\nabla_{\tilde{x}}G^\varepsilon\right\|_{L^\infty\left(\frac{1}{5}\leq|x-\tilde{x}|\leq 5\right)}+\left\|\nabla_{\tilde{x}}G^0\right\|_{L^\infty\left(\frac{1}{5}\leq|x-\tilde{x}|\leq 5\right)}\right]\left|\tilde{x}_{\varepsilon,x}-\tilde{x}\right|\\
&\leq C_1\rho_{\varepsilon,x}=\frac{\sigma_{\varepsilon,x}}{2N^2},
\end{align*}
and finally
\begin{equation*}
\left|G^{\varepsilon}_{ij}(x,\tilde{x})-G^0_{ij}(x,\tilde{x})\right|\geq \frac{\sigma_{\varepsilon,x}}{2N^2}.\qedhere
\end{equation*}
\end{proof}

Take $\varphi\in C^\infty_c\left(B(\tilde{x}_{\varepsilon,x},1)\right)$, with $0\leq\varphi\leq 1$ and $\varphi\equiv 1$ on $B\left(\tilde{x}_{\varepsilon,x},\frac{1}{2}\right)$. Note that the support of $\varphi$ may intersect the boundary $\partial\Omega_n$. For $\rho>0$, we define $\varphi_\rho$ by
$\varphi_\rho:=\varphi\left(\frac{\cdot}{\rho}\right)\in C^\infty_c\left(B\left(\tilde{x}_{\varepsilon,x},\rho\right)\right)$; we have:
\begin{equation*}
\left\|\varphi_\rho\right\|_{L^\infty\left(\Omega_n\right)}\leq C\quad\mbox{and}\qquad \left\|\nabla\varphi_\rho\right\|_{L^\infty\left(\Omega_n\right)}\leq\frac{C}{\rho}.
\end{equation*}
For $\rho=\rho_{\varepsilon,x}$, the constants above do not depend on $\varepsilon$ or $x$.

Let $i,\ j$ the integers given by lemma \ref{lemdualGeps}. The intermediate value theorem implies that $G^\varepsilon_{ij}(x,\tilde{x})-G^0_{ij}(x,\tilde{x})$ has a constant sign on $D\left(\tilde{x}_{\varepsilon,x},\rho_{\varepsilon,x}\right)$. Up to the change of $f$ in $-f$ in what follows, one can always assume that $G^\varepsilon_{ij}(x,\tilde{x})-G^0_{ij}(x,\tilde{x})\geq 0$, which automatically yields
\begin{equation*}
G^\varepsilon_{ij}(x,\tilde{x})-G^0_{ij}(x,\tilde{x})\geq \frac{\sigma_{\varepsilon,x}}{2N^2}
\end{equation*}
for all $\tilde{x}\in D\left(\tilde{x}_{\varepsilon,x},\rho_{\varepsilon,x}\right)$. We now carry out the duality argument. For this purpose, take $f=f(\tilde{x}):=\varphi_{\rho_{\varepsilon,x}}(\tilde{x})e_j$, where $e_j$ is the $j^{th}$ vector of the canonical basis of $\mathbb R^N$ and denote by $u^\varepsilon=u^\varepsilon(\tilde{x})\in\mathbb R^N$ and $u^0=u^0(\tilde{x})\in\mathbb R^N$ the solutions of \eqref{sysueps} and \eqref{sysu0} with r.h.s. $f$. We remind that $f$, $u^\varepsilon$ as well as $u^0$ depend on $x$. Thanks to Green's representation formula,
\begin{align*}
\left[u^\varepsilon(x)-u^0(x)\right]_i&=\int_{\Omega_n}\left[\left(G^\varepsilon(x,\tilde{x})-G^0(x,\tilde{x})\right)f(\tilde{x})\right]_id\tilde{x}\\
&=\int_{D\left(\tilde{x}_{\varepsilon,x},\rho_{\varepsilon,x}\right)}\left[G^\varepsilon_{ij}(x,\tilde{x})-G^0_{ij}(x,\tilde{x})\right]\varphi_{\rho_{\varepsilon,x}}(\tilde{x})d\tilde{x}\\
&\geq \int_{D\left(\tilde{x}_{\varepsilon,x},\frac{\rho_{\varepsilon,x}}{2}\right)}G^\varepsilon_{ij}(x,\tilde{x})-G^0_{ij}(x,\tilde{x})d\tilde{x}\\
&\geq \int_{D\left(\tilde{x}_{\varepsilon,x},\frac{\rho_{\varepsilon,x}}{2}\right)}\frac{\sigma_{\varepsilon,x}}{2N^2}d\tilde{x}\geq C\rho_{\varepsilon,x}^{d+1}.
\end{align*}
Yet, we know from \eqref{estuepspropc}, an estimate of $u^\varepsilon-u^0$:
\begin{multline*}
\left\|u^\varepsilon-u^0\right\|_{L^\infty\left(\Omega_n\right)}\leq C_\delta\varepsilon\left\|f\right\|_{W^{1,\frac{d}{2}+\delta}\left(\Omega_n\right)}\\
=C_\delta\varepsilon\left[\left\|\varphi_{\rho_{\varepsilon,x}}\right\|_{L^{\frac{d}{2}+\delta}\left(\Omega_n\right)}+\left\|\nabla\varphi_{\rho_{\varepsilon,x}}\right\|_{L^{\frac{d}{2}+\delta}\left(\Omega_n\right)}\right]\leq C_\delta\varepsilon\rho_{\varepsilon,x}^{\frac{d-2\delta}{d+2\delta}}.
\end{multline*}
Putting together these bounds, we get
\begin{multline*}
C\rho_{\varepsilon,x}^{d+1}\leq \left|\left[u^\varepsilon(x)-u^0(x)\right]_i\right|\leq\left|u^\varepsilon(x)-u^0(x)\right|\leq \left\|u^\varepsilon-u^0\right\|_{L^\infty\left(\Omega_n\right)}\leq C_\delta\varepsilon\rho_{\varepsilon,x}^{\frac{d-2\delta}{d+2\delta}},
\end{multline*}
which summarizes in 
\begin{equation*}
\varepsilon\geq C_\delta\rho_{\varepsilon,x}^{-\frac{d-2\delta}{d+2\delta}+d+1}=C_\delta\rho_{\varepsilon,x}^{d+\frac{4\delta}{d+2\delta}}
\end{equation*}
for every $\delta>0$, with a constant $C_\delta$ independent from $\varepsilon$ and $x$. The inequalities 
\begin{equation*}
\sigma_{\varepsilon,x}\leq C_\delta\rho_{\varepsilon,x}\leq C_\delta\varepsilon^{\frac{1}{d+\frac{4\delta}{d+2\delta}}}
\end{equation*}
contain the asymptotic expansion of $G^\varepsilon$ at zeroth order of proposition \ref{propasGeps0}. One can follow the same reasoning as above, changing $A$ in $A^*$, to obtain for all $0<\kappa<\frac{1}{d}$, for all $x,\ \tilde{x}\in\Omega_n$, $\frac{1}{4}\leq|x-\tilde{x}|\leq 4$, 
\begin{equation}\label{estG*eps-Gp0}
\left|G^{*,\varepsilon}(x,\tilde{x})-G^{*,0}(x,\tilde{x})\right|\leq C_\kappa\varepsilon^\kappa.
\end{equation}

\subsection{Homogenization of Poisson's kernel}
\label{sechompoisson}
Let $0<\varepsilon<1$ and $x\in\Omega_n$ be fixed. Assume that $x$ is close to the boundary, say $x\cdot n<4$ to fix the ideas. According to \eqref{sysGreeneps}, $G^{*,\varepsilon}(\cdot,x)$ satisfies
\begin{equation*}
\left\{\begin{array}{rll}
-\nabla_{\tilde{x}}\cdot A^*\left(\frac{\tilde{x}}{\varepsilon}\right)\nabla_{\tilde{x}} G^{*,\varepsilon}(\tilde{x},x)&=0,& \tilde{x}\in D(x,4)\setminus \overline{D}\left(x,\frac{1}{4}\right)\\
G^{*,\varepsilon}(\tilde{x},x)&=0,& \tilde{x}\in \Gamma(x,4)\setminus \overline{\Gamma}\left(x,\frac{1}{4}\right)
\end{array}\right.
\end{equation*}
This leads to the idea that one can apply theorem \ref{theolocalboundarygradientueps} to an expansion of $G^{*,\varepsilon}$ for which a local estimate in $L^\infty$ is known. Doing so, one has to handle carefully with the trace on $\Gamma(x,4)\setminus\overline{\Gamma}\left(x,\frac{1}{4}\right)$. Take for example $Z^{*,\varepsilon,x}=Z^{*,\varepsilon,x}(\tilde{x})\in M_N(\mathbb R)$ defined for all $\tilde{x}\in D(x,4)\setminus\overline{D}\left(x,\frac{1}{4}\right)$ by
\begin{multline*}
Z^{*,\varepsilon,x}(\tilde{x}):=G^{*,\varepsilon}(\tilde{x},x)-G^{*,0}(\tilde{x},x)-\varepsilon\chi^{*,\alpha}\left(\frac{\tilde{x}}{\varepsilon}\right)\partial_{\tilde{x}_\alpha}G^{*,0}(\tilde{x},x)\\
-\varepsilon^2\Gamma^{*,\alpha\beta}\left(\frac{\tilde{x}}{\varepsilon}\right)\partial_{\tilde{x}_\alpha}\partial_{\tilde{x}_\beta}G^{*,0}(\tilde{x},x).
\end{multline*}
Then, for all $0<\nu<1$,
\begin{align*}
&\left\|Z^{*,\varepsilon,x}\right\|_{C^{1,\nu}\left(\Gamma(x,4)\setminus\overline{\Gamma}\left(x,\frac{1}{4}\right)\right)}\\
&=\!\varepsilon\left\|\chi^{*,\alpha}\left(\frac{\tilde{x}}{\varepsilon}\right)\partial_{\tilde{x}_\alpha}G^{*,0}(\tilde{x},x)+\varepsilon\Gamma^{*,\alpha\beta}\left(\frac{\tilde{x}}{\varepsilon}\right)\partial_{\tilde{x}_\alpha}\partial_{\tilde{x}_\beta}G^{*,0}(\tilde{x},x)\right\|_{C^{1,\nu}\left(\Gamma(x,4)\setminus\overline{\Gamma}\left(x,\frac{1}{4}\right)\right)}\!\!=O\left(\varepsilon^{-\nu}\right)
\end{align*}
which worsens our estimates. One way of getting around this difficulty is again to introduce a boundary layer term in the expansion \label{defzepsx}. This term has to cancel the trace on the boundary due to the first-order term.

For all $\gamma\in\{1,\ldots\ d\}$, let $G^{*,1,\gamma}_{bl}=G^{*,1,\gamma}_{bl}(\tilde{y})\in M_N(\mathbb R)$ be the solution of
\begin{equation*}
\left\{
\begin{array}{rll}
-\nabla_{\tilde{y}}\cdot A^*(\tilde{y})\nabla_{\tilde{y}}G^{*,1,\gamma}_{bl}&=0,&\tilde{y}\in\Omega_n\\
G^{*,1,\gamma}_{bl}&=-\chi^{*,\gamma}(\tilde{y}),&\tilde{y}\in\partial\Omega_n
\end{array}
\right. 
\end{equation*}
in the sense of theorem \ref{theodecaydiv} or proposition \ref{solreprgreen}, both being identical according to theorem \ref{theov=w}. We then define the boundary layer $G^{*,1,\varepsilon}_{bl}=G^{*,1,\varepsilon}_{bl}(\tilde{x},x)\in M_N(\mathbb R)$ at first order in $\varepsilon$ by
\begin{equation*}
G^{*,1,\varepsilon}_{bl}(\tilde{x},x)=G^{*,1,\gamma}_{bl}\left(\frac{\tilde{x}}{\varepsilon}\right)\partial_{\tilde{x}_\gamma}G^{*,0}(\tilde{x},x)
\end{equation*}
for all $\tilde{x}\in\Omega_n$, $\tilde{x}\neq x$. Instead of $Z^{*,\varepsilon,x}$ one focuses now on $Z^{*,\varepsilon,x}_{bl}=Z^{*,\varepsilon,x}_{bl}(\tilde{x})\in M_N(\mathbb R)$ such that
\begin{multline*}
Z^{*,\varepsilon,x}_{bl}(\tilde{x}):=G^{*,\varepsilon}(\tilde{x},x)-G^{*,0}(\tilde{x},x)-\varepsilon\chi^{*,\alpha}\left(\frac{\tilde{x}}{\varepsilon}\right)\partial_{\tilde{x}_\alpha}G^{*,0}(\tilde{x},x)-\varepsilon G^{*,1,\varepsilon}_{bl}(\tilde{x},x)\\
-\varepsilon^2\Gamma^{*,\alpha\beta}\left(\frac{\tilde{x}}{\varepsilon}\right)\partial_{\tilde{x}_\alpha}\partial_{\tilde{x}_\beta}G^{*,0}(\tilde{x},x)-\varepsilon^2\chi^{*,\alpha}\left(\frac{\tilde{x}}{\varepsilon}\right)G^{*,1,\beta}_{bl}\left(\frac{\tilde{x}}{\varepsilon}\right)\partial_{\tilde{x}_\alpha}\partial_{\tilde{x}_\beta}G^{*,0}(\tilde{x},x)
\end{multline*}
for all $\tilde{x}\in D(x,4)\setminus\overline{D}\left(x,\frac{1}{4}\right)$. The method is now similar to the one, which led to the estimate on $r^{1,\varepsilon}_{bl}$: $Z^{*,\varepsilon,x}_{bl}$ satisfies
\begin{equation*}
\left\{
\begin{array}{rll}
-\nabla\cdot A^*\left(\frac{\tilde{x}}{\varepsilon}\right)\nabla Z^{*,\varepsilon,x}_{bl}&=F^\varepsilon+F^\varepsilon_{bl},&\tilde{x}\in D(x,4)\setminus\overline{D}\left(x,\frac{1}{4}\right)\\
Z^{*,\varepsilon,x}_{bl}&=-\varepsilon^2\Gamma^{*,\alpha\beta}\left(\frac{\tilde{x}}{\varepsilon}\right)\partial_{\tilde{x}_\alpha}\partial_{\tilde{x}_\beta}G^{*,0}(\tilde{x},x),&\tilde{x}\in \Gamma(x,4)\setminus\overline{\Gamma}\left(x,\frac{1}{4}\right)
\end{array}
\right. 
\end{equation*}
where 
\begin{multline*}
F^\varepsilon:=\nabla_{\tilde{x}}\cdot A^*\left(\frac{\tilde{x}}{\varepsilon}\right)\nabla_{\tilde{x}} G^{*,0}(\tilde{x},x)+\varepsilon\nabla_{\tilde{x}}\cdot A^*\left(\frac{\tilde{x}}{\varepsilon}\right)\nabla_{\tilde{x}}\left(\chi^{*,\alpha}\left(\frac{\tilde{x}}{\varepsilon}\right)\partial_{\tilde{x}_\alpha}G^{*,0}(\tilde{x},x)\right)\\
+\varepsilon^2\nabla_{\tilde{x}}\cdot A^*\left(\frac{\tilde{x}}{\varepsilon}\right)\nabla_{\tilde{x}}\left(\Gamma^{*,\alpha\beta}\left(\frac{\tilde{x}}{\varepsilon}\right)\partial_{\tilde{x}_\alpha}\partial_{\tilde{x}_\beta}G^{*,0}(\tilde{x},x)\right)
\end{multline*}
and 
\begin{multline}\label{Fepsbl}
F^\varepsilon_{bl}:=\varepsilon\nabla_{\tilde{x}}\cdot A^*\left(\frac{\tilde{x}}{\varepsilon}\right)\nabla_{\tilde{x}}G^{*,1,\varepsilon}_{bl}(\tilde{x})\\+\varepsilon^2\nabla_{\tilde{x}}\cdot A^*\left(\frac{\tilde{x}}{\varepsilon}\right)\nabla_{\tilde{x}}\left(\chi^{*,\alpha}\left(\frac{\tilde{x}}{\varepsilon}\right)G^{*,1,\beta}_{bl}\left(\frac{\tilde{x}}{\varepsilon}\right)\partial_{\tilde{x}_\alpha}\partial_{\tilde{x}_\beta}G^{*,0}(\tilde{x},x)\right). 
\end{multline}
Applying \eqref{boundarygradest}, one obtains
\begin{multline}\label{boundarygradestZ*bl}
\left\|\nabla Z^{*,\varepsilon,x}_{bl}\right\|_{L^\infty\left(D\left(x,3\right)\setminus \overline{D}\left(x,\frac{1}{3}\right)\right)}\leq C\Biggl[\left\|Z^{*,\varepsilon,x}_{bl}\right\|_{L^\infty\left(D\left(x,4\right)\setminus \overline{D}\left(x,\frac{1}{4}\right)\right)}\Biggr.\\
+\left\|F^\varepsilon\right\|_{L^{d+\delta}\left(D\left(x,4\right)\setminus \overline{D}\left(x,\frac{1}{4}\right)\right)}+\left\|F^\varepsilon_{bl}\right\|_{L^{d+\delta}\left(D\left(x,4\right)\setminus \overline{D}\left(x,\frac{1}{4}\right)\right)}\\
\Biggl.+\left\|\varepsilon^2\Gamma^{*,\alpha\beta}\left(\frac{\tilde{x}}{\varepsilon}\right)\partial_{\tilde{x}_\alpha}\partial_{\tilde{x}_\beta}G^{*,0}(\tilde{x},x)\right\|_{C^{1,\nu}\left(\Gamma\left(x,4\right)\setminus \overline{\Gamma}\left(x,\frac{1}{4}\right)\right)}\Biggr].
\end{multline}

From \eqref{estG*eps-Gp0} and the boundedness of $G^{*,1,\gamma}_{bl}$, one immediately obtains 
\begin{equation}\label{estZ*blLinfty}
\left\|Z^{*,\varepsilon,x}_{bl}\right\|_{L^\infty\left(D\left(x,4\right)\setminus \overline{D}\left(x,\frac{1}{4}\right)\right)}=O\left(\varepsilon^\frac{1}{d}\right).
\end{equation}
Now, for $0<\nu<1$,
\begin{multline}\label{estZ*blboundary}
\left\|Z^{*,\varepsilon,x}_{bl}\right\|_{C^{1,\nu}\left(\Gamma\left(x,4\right)\setminus \overline{\Gamma}\left(x,\frac{1}{4}\right)\right)}\\
=\varepsilon^2\left\|\Gamma^{*,\alpha\beta}\left(\frac{\tilde{x}}{\varepsilon}\right)\partial_{\tilde{x}_\alpha}\partial_{\tilde{x}_\beta}G^{*,0}(\tilde{x},x)\right\|_{C^{1,\nu}\left(\Gamma\left(x,4\right)\setminus \overline{\Gamma}\left(x,\frac{1}{4}\right)\right)}=O\left(\varepsilon^{1-\nu}\right).
\end{multline}
Expanding $F^\varepsilon$ in powers of $\varepsilon$, one notices that the terms of order $-1$ and $0$ in $\varepsilon$ cancel and that for all $\tilde{x}\in D(x,4)\setminus\overline{D}\left(x,\frac{1}{4}\right)$
\begin{align*}
F^\varepsilon(\tilde{x})&=\varepsilon\Bigl[\nabla_{\tilde{x}}\cdot A^*(\tilde{y})\nabla_{\tilde{x}}\left(\chi^{*,\alpha}(\tilde{y})\partial_{\tilde{x}_\alpha}G^{*,0}(\tilde{x},x)\right)\Bigr.\\
&\qquad+\nabla_{\tilde{x}}\cdot A^*(\tilde{y})\nabla_{\tilde{y}}\left(\Gamma^{*,\alpha\beta}(\tilde{y})\partial_{\tilde{x}_\alpha}\partial_{\tilde{x}_\beta}G^{*,0}(\tilde{x},x)\right)\\
&\qquad\qquad+\Bigl.\nabla_{\tilde{y}}\cdot A^*(\tilde{y})\nabla_{\tilde{x}}\left(\Gamma^{*,\alpha\beta}(\tilde{y})\partial_{\tilde{x}_\alpha}\partial_{\tilde{x}_\beta}G^{*,0}(\tilde{x},x)\right)\Bigr]\left(\tilde{x},\frac{\tilde{x}}{\varepsilon}\right)\\
&\qquad\qquad\qquad+\varepsilon^2\left[\nabla_{\tilde{x}}\cdot A^*(\tilde{y})\nabla_{\tilde{x}}\left(\Gamma^{*,\alpha\beta}(\tilde{y})\partial_{\tilde{x}_\alpha}\partial_{\tilde{x}_\beta}G^{*,0}(\tilde{x},x)\right)\right]\left(\tilde{x},\frac{\tilde{x}}{\varepsilon}\right).
\end{align*}
This demonstrates that $F^\varepsilon$ is $O\left(\varepsilon\right)$ in $L^{d+\delta}\left(D\left(x,4\right)\setminus \overline{D}\left(x,\frac{1}{4}\right)\right)$. 

The source term $F^\varepsilon_{bl}$ due to the boundary layer deserves more attention. The expansion of $F^\varepsilon_{bl}$ in powers of $\varepsilon$ is quite heavy. For the first term in the r.h.s. of \eqref{Fepsbl} we get
\begin{subequations}
\begin{align}
&\varepsilon\nabla_{\tilde{x}}\cdot A^*\left(\frac{\tilde{x}}{\varepsilon}\right)\nabla_{\tilde{x}}\left(G^{*,1,\gamma}_{bl}\left(\frac{\tilde{x}}{\varepsilon}\right)\partial_{\tilde{x}_\gamma}G^{*,0}(\tilde{x},x)\right)\nonumber\\
&=\partial_{\tilde{y}_\alpha}A^{*,\alpha\beta}\left(\frac{\tilde{x}}{\varepsilon}\right)G^{*,1,\gamma}_{bl}\left(\frac{\tilde{x}}{\varepsilon}\right)\partial_{\tilde{x}_\beta}\partial_{\tilde{x}_\gamma}G^{*,0}(\tilde{x},x)\label{Fepsbl12}\\
&\qquad+A^{*,\alpha\beta}\left(\frac{\tilde{x}}{\varepsilon}\right)\partial_{\tilde{y}_\beta}G^{*,1,\gamma}_{bl}\left(\frac{\tilde{x}}{\varepsilon}\right)\partial_{\tilde{x}_\alpha}\partial_{\tilde{x}_\gamma}G^{*,0}(\tilde{x},x)\label{Fepsbl11}\\
&\qquad\qquad+A^{*,\alpha\beta}\left(\frac{\tilde{x}}{\varepsilon}\right)\partial_{\tilde{y}_\alpha}G^{*,1,\gamma}_{bl}\left(\frac{\tilde{x}}{\varepsilon}\right)\partial_{\tilde{x}_\beta}\partial_{\tilde{x}_\gamma}G^{*,0}(\tilde{x},x)\label{Fepsbl13}\\
&\qquad\qquad\qquad+\varepsilon A^{*,\alpha,\beta}\left(\frac{\tilde{x}}{\varepsilon}\right)G^{*,1,\gamma}_{bl}\left(\frac{\tilde{x}}{\varepsilon}\right)\partial_{\tilde{x}_\alpha}\partial_{\tilde{x}_\beta}\partial_{\tilde{x}_\gamma}G^{*,0}(\tilde{x},x)\label{Fepsbl14}.
\end{align}
\end{subequations}
We write the second term in the r.h.s. of \eqref{Fepsbl} as a sum of three terms
\begin{multline*}
\varepsilon^2\nabla_{\tilde{x}}\cdot A^*\left(\frac{\tilde{x}}{\varepsilon}\right)\nabla_{\tilde{x}}\left(\chi^{*,\gamma}\left(\frac{\tilde{x}}{\varepsilon}\right)G^{*,1,\eta}_{bl}\left(\frac{\tilde{x}}{\varepsilon}\right)\partial_{\tilde{x}_\gamma}\partial_{\tilde{x}_\eta}G^{*,0}(\tilde{x},x)\right)\\
=T^0\left(\tilde{x},x\right)+\varepsilon T^1\left(\tilde{x},x\right)+\varepsilon^2 T^2\left(\tilde{x},x\right),
\end{multline*}
with at order $\varepsilon^0$
\begin{subequations}
\begin{align}
T^0\left(\tilde{x},x\right)&:=\left[\partial_{\tilde{y}_\alpha}\left(A^{*,\alpha\beta}(\tilde{y})\partial_{\tilde{y}_\beta}\chi^{*,\gamma}(\tilde{y})\right)G^{*,1,\eta}_{bl}(\tilde{y})\right]\left(\frac{\tilde{x}}{\varepsilon}\right)\partial_{\tilde{x}_\gamma}\partial_{\tilde{x}_\eta} G^{*,0}(\tilde{x},x)\label{Fepsbl21}\\
&\qquad+\left[A^{*,\alpha\beta}(\tilde{y})\partial_{\tilde{y}_\beta}\chi^{*,\gamma}(\tilde{y})\partial_{\tilde{y}_\alpha}G^{*,1,\eta}_{bl}(\tilde{y})\right]\left(\frac{\tilde{x}}{\varepsilon}\right)\partial_{\tilde{x}_\gamma}\partial_{\tilde{x}_\eta} G^{*,0}(\tilde{x},x)\label{Fepsbl22}\\
&\qquad\qquad+\left[\partial_{\tilde{y}_\alpha}\left(A^{*,\alpha\beta}(\tilde{y})\chi^{*,\gamma}(\tilde{y})\right)\partial_{\tilde{y}_\beta}G^{*,1,\eta}_{bl}(\tilde{y})\right]\left(\frac{\tilde{x}}{\varepsilon}\right)\partial_{\tilde{x}_\gamma}\partial_{\tilde{x}_\eta} G^{*,0}(\tilde{x},x)\label{Fepsbl23}\\
&\qquad\qquad\qquad+\left[A^{*,\alpha\beta}(\tilde{y})\chi^{*,\gamma}(\tilde{y})\partial_{\tilde{y}_\alpha}\partial_{\tilde{y}_\beta}G^{*,1,\eta}_{bl}(\tilde{y})\right]\left(\frac{\tilde{x}}{\varepsilon}\right)\partial_{\tilde{x}_\gamma}\partial_{\tilde{x}_\eta} G^{*,0}(\tilde{x},x)\label{Fepsbl24},
\end{align}
at order $\varepsilon^1$
\begin{align}
T^1\left(\tilde{x},x\right)&:=\left[A^{*,\alpha\beta}(\tilde{y})\partial_{\tilde{y}_\beta}\chi^{*,\gamma}(\tilde{y})G^{*,1,\eta}_{bl}(\tilde{y})\right]\left(\frac{\tilde{x}}{\varepsilon}\right)\partial_{\tilde{x}_\alpha}\partial_{\tilde{x}_\gamma}\partial_{\tilde{x}_\eta} G^{*,0}(\tilde{x},x)\label{Fepsbl25}\\
&\qquad+\left[A^{*,\alpha\beta}(\tilde{y})\chi^{*,\gamma}(\tilde{y})\partial_{\tilde{y}_\beta}G^{*,1,\eta}_{bl}(\tilde{y})\right]\left(\frac{\tilde{x}}{\varepsilon}\right)\partial_{\tilde{x}_\alpha}\partial_{\tilde{x}_\gamma}\partial_{\tilde{x}_\eta} G^{*,0}(\tilde{x},x)\label{Fepsbl26}\\
&\qquad\qquad+\left[\partial_{\tilde{y}_\alpha}\left(A^{*,\alpha\beta}(\tilde{y})\chi^{*,\gamma}(\tilde{y})\right)G^{*,1,\eta}_{bl}(\tilde{y})\right]\left(\frac{\tilde{x}}{\varepsilon}\right)\partial_{\tilde{x}_\beta}\partial_{\tilde{x}_\gamma}\partial_{\tilde{x}_\eta} G^{*,0}(\tilde{x},x)\label{Fepsbl27}\\
&\qquad\qquad\qquad+\left[A^{*,\alpha\beta}(\tilde{y})\chi^{*,\gamma}(\tilde{y})\partial_{\tilde{y}_\alpha}G^{*,1,\eta}_{bl}(\tilde{y})\right]\left(\frac{\tilde{x}}{\varepsilon}\right)\partial_{\tilde{x}_\beta}\partial_{\tilde{x}_\gamma}\partial_{\tilde{x}_\eta} G^{*,0}(\tilde{x},x),\label{Fepsbl28}
\end{align}
and at order $\varepsilon^2$
\begin{align}
T^2\left(\tilde{x},x\right)&:=\left[A^{*,\alpha\beta}(\tilde{y})\chi^{*,\gamma}(\tilde{y})G^{*,1,\eta}_{bl}(\tilde{y})\right]\left(\frac{\tilde{x}}{\varepsilon}\right)\partial_{\tilde{x}_\alpha}\partial_{\tilde{x}_\beta}\partial_{\tilde{x}_\gamma}\partial_{\tilde{x}_\eta} G^{*,0}(\tilde{x},x)\label{Fepsbl29}.
\end{align}
\end{subequations}
We intend to show that all these terms are $O\left(\varepsilon^\kappa\right)$, with $\kappa>0$, in $L^{d+\delta}\left(D\left(x,4\right)\setminus \overline{D}\left(x,\frac{1}{4}\right)\right)$. This seems tricky for some terms in the expansion above. Indeed, \eqref{Fepsbl12} and \eqref{Fepsbl21} are of order $O(1)$, as we do not know more than $G^{*,1,\eta}_{bl}\in L^\infty(\Omega_n)$. However, the sum of \eqref{Fepsbl12} and \eqref{Fepsbl21} cancels:
\begin{multline*}
\left[\partial_{\tilde{y}_\alpha}\left(A^{*,\alpha\beta}(\tilde{y})\partial_{\tilde{y}_\beta}\chi^{*,\gamma}(\tilde{y})\right)G^{*,1,\eta}_{bl}(\tilde{y})\right]\left(\frac{\tilde{x}}{\varepsilon}\right)\partial_{\tilde{x}_\gamma}\partial_{\tilde{x}_\eta} G^{*,0}(\tilde{x},x)\\
=\left[\partial_{\tilde{y}_\alpha}\left(A^{*,\alpha\eta}(\tilde{y})\partial_{\tilde{y}_\eta}\chi^{*,\beta}(\tilde{y})\right)G^{*,1,\eta}_{bl}(\tilde{y})\right]\left(\frac{\tilde{x}}{\varepsilon}\right)\partial_{\tilde{x}_\beta}\partial_{\tilde{x}_\gamma} G^{*,0}(\tilde{x},x)\\
=-\partial_{\tilde{y}_\alpha}A^{*,\alpha\beta}\left(\frac{\tilde{x}}{\varepsilon}\right)G^{*,1,\eta}_{bl}\left(\frac{\tilde{x}}{\varepsilon}\right)\partial_{\tilde{x}_\beta}\partial_{\tilde{x}_\gamma} G^{*,0}(\tilde{x},x).
\end{multline*}
For the remaining terms, we use either \eqref{ineqintnormalderivees}, if at least one derivative of $G_{bl}^{*,1,\eta}$ is involved, or \eqref{estcroissuLinfty} if not. Therefore we proceed in the same manner for \eqref{Fepsbl11}, \eqref{Fepsbl13}, \eqref{Fepsbl22}, \eqref{Fepsbl23}, \eqref{Fepsbl24}, \eqref{Fepsbl26}, \eqref{Fepsbl28} on the one hand, and \eqref{Fepsbl14}, \eqref{Fepsbl25}, \eqref{Fepsbl27}, \eqref{Fepsbl29} on the other hand. Let us estimate \eqref{Fepsbl11} in $L^{d+\delta}\left(D\left(x,4\right)\setminus \overline{D}\left(x,\frac{1}{4}\right)\right)$: by \eqref{ineqintnormalderivees} with $k'=1$,
\begin{align*}
&\left\|A^{*,\alpha\beta}\left(\frac{\tilde{x}}{\varepsilon}\right)\partial_{\tilde{y}_\beta}G^{*,1,\gamma}_{bl}\left(\frac{\tilde{x}}{\varepsilon}\right)\partial_{\tilde{x}_\alpha}\partial_{\tilde{x}_\gamma}G^{*,0}(\tilde{x},x)\right\|_{L^{1}\left(D\left(x,4\right)\setminus \overline{D}\left(x,\frac{1}{4}\right)\right)}\\
&\quad\leq\left(\int_{D\left(x,4\right)\setminus \overline{D}\left(x,\frac{1}{4}\right)}\left|\nabla_{\tilde{y}}G^{*,1,\gamma}_{bl}\left(\frac{\tilde{x}}{\varepsilon}\right)\right|d\tilde{x}\right)\\
&\quad\qquad\leq C\left(\int_0^8\sup_{(\tilde{z}_1,\ldots\tilde{z}_{d-1})\in\mathbb R^{d-1}}\left|\nabla_{\tilde{y}}G^{*,1,\gamma}_{bl}\left(\mathrm{M}\left(\tilde{z}_1,\ldots\tilde{z}_{d-1},\frac{t}{\varepsilon}\right)\right)\right|dt\right)\\
&\quad\qquad\qquad\leq C\left(\int_0^8\sup_{(\tilde{z}_1,\ldots\tilde{z}_{d-1})\in\mathbb R^{d-1}}\left|\nabla_{\tilde{y}}G^{*,1,\gamma}_{bl}\left(\mathrm{M}\left(\tilde{z}_1,\ldots\tilde{z}_{d-1},\frac{t}{\varepsilon}\right)\right)\right|^2dt\right)^\frac{1}{2}\\
&\quad\qquad\qquad\qquad\leq C\varepsilon^\frac{1}{2}\left(\int_0^\infty\sup_{(\tilde{z}_1,\ldots\tilde{z}_{d-1})\in\mathbb R^{d-1}}\left|\nabla_{\tilde{y}}G^{*,1,\gamma}_{bl}\left(\mathrm{M}\left(\tilde{z}_1,\ldots\tilde{z}_{d-1},t\right)\right)\right|^2dt\right)^\frac{1}{2}\leq C\varepsilon^\frac{1}{2}
\end{align*}
and 
\begin{equation*}
\left\|A^{*,\alpha\beta}\left(\frac{\tilde{x}}{\varepsilon}\right)\partial_{\tilde{y}_\beta}G^{*,1,\gamma}_{bl}\left(\frac{\tilde{x}}{\varepsilon}\right)\partial_{\tilde{x}_\alpha}\partial_{\tilde{x}_\gamma}G^{*,0}(\tilde{x},x)\right\|_{L^{\infty}\left(D\left(x,4\right)\setminus \overline{D}\left(x,\frac{1}{4}\right)\right)}=O(1)
\end{equation*}
from which we get, by interpolation,
\begin{equation*}
\left\|A^{*,\alpha\beta}\left(\frac{\tilde{x}}{\varepsilon}\right)\partial_{\tilde{y}_\beta}G^{*,1,\gamma}_{bl}\left(\frac{\tilde{x}}{\varepsilon}\right)\partial_{\tilde{x}_\alpha}\partial_{\tilde{x}_\gamma}G^{*,0}(\tilde{x},x)\right\|_{L^{d+\delta}\left(D\left(x,4\right)\setminus \overline{D}\left(x,\frac{1}{4}\right)\right)}=O\left(\varepsilon^\frac{1}{2(d+\delta)}\right).
\end{equation*}
We have gained a positive power of $\varepsilon$ for all corresponding terms listed above. In the other cases, the use of \eqref{estcroissuLinfty} deteriorates the exponent of $\varepsilon$ in the estimate, though it remains positive. Let us bound \eqref{Fepsbl14} in $L^{d+\delta}\left(D\left(x,4\right)\setminus \overline{D}\left(x,\frac{1}{4}\right)\right)$:
\begin{align*}
&\left\|\varepsilon A^{*,\alpha,\beta}\left(\frac{\tilde{x}}{\varepsilon}\right)G^{*,1,\eta}_{bl}\left(\frac{\tilde{x}}{\varepsilon}\right)\partial_{\tilde{x}_\alpha}\partial_{\tilde{x}_\beta}\partial_{\tilde{x}_\eta}G^{*,0}(\tilde{x},x)\right\|_{L^{d+\delta}\left(D\left(x,4\right)\setminus \overline{D}\left(x,\frac{1}{4}\right)\right)}\\
&\quad\leq C\varepsilon\left(\int_{D\left(x,4\right)\setminus \overline{D}\left(x,\frac{1}{4}\right)}\left|G^{*,1,\gamma}_{bl}\left(\frac{\tilde{x}}{\varepsilon}\right)\right|^{d+\delta}d\tilde{x}\right)^\frac{1}{d+\delta}\leq C\varepsilon\left(\int_0^8\sqrt{\frac{t}{\varepsilon}}^{d+\delta}dt\right)^\frac{1}{d+\delta}\leq C\varepsilon^\frac{1}{2}.
\end{align*}
Hence we have shown
\begin{equation*}
\left\|F^\varepsilon_{bl}\right\|_{L^{d+\delta}\left(D(x,4)\setminus\overline{D}\left(x,\frac{1}{4}\right)\right)}=O\left(\varepsilon^\frac{1}{2(d+\delta)}\right).
\end{equation*}
This bound, with \eqref{estZ*blLinfty}, \eqref{estZ*blboundary}, the estimate on $F^\varepsilon$ and \eqref{boundarygradestZ*bl}, boils down to 
\begin{equation}\label{estnablaZ*blinfty}
\left\|\nabla Z^{*,\varepsilon,x}_{bl}\right\|_{L^\infty\left(D\left(x,3\right)\setminus \overline{D}\left(x,\frac{1}{3}\right)\right)}\leq C\varepsilon^\kappa
\end{equation}
with a positive, albeit small, $\kappa:=\min\left(1-\nu,\frac{1}{2(d+\delta)}\right)$. One can always take $\nu$ sufficiently small so that $1-\nu>\frac{1}{2(d+\delta)}$. As $\nabla Z^{*,\varepsilon,x}_{bl}$ is $C^\infty$ up to the boun\-dary $\partial\Omega_n$, it follows from \eqref{exprPeps} and \eqref{estnablaZ*blinfty} that one can expand $P^\varepsilon(x,\tilde{x})$. For all $\tilde{x}\in D\left(x,3\right)\setminus \overline{D}\left(x,\frac{1}{3}\right)$,
\begin{equation}\label{dvptPepsT}
\left|P^\varepsilon(x,\tilde{x})-P^0\left(x,\tilde{x},\frac{\tilde{x}}{\varepsilon}\right)-\varepsilon P^1\left(x,\tilde{x},\frac{\tilde{x}}{\varepsilon}\right)-\varepsilon^2P^2\left(x,\tilde{x},\frac{\tilde{x}}{\varepsilon}\right)\right|\leq C\varepsilon^\kappa,
\end{equation}
with at order $\varepsilon^0$
\begin{equation*}
\left[P^0(x,\tilde{x},\tilde{y})\right]^T:=\left[A^{*,\alpha\beta}(\tilde{y})+A^{*,\alpha\gamma}(\tilde{y})\left(\partial_{\tilde{y}_\gamma}\chi^{*,\beta}(\tilde{y})+\partial_{\tilde{y}_\gamma}G^{*,1,\beta}_{bl}(\tilde{y})\right)\right]\partial_{\tilde{x}_\beta}G^{*,0}(\tilde{x},x)n_\alpha,
\end{equation*}
at order $\varepsilon^1$
\begin{multline*}
\left[P^1(x,\tilde{x},\tilde{y})\right]^T:=\left[A^{*,\alpha\beta}(\tilde{y})\left(\chi^{*,\gamma}(\tilde{y})+G^{*,1,\gamma}_{bl}(\tilde{y})\right)\right.\\+\left.A^{*,\alpha\eta}(\tilde{y})\partial_{\tilde{y}_\eta}\Gamma^{*,\gamma\beta}(\tilde{y})+A^{*,\alpha\eta}(\tilde{y})\partial_{\tilde{y_\eta}}\left(\chi^{*,\gamma}(\tilde{y})G^{*,1,\beta}_{bl}(\tilde{y})\right)(\tilde{y})\right]\partial_{\tilde{x}_\beta}\partial_{\tilde{x}_\gamma}G^{*,0}(\tilde{x},x)n_\alpha,
\end{multline*}
and at order $\varepsilon^2$
\begin{equation*}
\left[P^2(x,\tilde{x},\tilde{y})\right]^T:=A^{*,\alpha\beta}(\tilde{y})\left[\Gamma^{*,\gamma\eta}(\tilde{y})+\chi^{*,\gamma}(\tilde{y})G^{*,1,\eta}_{bl}(\tilde{y})\right]\partial_{\tilde{x}_\beta}\partial_{\tilde{x}_\gamma}\partial_{\tilde{x}_\eta}G^{*,0}(\tilde{x},x)n_\alpha.
\end{equation*}
For $y\in\Omega_n$, $\tilde{y}\in\partial\Omega_n$ and $\varepsilon:=\frac{1}{|y-\tilde{y}|}$, $x:=\varepsilon y$ and $\tilde{x}:=\varepsilon\tilde{y}$ are such that $|x-\tilde{x}|=1$. Applying now \eqref{dvptPepsT} with $x$ and $\tilde{x}$ defined like this, and rescaling the estimate in the variables $y,\ \tilde{y}$ using the scaling properties of $P^\varepsilon$ and $G^{*,0}$ (see lemma \ref{lemPepsP}), we finally have:
\begin{theo}\label{theodvptP}
For all $0<\kappa<\frac{1}{2d}$, there exists $C_\kappa>0$, such that for all $y\in\Omega_{n,0}$ and $\tilde{y}\in\partial\Omega_{n,0}$,
\begin{multline}\label{theodvptPeq}
\left|P^T(y,\tilde{y})-A^*(\tilde{y})\nabla_{\tilde{y}}G^{*,0}(\tilde{y},y)\cdot n-A^*(\tilde{y})\nabla_{\tilde{y}}\left(\chi^*(\tilde{y})\cdot\nabla_{\tilde{y}}G^{*,0}(\tilde{y},y)\right)\cdot n\right.\\
-A^*(\tilde{y})\nabla_{\tilde{y}}\left(G^{*,1}_{bl}(\tilde{y})\cdot\nabla_{\tilde{y}}G^{*,0}(\tilde{y},y)\right)\cdot n-A^*(\tilde{y})\nabla_{\tilde{y}}\left(\Gamma^*(\tilde{y})\cdot\nabla^2_{\tilde{y}}G^{*,0}(\tilde{y},y)\right)\cdot n\\
\left.-A^*(\tilde{y})\nabla_{\tilde{y}}\left(\chi^*(\tilde{y})G^{*,1}_{bl}(\tilde{y})\cdot\nabla^2_{\tilde{y}}G^{*,0}(\tilde{y},y)\right)\cdot n\right|\leq\frac{C_\kappa}{|y-\tilde{y}|^{d-1+\kappa}}.
\end{multline}
\end{theo}

\selectlanguage{english}

\pagestyle{plain}

\section{Convergence towards a boundary layer tail}
\label{secCV}

\subsection{The convergence proof}

It follows from theorem \ref{theov=w} that the variational solution $v_{bl}$ of \eqref{sysbl} in $\Omega_{n,0}$ (see section \ref{subsecdgvnmdiv}) can be expressed by the mean of Poisson's kernel associated to $-\nabla\cdot A(y)\nabla\cdot$ and $\Omega_n$. Hence, by theorem \ref{theodvptP}, for all $y\in\Omega_n$, for all $i\in\{1,\ldots\ N\}$, 
\begin{align}\label{reprintvbldevptP}
&v_{bl,i}(y)=\int_{\partial\Omega_n}P_{ij}(y,\tilde{y})v_{0,j}(\tilde{y})d\tilde{y}=\int_{\partial\Omega_n}P^T_{ji}(y,\tilde{y})v_{0,j}(\tilde{y})d\tilde{y}\\
&=\int_{\partial\Omega_n}A^*_{jk}(\tilde{y})\nabla_{\tilde{y}}\Bigl[G^{*,0}(\tilde{y},y)+\chi^*(\tilde{y})\cdot\nabla_{\tilde{y}}G^{*,0}(\tilde{y},y)\Bigr.\nonumber\\
&\qquad+G^{*,1}_{bl}(\tilde{y})\cdot\nabla_{\tilde{y}}G^{*,0}(\tilde{y},y)+\Gamma^*(\tilde{y})\cdot\nabla^2_{\tilde{y}}G^{*,0}(\tilde{y},y)\nonumber\\
&\qquad\qquad\Bigl.+\chi^*(\tilde{y})G^{*,1}_{bl}(\tilde{y})\cdot\nabla^2_{\tilde{y}}G^{*,0}(\tilde{y},y)\Bigr]_{ki}\cdot n\ v_{0,j}(\tilde{y})d\tilde{y}+\int_{\partial\Omega_n}R_{i}(y,\tilde{y})d\tilde{y}\nonumber
\end{align}
where for all $\tilde{y}\in\partial\Omega_n$, $\left|R_i(y,\tilde{y})\right|\leq\frac{C}{|y-\tilde{y}|^{d-1+\kappa}}$. The bound on the remainder $R_i$ yields
\begin{multline*}
\left|\int_{\partial\Omega_n}R_{i}(y,\tilde{y})d\tilde{y}\right|\leq C\int_{\partial\Omega_n}\frac{1}{|y-\tilde{y}|^{d-1+\kappa}}d\tilde{y}\\
\leq C\int_{\mathbb R^{d-1}\times\{0\}}\frac{1}{|z-\tilde{z}|^{d-1+\kappa}}d\tilde{z}\leq\frac{C}{(y\cdot n)^\kappa}\int_{\mathbb R^{d-1}}\frac{1}{\left(1+|\tilde{z}'|^2\right)^\frac{d-1+\kappa}{2}}d\tilde{z}'\leq \frac{C}{(y\cdot n)^\kappa}\stackrel{y\cdot n\rightarrow\infty}{\longrightarrow}0.
\end{multline*}
where $z:=\mathrm{M}^Ty$. It remains to handle the other terms. The boundary function $v_0$ is quasiperiodic along the boundary $\partial\Omega_n$ albeit not periodic. This suggests to use the following lemma to take advantage of the ergodic properties related to the quasiperiodic setting.
\begin{lem}[\cite{Subin74} theorem S.$3$]\label{lemergodic}
Let $f=f(y)\in\mathbb R$ be a quasiperiodic function on $\mathbb R^d$. Then, there exists $\mathcal{M}\{f\}\in\mathbb R$ such that for all $\varphi\in L^1\left(\mathbb R^d\right)$,
\begin{equation*}
\int_{\mathbb R^d}\varphi(y)f\left(\frac{y}{\varepsilon}\right)dy\stackrel{\varepsilon\rightarrow 0}{\longrightarrow}\mathcal M\{f\}\int_{\mathbb R^d}\varphi(y)dy.
\end{equation*}
\end{lem}

Let $i,\ j\in\{1,\ldots\ N\}$ be fixed. We focus on the convergence when $y\cdot n\rightarrow\infty$ of
\begin{equation*}
\int_{\partial\Omega_n}A^*_{jk}(\tilde{y})\nabla_{\tilde{y}}G_{ki}^{*,0}(\tilde{y},y)\cdot n\ v_{0,j}(\tilde{y})d\tilde{y}=\int_{\partial\Omega_n}A^{*,\alpha\beta}_{jk}(\tilde{y})\partial_{\tilde{y}_\beta}G^{*,0}_{ki}(\tilde{y},y)n_\alpha v_{0,j}(\tilde{y})d\tilde{y}.
\end{equation*}
For all $R>0$, for all $y=y'+(y\cdot n)n\in\Omega_n$ with $y'\in\partial\Omega_n\cap B(0,R)$, taking $z:=\mathrm{M}^Ty$ and $\varepsilon:=\frac{1}{y\cdot n}>0$, we get
\begin{align}
&\int_{\partial\Omega_n}A^{*,\alpha\beta}_{jk}(\tilde{y})\partial_{\tilde{y}_\beta}G^{*,0}_{ki}(\tilde{y},y)n_\alpha v_{0,j}(\tilde{y})d\tilde{y}\nonumber\\
&=\int_{\mathbb R^{d-1}\times\{0\}}A^{*,\alpha\beta}_{jk}(\mathrm{M}\tilde{z})\partial_{1,\beta}G^{*,0}_{ki}(\mathrm{M}\tilde{z},\mathrm{M}(z',z_d))n_\alpha v_{0,j}(\mathrm{M}\tilde{z})d\tilde{z}\nonumber\\
&=\int_{\mathbb R^{d-1}}A^{*,\alpha\beta}_{jk}\left(\mathrm{M}\left(\frac{\tilde{z}'}{\varepsilon},0\right)\right)\frac{1}{\varepsilon^{d-1}}\partial_{1,\beta}G^{*,0}_{ki}\left(\frac{\mathrm{M}(\tilde{z}',0)}{\varepsilon},\frac{\mathrm{M}(\varepsilon z',1)}{\varepsilon}\right)n_\alpha v_{0,j}\!\left(\!\mathrm{M}\left(\frac{\tilde{z}'}{\varepsilon},0\!\right)\!\right)d\tilde{z}'\nonumber\\
&=\int_{\mathbb R^{d-1}}\partial_{1,\beta}G^{*,0}_{ki}\left(\mathrm{M}(\tilde{z}',0),\mathrm{M}(0,1)\right)n_\alpha A^{*,\alpha\beta}_{jk}\left(\mathrm{M}\left(\frac{\tilde{z}'}{\varepsilon},0\right)\right)v_{0,j}\!\left(\!\mathrm{M}\left(\frac{\tilde{z}'}{\varepsilon},0\!\right)\!\right)d\tilde{z}'\nonumber\\
&\qquad+\int_{\mathbb R^{d-1}}\Bigl[\partial_{1,\beta}G^{*,0}_{ki}\left(\mathrm{M}(\tilde{z}',0),\mathrm{M}(\varepsilon z',1)\right)\Bigr.\label{eqCVBLtailterme1}\\
&\qquad\qquad\qquad\qquad\Bigl.-\partial_{1,\beta}G^{*,0}_{ki}\left(\mathrm{M}(\tilde{z}',0),\mathrm{M}(0,1)\right)\Bigr]n_\alpha A^{*,\alpha\beta}_{jk}\left(\mathrm{M}\left(\frac{\tilde{z}'}{\varepsilon},0\right)\right)v_{0,j}\!\left(\!\mathrm{M}\left(\frac{\tilde{z}'}{\varepsilon},0\!\right)\!\right)d\tilde{z}'.\nonumber
\end{align}
Let us show that the second term in \eqref{eqCVBLtailterme1} tends to $0$ when $\varepsilon\rightarrow 0$, uniformly in $z'\in B(0,R)\subset\mathbb R^{d-1}$. For $\varepsilon$ sufficiently small such that $\varepsilon R\leq 1$,
\begin{align*}
&\left|\int_{\mathbb R^{d-1}}\left[\partial_{1,\beta}G^{*,0}_{ki}\left(\mathrm{M}(\tilde{z}',0),\mathrm{M}(\varepsilon z',1)\right)-\partial_{1,\beta}G^{*,0}_{ki}\left(\mathrm{M}(\tilde{z}',0),\mathrm{M}(0,1)\right)\right]\right.\\
&\qquad\qquad\qquad\qquad\qquad\qquad\qquad\left. n_\alpha A^{*,\alpha\beta}_{jk}\left(\mathrm{M}\left(\frac{\tilde{z}'}{\varepsilon},0\right)\right)v_{0,j}\!\left(\!\mathrm{M}\left(\frac{\tilde{z}'}{\varepsilon},0\!\right)\!\right)d\tilde{z}'\right|\\
&\leq C\int_{\mathbb R^{d-1}}\sup_{u'\in B\left(0,\varepsilon R\right)}\left|\nabla_{1}\nabla_{2}G^{*,0}\left(\mathrm{M}(\tilde{z}',0),\mathrm{M}(u',1)\right)\right|\varepsilon|z'|d\tilde{z}'\\
&\leq C\varepsilon R\int_{\mathbb R^{d-1}}\sup_{u'\in B\left(0,1\right)}\frac{C}{\left(1+|\tilde{z}'-u'|^2\right)^\frac{d}{2}}d\tilde{z}'\\
&\leq C\varepsilon.
\end{align*}
From the bound \eqref{estALinGdqcq}, we get for all $\tilde{y}\in\partial\Omega_n$,
\begin{equation*}
\left|\nabla_{1}G^{*,0}(\tilde{y},y)\right|\leq C\frac{y\cdot n}{|y-\tilde{y}|^d},
\end{equation*}
which shows that 
\begin{equation*}
\left|\partial_{1,\beta}G^{*,0}_{ki}\left(\mathrm{M}(\tilde{z}',0),\mathrm{M}(0,1)\right)\right|\leq \frac{C}{\left(1+|\tilde{z}'|^2\right)^\frac{d}{2}}.
\end{equation*}
Hence $\partial_{1,\beta}G^{*,0}_{ki}\left(\mathrm{M}(\tilde{z}',0),\mathrm{M}(0,1)\right)\in L^1_{\tilde{z}'}\left(\mathbb R^{d-1}\right)$ and one can apply lemma \ref{lemergodic} to get the convergence of the first term in \eqref{eqCVBLtailterme1}:
\begin{multline*}
\int_{\mathbb R^{d-1}}\partial_{1,\beta}G^{*,0}_{ki}\left(\mathrm{M}(\tilde{z}',0),\mathrm{M}(0,1)\right)n_\alpha A^{*,\alpha\beta}_{jk}\left(\mathrm{M}\left(\frac{\tilde{z}'}{\varepsilon},0\right)\right)v_{0,j}\!\left(\!\mathrm{M}\left(\frac{\tilde{z}'}{\varepsilon},0\!\right)\!\right)d\tilde{z}'\\
\stackrel{\varepsilon\rightarrow 0}{\longrightarrow}\mathcal M
\left\{\begin{array}{rcl}
\partial\Omega_n&\longrightarrow&\mathbb R\\
\tilde{y}&\longmapsto&A^{*,\alpha\beta}_{jk}(\tilde{y})v_{0,j}(\tilde{y})n_\alpha
\end{array}\right\}\int_{\Omega_n}\partial_{\tilde{y}_\beta}G_{ki}^{*,0}(\tilde{y},n)d\tilde{y}.
\end{multline*}
The reasoning is identical for 
\begin{subequations}
\begin{align}
&\int_{\partial\Omega_n}A^*_{jk}(\tilde{y})\nabla_{\tilde{y}}\left[\chi^*_{kl}(\tilde{y})\nabla_{\tilde{y}}G^{*,0}_{li}(\tilde{y},y)\right]\cdot n v_{0,j}(\tilde{y})d\tilde{y},\\
&\int_{\partial\Omega_n}A^*_{jk}(\tilde{y})\nabla_{\tilde{y}}\left[G^{*,1}_{bl,kl}(\tilde{y})\nabla_{\tilde{y}}G^{*,0}_{li}(\tilde{y},y)\right]\cdot n v_{0,j}(\tilde{y})d\tilde{y},\label{termeeps0CVCL}
\end{align}
\end{subequations}
as both terms involve just one derivative of $G^{*,0}$. For \eqref{termeeps0CVCL}, we notice that $\tilde{y}\mapsto\partial_{\tilde{y}_\beta}G^{*,1,\gamma}_{bl}(\tilde{y})$ is quasiperiodic on $\partial\Omega_n$: we know from section \ref{subsecdgvnmdiv} that there is a unique smooth $V^\gamma=V^\gamma(\theta,t)\in\mathbb R^N$, defined for $\theta\in\mathbb T^d$ and $t\geq 0$, such that for all $\tilde{y}\in\partial\Omega_n$,
\begin{equation*}
G^{*,1,\gamma}_{bl}(\tilde{y})=G^{*,1,\gamma}_{bl}(\mathrm{M}\tilde{z})=V^\gamma(\mathrm{N}\tilde{z}',0).
\end{equation*}

The other terms in \eqref{reprintvbldevptP} involving strictly more than one derivative of $G^{*,0}$ tend to $0$ when $\varepsilon\rightarrow 0$. Indeed, for all $R>0$, for all $y=y'+(y\cdot n)n\in\Omega_n$ with $y'\in\partial\Omega_n\cap B(0,R)$, taking again $z:=\mathrm{M}^Ty$ and $\varepsilon:=\frac{1}{y\cdot n}>0$, 
\begin{align*}
&\int_{\partial\Omega_n}A^{*,\alpha\beta}_{jk}(\tilde{y})\chi^{*,\gamma}_{kl}(\tilde{y})\partial^2_{\tilde{y}_{\beta\gamma}}G^{*,0}_{li}(\tilde{y},y)n_\alpha v_{0,j}(\tilde{y})d\tilde{y}\\
&=\int_{\mathbb R^{d-1}\times\{0\}}A^{*,\alpha\beta}_{jk}(\mathrm{M}\tilde{z})\chi^{*,\gamma}_{kl}(\mathrm{M}\tilde{z})\partial^2_{1,\beta\gamma}G^{*,0}_{li}(\mathrm{M}\tilde{z},\mathrm{M}(z',z_d))n_\alpha v_{0,j}(\mathrm{M}\tilde{z})d\tilde{z}\\
&=\varepsilon\int_{\mathbb R^{d-1}}A^{*,\alpha\beta}_{jk}\left(\mathrm{M}\left(\frac{\tilde{z}'}{\varepsilon},0\right)\right)\chi^{*,\gamma}_{kl}\left(\mathrm{M}\left(\frac{\tilde{z}'}{\varepsilon},0\right)\right)\partial^2_{1,\beta\gamma}G^{*,0}_{ki}\left(\mathrm{M}(\tilde{z}',0),\mathrm{M}(\varepsilon z',1)\right)\\
&\qquad\qquad\qquad\qquad\qquad\qquad\qquad\qquad\qquad\qquad\qquad\qquad\qquad\qquad n_\alpha v_{0,j}\!\left(\!\mathrm{M}\left(\frac{\tilde{z}'}{\varepsilon},0\!\right)\!\right)d\tilde{z}',
\end{align*}
which is easily shown to be of order $O(\varepsilon)$. We proceed following the same method for the remaing terms. This demonstrates the convergence of the boundary layer towards a constant vector field $v_{bl}^\infty$, the boundary layer tail: for all $y=y'+(y\cdot n)n\in\Omega_n$ with $y'\in\partial\Omega_n\cap B(0,R)$,
\begin{multline}\label{CVBLexprBLtail}
v_{bl}(\tilde{y})\stackrel{y\cdot n\rightarrow\infty}{\longrightarrow}v_{bl}^\infty:=\int_{\partial\Omega_n}\partial_{\tilde{y}_\alpha}G^0(n,\tilde{y})d\tilde{y}\Biggl[\mathcal M\left\{A^{\beta\alpha}(\tilde{y})v_0(\tilde{y})n_\beta\right\}\Biggr.\\
\Biggl.+\mathcal M\left\{\partial_{\tilde{y}_\beta}\left(\chi^{*,\alpha}\right)^T(\tilde{y})A^{\beta\gamma}(\tilde{y})v_0(\tilde{y})n_\gamma\right\}+\mathcal M\left\{\partial_{\tilde{y}_\beta}\bigl(G^{*,1,\alpha}_{bl}\bigr)^T\!\!\!(\tilde{y})A^{\beta\gamma}(\tilde{y})v_0(\tilde{y})n_\gamma\right\}\Biggr]
\end{multline}
locally uniformly in $y'$. Moreover, \eqref{CVBLexprBLtail} yields an explicit expression for $v_{bl}^\infty$ in terms of the means $\mathcal M\left\{\cdot\right\}$ on $\partial\Omega_n$. 

\subsection{The boundary layer tail does not depend on $a$}
\label{secindep}

In the preceding section we have shown the convergence of $v_{bl}$ defined in $\Omega_{n,a}$ towards $v_{bl}^{a,\infty}$ when $y\cdot n\rightarrow\infty$. It remains to show that $v^{a,\infty}_{bl}$ is independent from $a$ in order to complete the proof of theorem \ref{theoCVBLtail}. The fact that $n\notin\mathbb R\mathbb Q^d$ is crucial. To do so, we generalize lemma $6$ in \cite{dgvnm}:
\begin{prop}\label{propcontinuityava}
Assume that $n\notin\mathbb R\mathbb Q^d$. Then,
\begin{equation*}
a\in\mathbb R\longmapsto v_{bl}^{a,\infty}\in\mathbb R^N
\end{equation*}
is Lipschitz continuous.
\end{prop}
The proof in \cite{dgvnm} relies on the small divisors assumption \eqref{smalldivisors1} and energy estimates. We, instead, have recourse to Poisson's integral formula and estimates on Poisson's kernel. Let $a,\ a'\in\mathbb R$ and $\nu:=\left|a'-a\right|$. We call $G^a$ (resp. $P^a$) the Green (resp. Poisson) kernel associated to $-\nabla\cdot A\left(\cdot+an\right)\nabla\cdot$ and $\Omega_{n}=\Omega_{n,0}$. We define analogously $G^{a'}$ and $P^{a'}$. We also have the $*$-versions $G^{*,a}$ and $G^{*,a'}$ corresponding to the transposed operator (see section \ref{subsecGreenPoisson}). The following lemma is an adaptation of the results due to Avellaneda and Lin (see \cite{alin}), Kenig and Shen (see \cite{ks11} section $2$) and G\'erard-Varet and Masmoudi (see \cite{dgvnm2} appendix A).
\begin{lem}
Let $0<\mu<1$. There exists $C>0$ such that for all $y,\ \tilde{y}\in\Omega_{n}$, $y\neq\tilde{y}$,
\begin{subequations}
\begin{align}
\left|G^{a}(y,\tilde{y})-G^{a'}(y,\tilde{y})\right|&\leq C\nu\frac{1}{\left|y-\tilde{y}\right|^{d-2}},\qquad\mbox{if }d\geq 3,\label{estG*aG*a'd-2}\\
\left|G^{a}(y,\tilde{y})-G^{a'}(y,\tilde{y})\right|&\leq C\nu\frac{\left(y\cdot n\right)^\mu\left(\tilde{y}\cdot n\right)^\mu}{|y-\tilde{y}|^{d-2+2\mu}},\qquad\mbox{for all }d\geq 2\label{estG*aG*a'd}.
\end{align}
\end{subequations}
\end{lem}

\begin{proof}
The ideas for the proofs are in large part taken from the reference above. For \eqref{estG*aG*a'd-2}, we rely on a representation formula of $G^{a}(y,\tilde{y})-G^{a'}(y,\tilde{y})$: 
\begin{multline}\label{relationG*aG*a'}
G^{a}(y,\tilde{y})-G^{a'}(y,\tilde{y})\\
=\int_{\Omega_n}\partial_{2,\alpha}G^{a}(y,\hat{y})\left(A^{\alpha\beta}(\hat{y}+na')-A^{\alpha\beta}(\hat{y}+na)\right)\partial_{1,\beta}G^{a'}(\hat{y},\tilde{y})d\hat{y}.
\end{multline}
This formula, which is proven in \cite{HofKim07} (see corollary $3.5$) for the domain $\mathbb R^d$, $d\geq 3$, is a consequence of Green's representation formula. The proof of Hofmann and Kim extends to the domain $\Omega_n$. From \eqref{relationG*aG*a'} and \eqref{estALinnablaGdqcq}, one deduces that for all $d\geq 3$, $y,\ \tilde{y}\in\Omega_n$, $y\neq\tilde{y}$,
\begin{align*}
\left|G^{a}(y,\tilde{y})-G^{a'}(y,\tilde{y})\right|&\leq C\nu\int_{\Omega_n}\frac{1}{|y-\hat{y}|^{d-1}}\frac{1}{|\tilde{y}-\hat{y}|^{d-1}}d\hat{y}\\
&\leq C\nu\int_{\mathbb R^d}\frac{1}{|y-\tilde{y}-\hat{y}|^{d-1}}\frac{1}{|\hat{y}|^{d-1}}d\hat{y}\\
&\leq C\nu\frac{1}{|y-\tilde{y}|^{d-2}}\int_{\mathbb R^d}\frac{1}{\left|\frac{y-\tilde{y}}{|y-\tilde{y}|}-\hat{y}\right|^{d-1}}\frac{1}{|\hat{y}|^{d-1}}d\hat{y}\\
&\leq C\nu\frac{1}{|y-\tilde{y}|^{d-2}}.
\end{align*}

For \eqref{estG*aG*a'd} we have recourse to the local boundary estimate \eqref{boundaryest}. Assume that $d\geq 3$. For given $y,\ \tilde{y}\in\Omega_n$, $y\neq\tilde{y}$, we first establish the bound
\begin{equation}\label{estG*aG*a'dinter}
\left|G^{a}(y,\tilde{y})-G^{a'}(y,\tilde{y})\right|\leq C\nu\frac{\left(y\cdot n\right)^\mu}{|y-\tilde{y}|^{d-2+\mu}}.
\end{equation}
Let $r:=|y-\tilde{y}|$. We distinguish between two cases. If $y\cdot n\geq \frac{r}{3}$, then \eqref{estG*aG*a'dinter} follows directly from \eqref{estG*aG*a'd-2}. Assume that $y\cdot n<\frac{r}{3}$ and let $\bar{y}\in\partial\Omega_n$ such that $y\cdot n=|y-\bar{y}|$. Then, from \eqref{sysGreen} it comes that $G^a(\cdot,\tilde{y})-G^{a'}(\cdot,\tilde{y})$ satisfies
\begin{equation}\label{sysGa-Ga'}
\left\{
\begin{array}{rl}
-\nabla_y\cdot A(\cdot+na)\nabla_y\left(G^{a}\left(\cdot,\tilde{y}\right)-G^{a'}\left(\cdot,\tilde{y}\right)\right)&=\nabla\cdot\left[A(\cdot+na)-A(\cdot+na')\right]\nabla G^{a'}(\cdot,\tilde{y})\\
&\qquad\qquad\qquad\qquad\qquad\mbox{in}\quad D\left(\bar{y},\frac{r}{3}\right)\\
G^{a}\left(\cdot,\tilde{y}\right)-G^{a'}\left(\cdot,\tilde{y}\right)&=0\qquad \mbox{in}\quad \Gamma\left(\bar{y},\frac{r}{3}\right)
\end{array}
\right. .
\end{equation}
Applying a rescaled version of \eqref{boundaryest}, one gets using \eqref{estG*aG*a'd-2} and \eqref{estALinnablaGdqcq},
\begin{align*}
&\left|G^a(y,\tilde{y})-G^{a'}(y,\tilde{y})\right|=\left|G^a(y,\tilde{y})-G^{a'}(\bar{y},\tilde{y})-\left(G^{a'}(y,\tilde{y})-G^a(\bar{y},\tilde{y})\right)\right|\\
&\leq\left\|G^a(\cdot,\tilde{y})-G^{a'}(\cdot,\tilde{y})\right\|_{C^{0,\mu}\left(D\left(\bar{y},\frac{r}{6}\right)\right)}\frac{|y\cdot n|^\mu}{r^\mu}r^\mu\\
&\leq C\left[\frac{1}{r^\frac{d}{2}}\left\|G^a(\cdot,\tilde{y})-G^{a'}(\cdot,\tilde{y})\right\|_{L^2\left(D\left(\bar{y},\frac{r}{3}\right)\right)}\right.\\
&\left.\qquad\qquad+\frac{r}{r^{d+\delta}}\left\|\left[A(\cdot+na)-A(\cdot+na')\right]\nabla G^{a'}(\cdot,\tilde{y})\right\|_{L^\frac{d}{d+\delta}\left(D\left(\bar{y},\frac{r}{3}\right)\right)}\right]\frac{|y\cdot n|^\mu}{r^\mu}\\
&\leq C\left[\left\|G^a(\cdot,\tilde{y})-G^{a'}(\cdot,\tilde{y})\right\|_{L^\infty\left(D\left(\bar{y},\frac{r}{3}\right)\right)}+r\nu\left\|\nabla G^{a'}(\cdot,\tilde{y})\right\|_{L^\infty\left(D\left(\bar{y},\frac{r}{3}\right)\right)}\right]\frac{|y\cdot n|^\mu}{r^\mu}\\
&\leq C\nu\left[\sup_{\hat{y}\in D\left(\bar{y},\frac{r}{3}\right)}\frac{1}{|\hat{y}-\tilde{y}|^{d-2}}+r\sup_{\hat{y}\in D\left(\bar{y},\frac{r}{3}\right)}\frac{1}{|\hat{y}-\tilde{y}|^{d-1}}\right]\frac{|y\cdot n|^\mu}{r^\mu}\\
&\leq C\nu\frac{|y\cdot n|^\mu}{r^{d-2+\mu}},
\end{align*}
as for all $\hat{y}\in D\left(\bar{y},\frac{r}{3}\right)$, $|\hat{y}-\tilde{y}|>\frac{r}{6}$. This shows \eqref{estG*aG*a'dinter}. We now turn to the proof of \eqref{estG*aG*a'd} itself. If $\tilde{y}\cdot n\geq \frac{r}{3}$, then \eqref{estG*aG*a'd} follows directly from \eqref{estG*aG*a'dinter}. Assume that $\tilde{y}\cdot n<\frac{r}{3}$ and let $\bar{y}\in\partial\Omega_n$ such that $\tilde{y}\cdot n=|\tilde{y}-\bar{y}|$. Applying a rescaled version of \eqref{boundaryest} to $G^{*,a}\left(\cdot,y\right)-G^{*,a'}\left(\cdot,y\right)$ satisfying \eqref{sysGa-Ga'} with $A$ (resp. $G^{a'}$, $y$, $\tilde{y}$) replaced by $A^*$ (resp. $G^{*,a'}$, $\tilde{y}$, $y$), one gets using \eqref{estG*aG*a'dinter}, \eqref{estALinnablaGdqcq2}, and for all $\hat{y}\in D\left(\bar{y},\frac{r}{3}\right)$ $|\hat{y}-y|>\frac{r}{6}$,
\begin{align*}
&\left|G^a(y,\tilde{y})-G^{a'}(y,\tilde{y})\right|=\left|G^{*,a}(\tilde{y},y)-G^{*,a'}(\tilde{y},y)\right|\\
&\leq C\left[\left\|G^{*,a}(\cdot,y)-G^{*,a'}(\cdot,y)\right\|_{L^\infty\left(D\left(\bar{y},\frac{r}{3}\right)\right)}+r\nu\left\|\nabla G^{*,a'}(\cdot,y)\right\|_{L^\infty\left(D\left(\bar{y},\frac{r}{3}\right)\right)}\right]\frac{|\tilde{y}\cdot n|^\mu}{r^\mu}\\
&\leq C\nu\left[\sup_{\hat{y}\in D\left(\bar{y},\frac{r}{3}\right)}\left|G^{a}(y,\hat{y})-G^{a'}(y,\hat{y})\right|+r\sup_{\hat{y}\in D\left(\bar{y},\frac{r}{3}\right)}\frac{y\cdot n}{|\hat{y}-y|^{d}}\right]\frac{|\tilde{y}\cdot n|^\mu}{r^\mu}\\
&\leq C\nu\left[\sup_{\hat{y}\in D\left(\bar{y},\frac{r}{3}\right)}\frac{\left(y\cdot n\right)^\mu}{|\hat{y}-y|^{d-2+\mu}}+r\sup_{\hat{y}\in D\left(\bar{y},\frac{r}{3}\right)}\frac{\left(y\cdot n\right)^\mu}{|\hat{y}-y|^{d-1+\mu}}\frac{\left(y\cdot n\right)^{1-\mu}}{|\hat{y}-y|^{1-\mu}}\right]\frac{|\tilde{y}\cdot n|^\mu}{r^\mu}\\
&\leq C\nu\left[\sup_{\hat{y}\in D\left(\bar{y},\frac{r}{3}\right)}\frac{\left(y\cdot n\right)^\mu}{|\hat{y}-y|^{d-2+\mu}}+r\sup_{\hat{y}\in D\left(\bar{y},\frac{r}{3}\right)}\frac{\left(y\cdot n\right)^\mu}{|\hat{y}-y|^{d-1+\mu}}\right]\frac{|\tilde{y}\cdot n|^\mu}{r^\mu}\\
&\leq C\nu\frac{\left(y\cdot n\right)^\mu\left(\tilde{y}\cdot n\right)^\mu}{|y-\tilde{y}|^{d-2+2\mu}},
\end{align*}
on condition that $y\cdot n\leq |\hat{y}-y|$ for all $\hat{y}\in D\left(\bar{y},\frac{r}{3}\right)$. If the latter is not true, i.e. if there is $\hat{y}\in D\left(\bar{y},\frac{r}{3}\right)$ such that $y\cdot n>|\hat{y}-y|>\frac{r}{6}$, then we apply the same reasoning as above with $A$ (resp. $G^{a'}$, $y$, $\tilde{y}$) replaced by $A^*$ (resp. $G^{*,a'}$, $\tilde{y}$, $y$) to get
\begin{equation}\label{estG*aG*a'dinter2}
\left|G^a(y,\tilde{y})-G^{a'}(y,\tilde{y})\right|=\left|G^{*,a}(\tilde{y},y)-G^{*,a'}(\tilde{y},y)\right|\leq C\nu\frac{(\tilde{y}\cdot n)^\mu}{|y-\tilde{y}|^{d-2+\mu}},
\end{equation}
and deduce \eqref{estG*aG*a'd} from \eqref{estG*aG*a'dinter2} and $y\cdot n>\frac{r}{6}$. The two-dimensional bound follows from the three-dimensional one as explained in \cite{dgvnm2} and \cite{alin}.
\end{proof}

Let us notice that proposition \ref{propcontinuityava} implies that for all $a,\ a'\in\mathbb R$, $v^{a,\infty}_{bl}=v^{a',\infty}_{bl}$. Indeed, for $\xi\in\mathbb Z^d$, $v^\xi_{bl}$ solving
\begin{equation*}
\left\{\begin{array}{rll}
-\nabla\cdot A(y)\nabla v^\xi_{bl}&=0,&y\cdot n-a+\xi\cdot n>0\\
v^\xi_{bl}&=v_0(y),& y\cdot n-a+\xi\cdot n=0
\end{array}\right. 
\end{equation*}
satisfies, by periodicity of the coefficients and of the Dirichlet data, $v^\xi_{bl}(\cdot):=v_{bl}(\cdot+\xi)$. Hence, $v^{a-\xi\cdot n,\infty}_{bl}=v^{a,\infty}_{bl}$. As $n\notin\mathbb R\mathbb Q^d$, the set $\{\xi\cdot n,\ \xi\in\mathbb Z^d\}$ is dense in $\mathbb R$. The independence of $v^{a,\infty}_{bl}$ from $a$ follows now from the continuity of $a\mapsto v^{a,\infty}_{bl}$. 

The rest of this section is devoted to the proof of proposition \ref{propcontinuityava}. Let as before $a,\ a'\in\mathbb R$ and $\nu:=|a'-a|$. Let $v_{bl}$ be the solution of \eqref{sysbl} and $v_{bl}'$ the solution of \eqref{sysbl} in the domain $\Omega_{n,a'}$ instead of $\Omega_{n,a}$. Then $v_{bl}=v^a_{bl}(\cdot-an)$ (resp. $v_{bl}=v^{a'}_{bl}(\cdot-a'n)$) where $v^a_{bl}$ (resp. $v^{a'}_{bl}$) solves
\begin{equation}\label{sysv^abl}
\left\{\begin{array}{rll}
-\nabla\cdot A(y+an)\nabla v^a_{bl}&=0,&y\cdot n>0\\
v^a_{bl}&=v_0(y+an),& y\cdot n=0
\end{array}\right. 
\end{equation}
(resp. \eqref{sysv^abl} with $a$ replaced by $a'$). Take $\varphi\in C^\infty_c\left(\mathbb R\right)$, compactly supported in $[-1,1]$ such that $0\leq\varphi\leq 1$ and $\varphi\equiv 1$ on $\left[-\frac{1}{2},\frac{1}{2}\right]$. Then $\tilde{v}^a_{bl}:=v^a_{bl}-\varphi(y\cdot n)v_0(y+na)$ (resp. $\tilde{v}^{a'}_{bl}:=v^{a'}_{bl}-\varphi(y\cdot n)v_0(y+na')$) solves
\begin{equation}\label{sysvtilde^abl}
\left\{\begin{array}{rll}
-\nabla\cdot A(y+an)\nabla \tilde{v}^a_{bl}&=\nabla\cdot A(y+na)\nabla\left(\varphi(y\cdot n)v_0(y+na)\right),&y\cdot n>0\\
\tilde{v}^a_{bl}&=0,& y\cdot n=0
\end{array}\right. 
\end{equation}
(resp. \eqref{sysvtilde^abl} with $a$ replaced by $a'$). One important point is that the source term in \eqref{sysvtilde^abl} is compactly supported in the direction normal to the boundary. We now estimate
\begin{equation*}
v^{a}_{bl}(y)-v^{a'}_{bl}(y)=\tilde{v}^a_{bl}(y)-\tilde{v}^{a'}_{bl}(y)+\varphi(y\cdot n)\left[v_0(y+na)-v_0(y+na')\right]
\end{equation*} 
for all $y\in\Omega_n$. We have, 
\begin{subequations}\label{splitingv^a-v^a'}
\begin{align}
&\tilde{v}^a_{bl}(y)-\tilde{v}^{a'}_{bl}(y)\nonumber\\
&\qquad=\int_{\Omega_n}G^a(y,\tilde{y})\nabla\cdot A(\tilde{y}+na)\nabla\left(\varphi(\tilde{y}\cdot n)v_0(\tilde{y}+na)\right)d\tilde{y}\nonumber\\
&\qquad\qquad-\int_{\Omega_n}G^{a'}(y,\tilde{y})\nabla\cdot A(\tilde{y}+na')\nabla\left(\varphi(\tilde{y}\cdot n)v_0(\tilde{y}+na')\right)d\tilde{y}\nonumber\\
&\qquad=\int_{\Omega_n}\left[G^a(y,\tilde{y})-G^{a'}(y,\tilde{y})\right]\nabla\cdot A(\tilde{y}+na)\nabla\left(\varphi(\tilde{y}\cdot n)v_0(\tilde{y}+na)\right)d\tilde{y}\label{splitingv^a-v^a'1}\\
&\qquad\quad+\int_{\Omega_n}G^{a'}(y,\tilde{y})\nabla\cdot\left[A(\tilde{y}+na)-A(\tilde{y}+na')\right]\nabla\left(\varphi(\tilde{y}\cdot n)v_0(\tilde{y}+na)\right)d\tilde{y}\label{splitingv^a-v^a'2}\\
&\qquad\quad\quad+\int_{\Omega_n}G^{a'}(y,\tilde{y})\nabla\cdot A(\tilde{y}+na')\nabla\left(\varphi(\tilde{y}\cdot n)\left[v_0(\tilde{y}+na)-v_0(\tilde{y}+na')\right]\right)d\tilde{y}\label{splitingv^a-v^a'3}.
\end{align}
\end{subequations}
We analyse the terms in r.h.s. separately. The first, \eqref{splitingv^a-v^a'1} deserves more attention. By estimate \eqref{estG*aG*a'd} and the usual change of variables $\tilde{z}=\mathrm{M}^T\tilde{y}$, for $y\cdot n\geq 1$,
\begin{align*}
&\left|\int_{\Omega_n}\left[G^a(y,\tilde{y})-G^{a'}(y,\tilde{y})\right]\nabla\cdot A(\tilde{y}+na)\nabla\left(\varphi(\tilde{y}\cdot n)v_0(\tilde{y}+na)\right)d\tilde{y}\right|\\
&\qquad\leq C\nu\int_{\Omega_n}\frac{\left(y\cdot n\right)^\mu\left(\tilde{y}\cdot n\right)^\mu}{|y-\tilde{y}|^{d-2+2\mu}}\left|\nabla\cdot A(\tilde{y}+na)\nabla\left(\varphi(\tilde{y}\cdot n)v_0(\tilde{y}+na)\right)\right|d\tilde{y}\\
&\qquad\leq C\nu\left(y\cdot n\right)^\mu\int_{\mathbb R^{d-1}}\frac{1}{\left[\left(y\cdot n-1\right)^2+|\tilde{z}'|^2\right]^\frac{d-2+2\mu}{2}}d\tilde{z}'\\
&\qquad\leq C\nu\frac{1}{\left(y\cdot n\right)^{-1+\mu}}\int_{\mathbb R^{d-1}}\frac{1}{\left[1+|u'|^2\right]^\frac{d-2+2\mu}{2}}du'.
\end{align*}
We need, $2\mu>1$ for the integral to be convergent, and $\mu<1$ for the r.h.s. to be bounded when $y\cdot n\rightarrow\infty$. We now work with such a $\mu$. For \eqref{splitingv^a-v^a'2}, we rely on \eqref{estALinGdqcq}: for $y\cdot n\geq 1$,
\begin{multline*}
\left|\int_{\Omega_n}G^{a'}(y,\tilde{y})\nabla\cdot\left[A(\tilde{y}+na)-A(\tilde{y}+na')\right]\nabla\left(\varphi(\tilde{y}\cdot n)v_0(\tilde{y}+na)\right)d\tilde{y}\right|\\
\leq C\nu\int_{\mathbb R^{d-1}}\frac{1}{\left[1+|u'|^2\right]^\frac{d}{2}}du'
\end{multline*}
which is a convergent integral. For \eqref{splitingv^a-v^a'3}, we argue analogously. We end with
\begin{equation*}
\left|v^{a}_{bl}(y)-v^{a'}_{bl}(y)\right|\leq\left|\tilde{v}^a_{bl}(y)-\tilde{v}^{a'}_{bl}(y)\right|+\varphi(y\cdot n)\left|v_0(y+na)-v_0(y+na')\right|\leq C\nu,
\end{equation*}
for $y\cdot n\geq 1$, which proves proposition \ref{propcontinuityava} letting $y\cdot n\rightarrow\infty$. This concludes the proof of theorem \ref{theoCVBLtail}.

\selectlanguage{english}

\pagestyle{plain}

\section{Almost arbitrarily slow convergence}
\label{secslow}

We exhibit examples in dimension $d=2$ showing that in general the convergence of $v_{bl}$ towards its boundary layer tail $v_{bl}^\infty$ can be nearly arbitrarily slow. Let us, for the rest of this section, focus on the case when $n\notin\mathbb R\mathbb Q^2$. We take $d=2$, $N=1$, $A=I_2$ and study the unique variational solution $v$ of
\begin{equation*}
\left\{
\begin{array}{rll}
-\Delta_z v=&0,& z_2>0\\
v(z)=&v_0(\mathrm{N}z_1),& z_2=0
\end{array}
\right. ,
\end{equation*}
where as usual $\mathrm{N}\in \mathbb R^2$ is the first column vector of an orthogonal matrix $\mathrm{M}$ sending $e_2$ on $n$. From theorem \ref{theodecaydiv} we know that $v=v(z_1,z_2)=V(\mathrm{N}z_1,z_2)$, with $V=V(\theta,t)\in\mathbb R$, $(\theta,t)\in\mathbb T^2\times\mathbb R_+$, solving
\begin{equation}\label{sysVd=2}
\left\{
\begin{array}{rll}
-\begin{pmatrix}
\mathrm{N}\cdot\nabla_\theta\\
\partial_t
\end{pmatrix}^2V=&0,& t>0\\
V(\theta,t)=&v_0(\theta),& t=0
\end{array}
\right. .
\end{equation}
Expanding $v_0$ in Fourier series yields for all $\theta\in\mathbb T^2$,
\begin{equation*}
v_0(\theta)=\sum_{\xi\in\mathbb Z^2}\widehat{v_0}(\xi)e^{2i\pi \xi\cdot\theta},
\end{equation*}
where $\bigl(\widehat{v_0}(\xi)\bigr)_\xi\in l^2(\mathbb Z;\mathbb R)$. From \eqref{aprioriboundV}, in particular $\partial_tV\in L^2(\mathbb T^2\times\mathbb R_+)$, it comes for all $(\theta,t)\in\mathbb T^2\times\mathbb R_+$,
\begin{equation*}
V(\theta,t)=\sum_{\xi\in\mathbb Z^2}\widehat{v_0}(\xi) e^{-2\pi|\mathrm{N}\cdot \xi|t}e^{2i\pi \xi\cdot\theta}.
\end{equation*}
Parceval's equality 
\begin{equation}\label{parceval}
\begin{Vmatrix}
V(\theta,t)-\widehat{v_0}(0)
\end{Vmatrix}_{L^2(\mathbb T^2)}^2=\sum_{\xi\in\mathbb Z^2\setminus\{0\}}\bigl|\widehat{v_0}(\xi)\bigr|^2e^{-4\pi|\mathrm{N}\cdot \xi|t}
\end{equation}
together with Lebesgue's dominated convergence theorem prove that
\begin{equation*}
\begin{Vmatrix}
V(\theta,t)-\widehat{v_0}(0)
\end{Vmatrix}_{L^2(\mathbb T^2)}^2\stackrel{t\rightarrow\infty}{\longrightarrow} 0.
\end{equation*}
As $v_0$ is a $C^\infty$ function, its Fourier coefficients $\bigl(\widehat{v_0}(\xi)\bigr)_\xi$ go to zero when $|\xi|\rightarrow\infty$, faster than any negative power of $|\xi|$. It follows from this, for all $\alpha\in\mathbb N^2$,
\begin{equation}\label{CVforallbeta}
\begin{Vmatrix}
\partial^\alpha_\theta \left(V(\theta,t)-\widehat{v_0}(0)\right)
\end{Vmatrix}_{L^2(\mathbb T^2)}^2\stackrel{t\rightarrow\infty}{\longrightarrow} 0.
\end{equation}
Using Sobolev's injections, one notices that \eqref{CVforallbeta} proves again the convergence of $v$ towards the boundary layer tail $v^\infty_{bl}:=\widehat{v_0}(0)$.

Assume for a moment that $n$ satisfies the small divisors condition \eqref{smalldivisors1}. Let us come back to \eqref{parceval}. For all $m\in\mathbb N$,
\begin{align}\label{majVc^0_0div}
\begin{Vmatrix}
V(\theta,t)-\widehat{v_0}(0)
\end{Vmatrix}_{L^2(\mathbb T^2)}^2&=t^{-m}\sum_{\xi\in\mathbb Z^2\setminus\{0\}}\bigl|\widehat{v_0}(\xi)\bigr|^2t^me^{-4\pi|\mathrm{N}\cdot\xi|t}\nonumber\\
&=t^{-m}\sum_{\xi\in\mathbb Z^2\setminus\{0\}}\frac{\bigl|\widehat{v_0}(\xi)\bigr|^2}{|\mathrm{N}\cdot \xi|^m}\left(|\mathrm{N}\cdot \xi|t\right)^me^{-4\pi|\mathrm{N}\cdot \xi|t}\nonumber\\
&\leq Ct^{-m}\sum_{\xi\in\mathbb Z^2\setminus\{0\}}\bigl|\widehat{v_0}(\xi)\bigr|^2|\xi|^{(2+\tau)m}\left(|\mathrm{N}\cdot \xi|t\right)^me^{-4\pi|\mathrm{N}\cdot \xi|t}\nonumber\\
&\leq Ct^{-m}\sum_{\xi\in\mathbb Z^2\setminus\{0\}}\bigl|\widehat{v_0}(\xi)\bigr|^2|\xi|^{(2+\tau)m}\leq C_mt^{-m},
\end{align}
the function $t\mapsto\left(|\mathrm{N}\cdot \xi|t\right)^me^{-4\pi|\mathrm{N}\cdot \xi|t}$ being bounded on $\mathbb R_+$. We have shown that $V(\cdot,t)$ converges to $\widehat{v_0}(0)$ in $L^2\left(\mathbb T^2\right)$, faster than every negative power of $t$.

Assume now that $n\notin\mathbb R\mathbb Q^2$ does not verify \eqref{smalldivisors1} and let $l>0$. Hence there are points of the lattice  $\mathbb Z^2$, except $0$, which are as close to the line $\mathrm{N}\cdot y=0$ as we wish. The sum in the r.h.s. of \eqref{parceval} does not necessarily keep the trace of the exponential behaviour of its terms. We show that the convergence is at least as slow as the convergence of $t\mapsto t^{-l}$ towards $0$ at $\infty$. In fact we aim at proving:
\begin{theo}\label{theoslowcvL2}
Assume that $n\notin\mathbb R\mathbb Q^2$ does not satisfy \eqref{smalldivisors1}.\\
Then, there exists a smooth $v_0$ and a strictly increasing sequence $(t_M)_{M\geq 1}$ of positive real numbers, tending to $\infty$, such that for all $\alpha\in\mathbb N^2$, for all $M\in\mathbb N\setminus\{0\}$,
\begin{equation*}
\begin{Vmatrix}
\partial_\theta^\alpha\left(V(\theta,t_M)-\widehat{v_0}(0)\right)
\end{Vmatrix}_{L^2(\mathbb T^2)}\geq t_M^{-l},
\end{equation*}
where $V$ is the solution of \eqref{sysVd=2} associated to $v_0$.
\end{theo}
Let us insist on the fact that theorem \ref{theoslowcvL2} holds for any $n\notin\mathbb R\mathbb Q^2$, which does not satisfy the small divisors assumption. The idea of the proof is to choose a family $\bigl(\widehat{v_0}(\xi)\bigr)_\xi$, whose support, that is the set of subscripts $\xi\in\mathbb Z^2$ such that $\widehat{v_0}(\xi)\neq 0$, is sufficiently close to the line $\mathrm{N}\cdot y=0$. We now construct a suitable sequence $(\xi_M)_{M\geq 1}$ of $\mathbb Z^2\setminus\{0\}$. The small divisors assumptions being not verified, for all $M\in\mathbb N\setminus\{0\}$, there exists $\xi\in\mathbb Z^2\setminus\{0\}$ such that 
\begin{equation}\label{nondivM}
|\mathrm{N}\cdot\xi|<\frac{1}{M}|\xi|^{-M}.
\end{equation}
One can construct $(\xi_M)$ recursively:
\begin{equation*}
\xi_1:=\argmin_{
\begin{subarray}{c}
\xi\in\mathbb Z^2\setminus\{0\}\\
|\mathrm{N}\cdot \xi|<|\xi|^{-1}
\end{subarray}}|\xi|,\quad 
\xi_2:=\argmin_{
\begin{subarray}{c}
\xi\in\mathbb Z^2\setminus\{0\}\\
|\xi_2|>|\xi_1|+1\\
|\mathrm{N}\cdot\xi|<\frac{1}{2}|\xi|^{-2}
\end{subarray}}|\xi|,\quad \ldots,\quad \xi_M:=\argmin_{
\begin{subarray}{c}
\xi\in\mathbb Z^2\setminus\{0\}\\
|\xi_M|>|\xi_{M-1}|+1\\
|\mathrm{N}\cdot\xi|<\frac{1}{M}|\xi|^{-M}
\end{subarray}}|\xi|,
\end{equation*} 
where $\argmin$ stands for a minimizor. The existence of a minimizor is ensured by \eqref{nondivM}. For $\xi_1$ the reasoning is straightforward. Let us sketch the proof of the existence of $\xi_2$, which immediately applies, with minor modifications, for all $\xi_M$. We assume that for all $|\xi|>|\xi_1|+1$, $|\mathrm{N}\cdot\xi|\geq\frac{1}{2}|\xi|^{-2}$. Then, for all $M\geq 2$, $|\mathrm{N}\cdot \xi|\geq\frac{1}{2}|\xi|^{-2}\geq\frac{1}{M}|\xi|^{-M}$. According to \eqref{nondivM}, this shows that there exists $\xi\in\mathbb Z^2\setminus\{0\}$, $|\xi|\leq|\xi_1|+1$, such that, for all $M\geq 1$, $|\mathrm{N}\cdot\xi|<\frac{1}{M}|\xi|^{-M}$. For this $\xi\neq 0$, $\mathrm{N}\cdot\xi=0$, which is incompatible with $n\notin\mathbb R\mathbb Q^2$. We have thus built a sequence $(\xi_M)_{M\geq 1}$ of vectors of $\mathbb Z^2\setminus\{0\}$ satisfying:
\begin{enumerate}
\item $\left(|\xi_M|\right)_M$ is strictly increasing;
\item for all $M\geq 1$, $|\xi_M|\geq M$;
\item for all $M\geq 1$, $|\mathrm{N}\cdot \xi_M|<\frac{1}{M}|\xi_M|^{-M}$.
\end{enumerate}

We now come to the construction of $v_0$ keeping in mind that $v_0$ has to be smooth. For all $\xi\in\mathbb Z^2$, we define 
\begin{equation*}
\widehat{v_0}(\xi):=\left\{
\begin{array}{cl}
0, & \mbox{if}\; \xi\neq \xi_M,\: -\xi_M \, \mbox{for all}\; M\geq 1\\
M^{-l}|\xi_M|^{-Ml}, & \mbox{if}\; \xi=\xi_M\: \mbox{or}\: -\xi_M
\end{array}
\right. .
\end{equation*}
Thanks to the construction of the sequence $(\xi_M)_M$, for all $m\in\mathbb N$, $\widehat{v_0}(\xi)=O\left(|\xi|^{-m}\right)$. Thus $v_0$ defined like this is a $C^\infty\left(\mathbb T^2\right)$ function and $V=V(\theta,t)$ defined by for all $\theta\in\mathbb T^2$, for all $t\geq 0$,
\begin{equation*}
V(\theta,t):=2\sum_{M\geq 1}M^{-l}|\xi_M|^{-Ml}e^{-2\pi |\mathrm{N}\cdot\xi_M|t}\cos(2\pi \xi_M\cdot\theta)
\end{equation*}
is a smooth solution to \eqref{sysVd=2}. 

For all $M\geq 1$, let $t_M:=l\frac{M|\xi_M|^M}{2\pi}$. The final step in the proof of theorem \ref{theoslowcvL2} is to estimate $\begin{Vmatrix}
\partial_\theta^\alpha\left(V(\theta,t_M)-\widehat{v_0}(0)\right)
\end{Vmatrix}_{L^2(\mathbb T^2)}$ for $\alpha\in\mathbb N^2$ and $M\geq 1$. Of course, in our example $\widehat{v_0}(0)=0$. Yet one is free to modify this coefficient without changing anything to the nature of the problem. Let $\alpha\in\mathbb N^2$. By Lebegue's dominated convergence theorem, for all $(\theta,t)\in\mathbb T^2\times\mathbb R_+$,
\begin{equation*}
\partial_\theta^\alpha V(\theta,t)=2\sum_{M\geq 1}M^{-l}|\xi_M|^{-Ml}(2\pi)^{|\alpha|}\xi_M^\alpha e^{-2\pi|\mathrm{N}\cdot\xi_M|t}\cos^{(|\alpha|)}(2\pi\xi_M\cdot\theta),
\end{equation*}
and Parceval's inequality yields
\begin{equation}\label{slowParceval}
\begin{Vmatrix}
\partial_\theta^\alpha V(\theta,t)
\end{Vmatrix}^2_{L^2(\mathbb T^2)}=2\sum_{M\geq 1}M^{-2l}|\xi_M|^{-2Ml}(2\pi)^{2|\alpha|}|\xi_M^\alpha|^2e^{-4\pi|\mathrm{N}\cdot\xi_M|t}.
\end{equation}
Due to our choice of sequence $(\xi_M)_M$, in particular because of the second property of $(\xi_M)$ listed above, there exists $M^{(0)}\geq 1$, such that for all $M\geq M^{(0)}$, $\xi_{M,1}\neq 0$ and $\xi_{M,2}\neq 0$, where $\xi_M=(\xi_{M,1},\xi_{M,2})$; for all $M\geq M^{(0)}$, $|\xi_M^\alpha|=|\xi_{M,1}^{\alpha_1}||\xi_{M,2}^{\alpha_2}|\geq 1$. Therefore, \eqref{slowParceval} yields for all $t\in\mathbb R_+$,
\begin{align*}
\begin{Vmatrix}
\partial_\theta^\alpha V(\theta,t)
\end{Vmatrix}^2_{L^2(\mathbb T^2)}&\geq 2\sum_{M\geq M^{(0)}}M^{-2l}|\xi_M|^{-2Ml}(2\pi)^{2|\alpha|}|\xi_M^\alpha|^2e^{-4\pi|\mathrm{N}\cdot\xi_M|t}\\
&\geq 2\sum_{M\geq M^{(0)}}M^{-2l}|\xi_M|^{-2Ml}e^{-4\pi|\mathrm{N}\cdot\xi_M|t}\\
&\geq \frac{2}{(4\pi)^{2l}}t^{-2l}\sum_{M\geq M^{(0)}}\left(\frac{4\pi t}{M|\xi_M|^M}\right)^{2l}e^{-\frac{4\pi}{M|\xi_M|^M}t},
\end{align*}
and for all $M\geq M^{(0)}$, $\begin{Vmatrix}
\partial_\theta^\alpha V(\theta,t_M)
\end{Vmatrix}_{L^2(\mathbb T^2)}\geq \sqrt{2}\bigl(\frac{e^{-1}l}{2\pi}\bigr)^lt^{-l}_M$, which proves theorem \ref{theoslowcvL2}.

Theorem \ref{theoslowcvL2} prevents $V(\theta,t)$ from decaying fast towards $v_{bl}^\infty$ when $t\rightarrow\infty$: in particular, \eqref{majVc^0_0div} is impossible, in general, when $n$ does not meet the small divisors assumption. However, theorem \ref{theoslow} cannot be deduced from theorem \ref{theoslowcvL2}. Let us turn to estimates in $L^\infty$ norm for $v$. Uniformity in the tangential variable $\theta$ is replaced by local uniformity in $z_1$. Let $R>0$. We slightly modify $\bigl(\widehat{v_0}(\xi)\bigr)_\xi$. As for all $|z_1|\leq R$,
\begin{equation*}
2\pi|\xi_M\cdot \mathrm{N}||z_1|<\frac{2\pi R}{M}|\xi_M|^{-M}\leq\frac{2\pi R}{M}<\frac{\pi}{4},
\end{equation*} 
for all $M\geq M^{(1)}$ sufficiently large, depending on $R$, we consider $\tilde{v}_0$ defined by 
\begin{equation*}
\widehat{\tilde{v}_0}(\xi):=\left\{
\begin{array}{cl}
0, & \mbox{if}\; \xi\neq \xi_M,\: -\xi_M \, \mbox{for all}\; M\geq 1\\
0, & \mbox{if}\; \xi=\xi_M\: \mbox{or}\: -\xi_M\; \mbox{for}\; 1\leq M< M^{(1)}\\
M^{-l}|\xi_M|^{-Ml}, & \mbox{if}\; \xi=\xi_M\: \mbox{or}\: -\xi_M\; \mbox{for}\; M\geq M^{(1)}
\end{array}
\right. .
\end{equation*}
Then, for all $|z_1|\leq R$, for all $t\in\mathbb R_+$,
\begin{align*}
v(z_1,t)&=2\sum_{M\geq M^{(1)}}M^{-l}|\xi_M|^{-Ml}e^{-2\pi|\mathrm{N}\cdot\xi_M|t}\cos(2\pi\xi_M\cdot \mathrm{N}z_1)\\
&\geq \frac{2}{(2\pi)^l}t^{-l}\sum_{M\geq M^{(1)}}\left(\frac{2\pi t}{M|\xi_M|^M}\right)^{l}e^{-\frac{2\pi}{M|\xi_M|^M}t}
\end{align*}
which demonstrates theorem \ref{theoslow} when evaluated in $t=t_M$ for $M\geq M^{(1)}$.

\section*{Acknowledgement}
The author would like to thank his PHD advisor David G\'erard-Varet for bringing this stimulating subject to him and for his useful remarks as well as advice. The research of this paper was supported by the Agence Nationale de la Recherche under the grant ANR-$08$-JCC-$0104$ coordinated by David G\'erard-Varet.


\bibliographystyle{alpha}

\bibliography{BLtail_bib.bib}


\end{document}